\documentclass[11pt]{amsart}
\usepackage{mathrsfs}
\usepackage{amssymb}

\usepackage{titletoc}
\usepackage{amscd}
\usepackage{amsmath, amssymb}
\usepackage{amsfonts}
\usepackage[colorlinks,linkcolor=blue,citecolor=blue, pdfstartview=FitH]{hyperref}

\usepackage{upgreek}

\usepackage{tikz-cd} 

\usepackage{bm}

\numberwithin{equation}{section}

\theoremstyle{plain}
\newtheorem{thm}{Theorem}[section]

\newtheorem{lem}[thm]{Lemma}
\newtheorem{corollary}[thm]{Corollary}
\newtheorem{prop}[thm]{Proposition}

\theoremstyle{definition}

\newtheorem{rmk}[thm]{Remark}

\newtheorem{definition}[thm]{Definition}

\newtheorem{exm}[thm]{Example}

\newtheorem{defn-thm}[thm]{Definition-Theorem}
\newtheorem{defn-pro}[thm]{Definition-Proposition}

\usepackage{fancyhdr}

\newcommand{\sA}{{\mathcal A}}

\newcommand{\sB}{{\mathcal B}}
\newcommand{\C}{{\mathbb C}}
\newcommand{\sC}{{\mathcal C}}

\newcommand{\sD}{{\mathcal D}}
\newcommand{\sfD}{\mathsf{D}}
\newcommand{\E}{{\mathrm E}}

\newcommand{\sF}{{\mathcal F}}

\newcommand{\sG}{{\mathcal G}}

\newcommand{\J}{{\mathbb J}}

\renewcommand{\k}{\Bbbk}

\newcommand{\bfL}{{\mathbf L}}
\newcommand{\sL}{{\mathcal L}}
\newcommand{\m}{{\mathfrak m}}

\newcommand{\sM}{{\mathcal M}}

\newcommand{\sN}{{\mathcal N}}
\newcommand{\sfN}{\mathsf{N}}

\newcommand{\sO}{{\mathcal O}}
\newcommand{\bfP}{{\mathbf P}}
\newcommand{\sP}{{\mathcal P}}

\newcommand{\sR}{{\mathcal R}}
\newcommand{\sfR}{\mathsf{R}}

\newcommand{\sS}{{\mathcal S}}

\newcommand{\sT}{{\mathcal T}}
\newcommand{\sfT}{{\mathsf T}}

\newcommand{\sU}{{\mathcal U}}

\newcommand{\sV}{{\mathcal V}}

\newcommand{\X}{{\mathbb X}}

\newcommand{\Z}{{\mathbb Z}}
\newcommand{\bfZ}{{\mathbf Z}}
\newcommand{\sZ}{{\mathcal Z}} 

\newcommand{\id}{\mathrm{id}}

\newcommand{\pr}{\mathrm{pr}}

\newcommand{\xs}{\xrightarrow{\sim}}

\newcommand{\Hom}{\mathrm{Hom}}
\newcommand{\End}{\mathrm{End}}

\newcommand{\Ext}{\mathrm{Ext}}
\newcommand{\gr}{{\mathrm{gr}}}
\newcommand{\dg}{{\mathrm{dg}}}

\newcommand{\op}{{\mathrm{op}}}

\newcommand{\Mod}{\mathrm{Mod}}
\renewcommand{\mod}{\mathrm{mod}}

\newcommand{\Rep}{\mathrm{Rep}}

\newcommand{\Coh}{\mathrm{Coh}}
\newcommand{\DGQCoh}{\mathrm{DGQCoh}}
\newcommand{\DGCoh}{\mathrm{DGCoh}}
\newcommand{\DGMod}{\mathrm{DGMod}}
\newcommand{\QCoh}{\mathrm{QCoh}}
\newcommand{\IndCoh}{\mathrm{QCoh}_{\mathrm{colim.coh.}}}
\newcommand{\IndMod}{\mathrm{Mod}_{\mathrm{colim.coh.}}}
\newcommand{\IndMof}{\mathrm{Mod}_{\mathrm{colim.fg.}}}
 
\newcommand{\Db}{D^{\mathrm{b}}}

\newcommand{\Spec}{\mathrm{Spec}}

\newcommand{\act}{\mathrm{act}}

\newcommand{\Gm}{\mathbb{G}_{\mathrm{m}}}
\newcommand{\Ga}{\mathbb{G}_{\mathrm{a}}}

\newcommand{\g}{{\mathfrak{g}}}
\renewcommand{\b}{{\mathfrak{b}}}
\renewcommand{\t}{{\mathfrak{t}}}
\newcommand{\n}{{\mathfrak{n}}}
\newcommand{\h}{{\mathfrak{h}}}

\newcommand{\p}{{\mathfrak{p}}}

\newcommand{\sEnd}{\mathcal{E}nd}
\newcommand{\sHom}{\mathcal{H}om}

\renewcommand{\for}{\mathrm{for}}
\newcommand{\res}{\mathrm{res}}

\newcommand{\Ind}{\mathrm{Ind}}

\newcommand{\ex}{\mathrm{ex}}

\newcommand{\hb}{\mathrm{hb}}
\newcommand{\Dist}{\mathrm{Dist}}

\newcommand{\Br}{\mathfrak{Br}}

\newcommand{\Spr}{\widetilde{\mathcal{N}}}
\newcommand{\Groth}{\widetilde{\mathfrak{g}}}

\newcommand{\Simp}{\mathsf{L}}
\newcommand{\bV}{\mathsf{Z}}
\newcommand{\ZFr}{\mathrm{Z}_{\mathrm{Fr}}}
\newcommand{\ZHC}{\mathrm{Z}_{\mathrm{HC}}}
\newcommand{\Pro}{\mathsf{P}}
\newcommand{\HC}{\mathsf{HC}}

\newcommand{\aff}{\mathrm{aff}}

\newcommand{\simto}{\xrightarrow{\sim}}
\newcommand{\hchi}{\widehat{\chi}}
\newcommand{\hla}{\widehat{\lambda}}
\newcommand{\hmu}{\widehat{\mu}}
\newcommand{\hnu}{\widehat{\nu}}
\newcommand{\ho}{\widehat{0}}
\newcommand{\Phis}{\Phi^{\mathrm{s}}}
\newcommand{\tDelta}{\widetilde{\Delta}}
\newcommand{\hDelta}{\widehat{\Delta}}
\newcommand{\tD}{\widetilde{\sD}}
\newcommand{\tpi}{\widetilde{\pi}}
\newcommand{\hpi}{\widehat{\pi}}
\newcommand{\hM}{\widehat{\mathcal{M}}}
\newcommand{\tM}{\widetilde{\mathcal{M}}}
\newcommand{\rH}{\mathrm{H}}

\begin{document}

\title{Equivariant Koszul duality, modular category $\mathcal{O}$, and periodic Kazhdan--Lusztig polynomials}

\makeatletter
\let\MakeUppercase\relax
\makeatother

\author{Simon Riche}
\address{Universit\'e Clermont Auvergne, CNRS, LMBP, F-63000 Clermont-Ferrand, France.}
\email{simon.riche@uca.fr}

\author{Quan Situ}
\address{Universit\'e Clermont Auvergne, CNRS, LMBP, F-63000 Clermont-Ferrand, France.}
\email{quan.situ@uca.fr}


\date{}

\begin{abstract}
Let $G$ be a connected reductive algebraic group over an algebraically closed field of positive characteristic, $\g$ be its Lie algebra, and $B$ be a Borel subgroup.
We prove a formula for the dimensions of extension groups, in the principal block of the category of strongly $B$-equivariant $\g$-modules (also called \emph{modular category $\mathcal{O}$}), from a simple object to a costandard object, under the assumption that Lusztig's conjecture holds (which is known in large characteristic). The answer is given by a coefficient of a periodic Kazhdan--Lusztig polynomial associated with the corresponding affine Weyl group. Among other things, the proof uses a torus-equivariant version of the Koszul duality for $\g$-modules constructed by the first author.
\end{abstract}

\maketitle

{\setcounter{tocdepth}{1} \tableofcontents} 

\section{Introduction} 

\subsection{Kazhdan--Lusztig combinatorics and representations of reductive grou\-ps and their Lie algebras} 
\label{ss:intro-KL-combinatorics}

It has been a fruitful philosophy, first suggested by Verma~\cite{verma}, that the combinatorial data one can extract from representations of a connected reductive algebraic group $G$ over a field of positive characteristic $\k$ should be closely related to data attached to the associated affine Weyl group $W_\aff$, endowed with its natural structure of Coxeter group. This idea was made more precise by Lusztig~\cite{lusztig-pbs}, after his introduction (with Kazhdan) of the Kazhdan--Lusztig polynomials attached to any Coxeter group, in the form a conjecture expressing (when $\mathrm{char}(\k)$ is not too small) characters of simple modules in the principal block in terms of the (known) characters of induced modules using values at one of the spherical Kazhdan--Lusztig polynomials of $W_\aff$. (Here, ``spherical'' Kazhdan--Lusztig polynomials means parabolic Kazhdan--Lusztig polynomials attached to $W_\aff$ and the parabolic subgroup given by the Weyl group $W$ of $G$, or in other words Kazhdan--Lusztig polynomials attached to elements $w \in W_\aff$ which are maximal in $W w$.) 
This conjecture is now known to be true in large characteristics. It is also known (due to work of Andersen~\cite{andersen-inversion}) that, when this conjecture is true, the coefficients of these polynomials compute the dimensions of extension groups from a simple $G$-module to an induced $G$-module. For a review of these results, see~\cite[\S\S II.C.1--10]{Jan03}.

Consider now the subgroup $G_1 T$ of $G$, where $G_1$ is the Frobenius kernel of $G$ and $T$ is a maximal torus. It is known that understanding the simple representations of $G_1 T$ is equivalent to understanding the simple $G$-modules. In this setting Lusztig also formulated a conjecture describing (under the same assumption on $\mathrm{char}(\k)$) the characters of the simple modules in the principal block in~\cite{lusztig-generic}. More specifically, in this case the ``basic'' objects whose characters are known are the baby Verma modules, and the conjecture expresses multiplicities of simple modules in baby Verma modules in terms of the value at $1$ of a \emph{periodic}\footnote{The periodic Kazhdan--Lusztig polynomials were introduced by Lusztig in~\cite{lusztig-generic}. Here we will follow the conventions of~\cite{Soe97}. The terminology in the literature is not always consistent; in particular, sometimes the word ``periodic'' is used for the polynomials that are called \emph{generic} in~\cite{Soe97} (and in the present introduction).} Kazhdan--Lusztig polynomial for $W_\aff$.
It is known that this conjecture is equivalent to the original Lusztig conjecture, see~\cite{fiebig}; in particular it is true in large characteristics. It is known also (due to work of Cline--Parshall--Scott~\cite{cps}) that, when this conjecture is true, the coefficients of the \emph{generic} Kazhdan--Lusztig polynomials attached to $W_\aff$ compute the dimensions of extension groups from a simple $G_1T$-module to a dual baby Verma module. For a review of these results, see~\cite[\S\S II.C.11--17]{Jan03}.

There is a third closely related category one can consider, that has appeared more recently in the literature in this setting, namely the category $\Mod(\g,B)$ of strongly $B$-equivariant $\g$-modules, where $B \subset G$ is a Borel subgroup containing $T$ and $\g$ is the Lie algebra of $G$. This category has in particular been considered recently in work of Losev~\cite{losev}, where its full subcategory of modules finitely generated over $\sU(\g)$ is dubbed ``modular category $\mathcal{O}$.'' (The reason for this name is that, replacing in the definition the base field $\k$ by the complex numbers, one obtains the sum of the integral blocks of the Bernstein--Gelfand--Gelfand category $\mathcal{O}$ of the associated complex reductive Lie algebra. This category is unrelated to the one which is given the same name in~\cite{soergel}.) In this case the ``basic'' objects whose characters are known are the (true) Verma modules, which are infinite-dimensional but have finite-dimensional $T$-weight spaces. Comparing the characters of Verma modules with their baby versions, and using the formula in~\cite[Theorem~6.3]{Soe97}, one easily sees that when the conjecture from~\cite{lusztig-generic} mentioned above is true, the multiplicities of simple objects in Verma modules can be expressed in terms of the value at $1$ of an appropriate generic Kazhdan--Lusztig polynomial. (See~\cite{situ} for similar considerations in a different context.) The main result of the present paper is an analogue of the consequence of this formula described above for $G$- and $G_1T$-modules: we show that when Lusztig's conjecture is true, coefficients of periodic Kazhdan--Lusztig polynomials compute the dimensions of extension groups in $\Mod(\g,B)$ from a simple object to a ``costandard'' object.

\subsection{Statement} 

Let us now state our main result more explicitly. Let $\k$ be an algebraically closed field of characteristic $p>0$, and
let $G$ be a connected reductive algebraic group over $\k$. Let also $h$ be the Coxeter number of $G$.
We fix a Borel subgroup $B \subset G$ and a maximal torus $T \subset B$, and denote by $\g$ and $\b$ the Lie algebras of $G$ and $B$ respectively.

We denote by $\Mod(\g,B)$ the category of $(\g,B)$-modules, i.e.~$B$-equivariant $\sU(\g)$-modules on which the differential of the $B$-action coincides with the restriction of the $\sU(\g)$-action to $\sU(\b)$. (Note that the action of $B$ is not uniquely determined by that of $\sU(\b)$, so that $\Mod(\g,B)$ is \emph{not} a full subcategory of the category of $\sU(\g)$-modules.) In this category we have the Verma modules $(\Delta(\lambda) : \lambda \in \X)$, labelled by 
the lattice $\X$ of characters of $T$. (They are defined by exactly the same recipe as in the setting of representations of complex reductive Lie algebras.) It is not difficult to see that each $\Delta(\lambda)$ admits a unique simple quotient $\Simp(\lambda)$, and that these objects form representatives for the isomorphism classes of simple objects in $\Mod(\g,B)$. (These modules are finite-dimensional; in fact they are closely related with the simple $G_1T$-modules.) To each $\lambda \in \X$ one can also attach a ``costandard'' object $\nabla(\lambda)$ which contains $\Simp(\lambda)$ as a subobject. (Contrary to the Verma modules or the simple objects, the costandard objects are never finitely generated.)

Let $W_\aff$ be the affine Weyl group of $G$, i.e.~the semidirect product of the Weyl group $W$ of $(G,T)$ with the sublattice in $\X$ generated by the roots of $G$. This group admits a natural structure of Coxeter group, such that $W$ is a parabolic subgroup. Attached with this group are the periodic Kazhdan--Lusztig polynomials $(p_{y,w} : y,w \in W_\aff)$, see~\cite[\S 4]{Soe97}. (Compared with~\cite{Soe97}, we label these polynomials by elements of $W_\aff$ rather than alcoves; so our polynomial $p_{y,w}$ is the polynomial $p_{y(A^+), w(A^+)}$ in the notation of~\cite{Soe97}.) We also consider the standard ``dot-action'' of $W_\aff$ on $\X$, denoted $(w,\lambda) \mapsto w \bullet \lambda$.


The following statement is our main result. Here, we say that \emph{Lusztig's conjecture holds} if the equivalent conditions in~\cite[Proposition~II.C.17(a)]{Jan03} are true. (It is known that this condition is satisfied if $p$ is larger than a certain bound depending on the root system of $G$, see~\cite{fiebig-upper}.)

\begin{thm}
\label{thm:main}
Assume that $p>h$, and that Lusztig's conjecture holds. Then for any $y,w \in W_\aff$ we have
\[
\sum_m \dim_\k \Ext^m_{\Mod(\g,B)}(\Simp(w \bullet 0), \nabla(y \bullet 0))\cdot v^{m}=p_{y,w}(v).
\]
\end{thm}

\begin{rmk}
For simplicity we state the formula in Theorem~\ref{thm:main} for the block of $0$. Using translation functors it immediately implies a similar formula in all regular blocks of the category $\Mod(\g,B)$. It would also be interesting to generalize this formula to singular blocks, but we do not know how to do this at this stage.
\end{rmk}

\subsection{Motivation} 
\label{ss:intro-motivation}

Our motivation for studying the formula in Theorem~\ref{thm:main} came from two different directions. The first one is related to our comments in~\S\ref{ss:intro-KL-combinatorics}: the formula in this theorem is an analogue for $(\g,B)$-modules of familiar formulas in the theories of $G$-modules and $G_1T$-modules. 

The second one is concerned with work of the first author with Achar and Dhil\-lon~\cite{adr}. Here a category of ``semiinfinite sheaves'' on the affine flag variety of a connected reductive algebraic group is introduced, and it is conjectured that the stalks of simple objects in this category are computed by periodic Kazhdan--Lusztig polynomials, in case the characteristic of the coefficient field is large (or $0$). In work in preparation by the same authors, an equivalence between this category of semiinfinite sheaves, for the reductive group Langlands dual to $G$ and the field of coefficients $\k$, and the principal block of the category $\Mod(\g,B)$ is constructed; see~\cite[\S 1.4]{adr} for details. The formula conjectured in~\cite{adr} can be transferred to a formula in $(\g,B)$-modules using this equivalence, and Theorem~\ref{thm:main} is precisely what one obtains. From this point of view, this theorem therefore provides a very indirect way to prove the said conjecture from~\cite{adr}.

\subsection{Description of the proof} 
\label{ss:intro-proof}

In the two other settings described in~\S\ref{ss:intro-KL-combinatorics}, the deduction of the analogue of Theorem~\ref{thm:main} from the character formula for simple modules is rather elementary. The proof proceeds by induction on $w$ (for the Bruhat order), proving in parallel another property, namely that (in the Jantzen region in the case of $G$-modules, or always in the case of $G_1T$-modules) when applying an appropriate wall-crossing functor to a simple module and killing the head and socle, the resulting object is semisimple, see~\cite[Proposition~II.C.2(a-ii) and Proposition~II.C.17(a-ii)]{Jan03}.
(This property is an analogue of a statement in the category $\mathcal{O}$ of a complex reductive Lie algebra, deduced by Vogan from the Kazhdan--Lusztig conjecture.) The similar strategy cannot apply in the case of $(\g,B)$-modules, because the analogue of the latter property is \emph{not} true. (This can easily be seen looking at Steinberg's tensor product formula; this is related to the fact that the restriction to $B^{(1)}$ of a simple $G^{(1)}$-module is not semisimple in general.) We therefore follow a different, and much less elementary, route.

The first step \emph{is} elementary (and does not require any restriction on $p$): we use easy homological algebra considerations (in particular, adjunctions) to identify the $\Ext$-groups under consideration with some extension groups, in the category of $T$-equivariant $\sU(\g)$-modules, from a universal Verma module to a simple module; see Proposition~\ref{prop:Ext-gB-gT}. The next steps require the use of appropriate versions of the ``localization'' equivalences constructed by Bezrukavnikov--Mirkovi{\'c}--Rumynin~\cite{BMR08,BMR06,BM13} and the ``Koszul duality'' constructed by the first author in~\cite{Ric10}. Namely, this Koszul duality, whose construction is based on localization theory, allows to relate a certain derived category of modules over the specialization $(\sU(\g))^0$ of $\sU(\g)$ at the Harish-Chandra central character $0$ with the (extended) principal block of representations of the Frobenius kernel $G_1$ of $G$, and is known to exchange simple $(\sU(\g))^0$-modules with projective $G_1$-modules. Here we construct a $T$-equivariant version of this duality, and prove that it sends appropriate completed versions of universal Verma modules to baby Verma modules. This allows to express the $\Ext$ groups in Theorem~\ref{thm:main} as a graded multiplicity space of a simple $G_1T$-module in a baby Verma module. The \emph{ungraded} multiplicities here are expressed by Lusztig's character formula, and given by values at $1$ of periodic Kazhdan--Lusztig polynomials. Because the involved grading is a Koszul grading by the results of~\cite{Ric10}, the graded multiplicity corresponds to the multiplicity in a certain layer of the Loewy filtration of the given baby Verma module, which is known (thanks to work of Andersen--Kaneda) to be given by the appropriate coefficient of a periodic Kazhdan--Lusztig polynomial, which allows to conclude.


\subsection{Contents} 

In
Section~\ref{sec:gB-modules} we define the category $\Mod(\g,B)$ of $(\g,B)$-modules, and explain its basic properties. In particular we introduce the families of standard and costandard objects in $\Mod(\g,B)$, prove that they satisfy $\Ext$-orthogonality (see Lemma~\ref{lem:Ext-Delta-nabla}) and explain the classification of simple objects. (These statements are certainly known to experts, but not proved in detail in the literature as far as we know.) Finally, we prove in Proposition~\ref{prop:Ext-gB-gT} a formula that describes extension spaces in $\Mod(\g,B)$ from a simple module to a costandard module in terms of extensions as $T$-equivariant $\sU(\g)$-modules from a universal Verma module to a simple module. The proofs in this section are elementary in the sense that they only rely on standard results in the representation theory of reductive algebraic groups, and do not require any restriction on the characteristic.
In Section~\ref{sec:central-completions} we introduce some central completions of the algebra $\sU(\g)$, the associated categories of modules, and the ``translation'' functors relating them. We also define completed versions of universal Verma modules and study their behavior under translation functors. 

In Section~\ref{sec:localization} we prove ``localization theorems'' relating some derived categories of (equivariant) $\sU(\g)$-modules to some categories of (equivariant) coherent sheaves on some schemes constructed out of the Springer and Grothendieck resolutions for $G$. These statements are variants of the theory developed by Bezrukavnikov--Mirkovi{\'c}--Rumynin in~\cite{BMR08, BMR06, BM13}. 
In Section~\ref{sec:geometric-wall-crossing} we introduce the geometric counterparts of translation functors, a braid group action constructed out of these functors, and explain how to describe the complexes of coherent sheaves corresponding to completed universal Verma modules and baby Verma modules in the principal block.

In Section~\ref{sec:Koszul-duality} we explain how to modify the constructions of~\cite{Ric10} to take into account the natural actions of the maximal torus $T$. We take this opportunity to revisit (and, sometimes, simplify or correct some proofs), which might be of independent interest.
%
Finally, in Section~\ref{sec:dim} we give the proof of our main result, following the strategy outlined in~\S\ref{ss:intro-proof}. Our arguments involve the study of $(p\X \times \Z)$-graded versions of modules for a finite-dimensional algebra governing the principal block of $G_1$-modules corresponding to projective, baby Verma and simple modules. Here the grading comes on the one hand from the action of $T$, and on the other hand from the Koszul grading of~\cite{Ric10}.

Our proofs require the development of some elements of a theory of equivariant coherent sheaves on some ``completions'' of schemes with an action of a flat affine group scheme, which are explained
in Appendix~\ref{app:equiv-sheaves}.
 What we prove here is sufficient for the applications in the body of the paper, but this is certainly not a complete and satisfactory theory. In particular, some of our statements require strong assumptions (as e.g.~in~\S\ref{ss:derived-pullback} or~\S\ref{ss:derived-pushforward}). These assumptions might not be necessary, but they might also be an indication that our definitions are possibly not the most adequate ones in a more general setting.

\subsection{Acknowledgements} 

This project has received
funding from the European Research Council (ERC) under the European Union's Horizon 2020
research and innovation programme (grant agreement No.~101002592).

Part of the writing process was completed while the second author visited the Max Planck Institute for Mathematics in Bonn.
He is grateful for its hospitality and financial support.

As explained in~\S\ref{ss:intro-motivation}, the question studied in this paper was motivated in part by joint work of the first author with Pramod Achar and Gurbir Dhillon. We thank them for inspiring discussions around this subject.

\subsection{Notation and conventions} 

\subsubsection{Grading shift} 
\label{sss:grading-shift}

Recall that if $\Lambda$ is an abelian group and $k$ a commutative ring we have an associated diagonalizable group scheme $\mathrm{Diag}_k(\Lambda)= \Spec(k[\Lambda])$ over $k$, see~\cite[\S I.2.5]{Jan03}; such group schemes are called diagonalizable. Their representations are particularly simple: the datum of a $\mathrm{Diag}_k(\Lambda)$-action on a $k$-module $M$ is equivalent to the datum of a $\Lambda$-grading on $M$.

Assume now that $k$ is a field, so that the monoidal category $\Rep^{\mathrm{fd}}(\mathrm{Diag}_k(\Lambda))$ of finite-dimensional representations of $\mathrm{Diag}_k(\Lambda)$ identifies with the monoidal category of finite-dimensional $\Lambda$-graded $k$-vector spaces. Given a category $\mathsf{A}$ endowed with an action of $\Rep^{\mathrm{fd}}(\mathrm{Diag}_k(\Lambda))$, for any $\lambda \in \Lambda$ we will denote by
\[
\langle \lambda \rangle : \mathsf{A} \simto \mathsf{A}
\]
the autoequivalence given by tensor product with the $1$-dimensional $\mathrm{Diag}_k(\Lambda)$-module determined by $\lambda$. This applies in particular to categories of $\mathrm{Diag}_k(\Lambda)$-equivariant (quasi-)coherent sheaves on noetherian $k$-schemes equipped with actions of $\mathrm{Diag}_k(\Lambda)$, and to categories of (finitely generated) $\Lambda$-graded modules over noetherian $\Lambda$-graded algebras. In this setting, $\langle \lambda \rangle$ corresponds to the ``shift of grading'' functor determined by the rule $(M \langle \lambda \rangle)_\mu = M_{\mu-\lambda}$ for $\mu \in \Lambda$.

\subsubsection{Rings and modules} 
\label{sss:rings-modules}

If $R$ is a commutative ring and $I \subset R$ an ideal, we will denote by $R_{\widehat{I}}$ the completion of $R$ with respect to $I$.

For a ring $A$, we denote by $\Mod(A)$ the category of left $A$-modules and, in case $A$ is left noetherian, we denote by $\mod(A)$ the full subcategory of finitely generated left $A$-modules. 
For an affine group scheme $K$ over a commutative ring $k$, we denote by $\Rep(K)$ the category of rational representations of $K$, i.e.~of (right) $\sO(K)$-comodules. 
If $A$ is a $k$-algebra equipped with a compatible action of $K$, we denote by $\Mod^K(A)$ the category of $K$-equivariant\footnote{By default, $K$-equivariant means \emph{weakly} $K$-equivariant: we require to have actions of $A$ and $K$ such that the action map $A \otimes_k M \to M$ is $K$-equivariant.} left $A$-modules; in case $A$ is noetherian we denote by $\mod^K(A)$ the full subcategory whose objects are the equivariant modules which are finitely generated over $A$.

\subsubsection{Algebraic geometry} 

If $X$ is a scheme we will denote by $\sO_X$ its structure sheaf and set $\sO(X):=\Gamma(X, \sO_X)$. We denote by $\QCoh(X)$ the category of quasi-coherent $\sO_X$-modules and, if $X$ is noetherian, by $\Coh(X)$ the full subcategory of coherent sheaves. If $\sA$ is a quasi-coherent sheaf of $\sO_X$-algebras, 
we denote by $\Mod(\sA)$ the category of sheaves of $\sA$-modules which are quasi-coherent as $\sO_X$-modules; in case $X$ is noetherian we will denote by $\mod(\sA)$ the full subcategory whose objects are the sheaves of $\sA$-modules which are coherent as $\sO_X$-modules. 

If $X$ is a scheme over a commutative ring $k$ equipped with an action of a flat affine group scheme $H$ over $k$, one can consider the category $\QCoh^H(X)$ of $H$-equivariant quasi-coherent sheaves on $X$ and, in case $X$ is noetherian, the full subcategory $\Coh^H(X)$ of $H$-equivariant coherent sheaves. For a reminder on basic aspects of this formalism, see~\cite[Appendix]{MR16}.
If $\sA$ is an $H$-equivariant quasi-coherent sheaf equipped with a compatible structure of $\sO_X$-algebra, we will denote by $\Mod^H(\sA)$ the category of $H$-equivariant quasi-coherent sheaves on $X$ equipped with a compatible structure of $\sA$-module. In case $X$ is noetherian, we will denote by $\mod^H(\sA)$ the full subcategory of objects which are coherent as $\sO_X$-modules.

In practice, our schemes will be defined over a fixed algebraically closed field $\k$. Given a $\k$-scheme $X$, the \emph{Frobenius twist} $X^{(1)}$ of $X$ is defined by
\[
X^{(1)} = \Spec(\k) \times_{\Spec(\k)} X
\]
where the morphism $\Spec(\k) \to \Spec(\k)$ is associated with the automorphism of $\k$ given by $a \mapsto a^p$, seen as a $\k$-scheme via projection on the first factor. The projection morphism $X^{(1)} \to X$ is an isomorphism of abstract schemes, but not of $\k$-schemes; in fact via this identification the structure morphism to $\Spec(\k)$ corresponds to the morphism $\k \to \sO(X)$ given by $a \mapsto a^{1/p}$. We will denote by $\mathrm{Fr}_X : X \to X^{(1)}$ the Frobenius morphism of $X$. In case $X$ is a group scheme over $\k$, $\mathrm{Fr}_X$ is a morphism of group schemes, whose kernel (called the Frobenius kernel of $X$) will be denoted $X_1$.

\subsubsection{Reductive groups} 
\label{sss:reductive-gps-notation}

Let $G_\Z$ be a split reductive group scheme over $\Z$. We fix a Borel subgroup $B_\Z$ and a split maximal torus $T_\Z$ of $G_\Z$ contained in $B_\Z$, and denote by $U_\Z$ the unipotent radical of $B_\Z$. Let also $B^-_\Z$ be the opposite Borel subgroup, and let $U^-_\Z$ be its unipotent radical. We will denote the Lie algebras of these group schemes by $\g_\Z$, $\b_\Z$, $\t_\Z$, $\n_\Z$, $\b^-_\Z$, $\n^-_\Z$ respectively. For any commutative ring $R$, from these data we obtain group schemes and Lie algebras over $R$ by base change, which will be denoted similarly, replacing the subscript $\Z$ by $R$. 
In particular, we fix an algebraically closed field $\k$, of characteristic $p>0$, and consider the corresponding objects over $\k$; for this special case we will omit the subscript ``$\k$.'' So, we have $G=G_\Z \times_{\Spec (\Z)} \Spec (\k)$, $B=B_\Z \times_{\Spec (\Z)} \Spec (\k)$, $\g=\g_\Z \otimes_\Z \k$, etc. 


Let $W$ be the Weyl group of $(G_\Z,T_\Z)$ (equivalently, of $(G,T)$), endowed with the Coxeter group structure determined by the choice of $B_\Z$, and let $w_\circ\in W$ be the longest element.
Let $\X=X^*(T_\Z)=X^*(T)$ be the character lattice, and let $W_\ex=W\ltimes \X$ be the extended affine Weyl group. 
Let $\Phi \subset \X$ be the root system of $G_\Z$ with respect to $T_\Z$, let $\Phi^+\subset \Phi$ be the subset of positive roots consisting of the weights in the Lie algebra of $B_\Z$, and let $\Phis \subset \Phi^+$ be the associated system of simple roots. For $\alpha \in \Phi$, we will denote as usual by $\alpha^\vee \in X_*(T_\Z)$ the associated coroot. This choice determines a subset $\X^+ \subset \X$ of dominant weights. We will denote by $\preceq$ the ``dominance'' order on $\X$ determined by $\Phi^+$, i.e.~the order such that $\lambda \preceq \mu$ if and only if $\mu-\lambda$ is a sum of positive roots.

For $\lambda\in \X$ we will denote by $t_\lambda := (1 \ltimes \lambda) \in W_\ex$ the translation element corresponding to $\lambda$. Let $\ell: W_\ex\rightarrow \Z_{\geq 0}$ be the ``length'' function defined by
\[
\ell(w \cdot t_\lambda) = \sum_{\substack{\alpha \in \Phi^+ \\ w(\alpha) \in \Phi^+}} |\langle \lambda, \alpha^\vee \rangle | + \sum_{\substack{\alpha \in \Phi^+ \\ w(\alpha) \in -\Phi^+}} |1+ \langle \lambda, \alpha^\vee \rangle |
\] 
for $w \in W$ and $\lambda \in \X$. Then the subgroup $W_\aff := W\ltimes \Z\Phi \subset W_\ex$ admits a canonical structure of Coxeter group, with set of Coxeter generators denoted $S_\aff$, such that the associated length function is the restriction of $\ell$. Moreover, $W$ is the parabolic subgroup in $W_\aff$ generated by the subset $S := S_\aff \cap W$ of $S_\aff$ and, setting $\Omega := \{w \in W_\ex \mid \ell(w)=0\}$, $\Omega$ is an abelian group whose action by conjugation preserves $S_\aff$, and multiplication induces a group isomorphism
\[
\Omega \ltimes W_\aff \simto W_\ex.
\]

We will denote by $\Br_\ex$ the braid group associated with $W_\ex$, i.e.~the group generated by the symbols $(\rH_w : w \in W_\ex)$, with relations $\rH_x \cdot \rH_y = \rH_{xy}$ for all $x,y \in W_\ex$ such that $\ell(xy)=\ell(x)+\ell(y)$. It is clear that we have a surjection $\Br_\ex \twoheadrightarrow W_\ex$ sending $\rH_w$ to $w$ for any $w \in W_\ex$.
As explained in~\cite[\S 1.1]{Ric08}, to each $\lambda \in \X$ one can attach a certain canonical element $\theta_\lambda \in \Br_\ex$, in such a way that $\theta_\lambda \cdot \theta_\mu = \theta_{\lambda+\mu}$ for any $\lambda,\mu \in \X$ and the collection consisting of $(H_s : s \in S)$ and $(\theta_\lambda : \lambda \in \X)$ generates $\Br_\ex$. More specifically, for $\lambda \in \X$, if $\lambda=\mu-\nu$ with $\mu,\nu \in \X^+$ we have $\theta_\lambda = \rH_{t_\mu} \cdot (\rH_{t_\nu})^{-1}$.

We will consider the action of $W_\ex$ on $\X$, denoted $\bullet$, given for $w\in W$ and $\lambda,\mu\in \X$ by 
\[
(wt_\lambda) \bullet \mu := w(\mu+p\lambda +\rho) - \rho,
\]
where $\rho$ is the halfsum of the positive roots. (Here, even through $\rho$ does not necessarily belong to $\X$, we have $w(\rho)-\rho \in \X$ for any $w \in W$, so that this formula indeed defines an action on $\X$.)

\subsubsection{Geometry associated with \texorpdfstring{$G$}{G}} 
\label{sss:geometry-G}

Let $\sB:=G/B^-$ be the flag variety of $G$. Consider the Springer and Grothendieck varieties
\[
\Spr=T^*\sB=G\times^{B^-} (\g/\b^-)^*, \qquad \Groth=G\times^{B^-} (\g/\n^-)^*
\]
respectively. Both schemes are naturally vector bundles over $\sB$, endowed with canonical actions of $G$, and $\Spr$ identifies with a ($G$-stable) sub-bundle in $\Groth$.
There is also a canonical map $\pi: \Groth \rightarrow \g^*$ induced by the co-adjoint action of $G$ on $\g^*$.
For $\lambda\in \X$, we will denote by $\sO_{\sB}(\lambda)$ the associated line bundle $G\times^{B^-} \k_{B^-}(\lambda) \rightarrow \sB$. 
For any morphism $X\rightarrow \sB$, we will write $\sO_X(\lambda)$ for the pullback of 
$\sO_{\sB}(\lambda)$ to $X$. We will use similar notation for morphisms $X\rightarrow \sB \times \sB$.

Let $\Gm$ be the multiplicative group over $\k$. 
We will consider the $\Gm$-action on $\Spr$, resp.~on $\Groth$, induced by the $\Gm$-action on $\g^{*}$ where $t\in \Gm$ acts by dilation with factor $t^2$, resp.~$t^{-2}$. (Note in particular that the action on $\Spr$ is \emph{not} the restriction of the action on $\Groth$.)

\section{\texorpdfstring{$(\g,B)$}{(g,B)}-modules} 
\label{sec:gB-modules}

In this section we introduce the category of $(\g,B)$-modules, and some important classes of objects therein.


\subsection{Preliminaries} 

\subsubsection{Distribution algebras}
\label{sss:Dist}

Recall that if $H$ is a group scheme over a commutative ground ring $R$, we have an $R$-algebra $\Dist(H)$ of distributions on $H$, see~\cite[\S I.7.7]{Jan03}. By definition this algebra admits a canonical $2$-sided ideal $\Dist^+(H) \subset \Dist(H)$.
Any $H$-module admits a canonical structure of $\Dist(H)$-module, see~\cite[\S I.7.11]{Jan03}, and this procedure defines a faithful and exact functor
\begin{equation}
\label{eqn:functor-Rep-Dist}
\Rep(H) \to \Mod(\Dist(H)).
\end{equation}
In particular, any morphism of group schemes $\lambda : H \to \mathbb{G}_{\mathrm{m},R}$ defines an algebra morphism $\Dist(H) \to R$. (In case $\lambda$ is the trivial character, this morphism factors through an isomorphism $\Dist(H)/\Dist^+(H) \cong R$.) Recall also that if $\mathfrak{h}$ is the Lie algebra of $H$, there exists a canonical morphism of $R$-algebras
\begin{equation}
\label{eqn:morph-U-Dist}
\mathcal{U} (\mathfrak{h}) \to \Dist(H)
\end{equation}
(where the left-hand side is the enveloping algebra of $\mathfrak{h}$),
which is an isomorphism in case $R$ is a field of characteristic $0$ and $H$ is of finite type over $R$; see~\cite[\S I.7.10]{Jan03}.

Recall from~\cite[\S I.7.4 and~\S I.7.9]{Jan03} that $H$ is called \emph{infinitesimally flat} if, denoting by $I \subset \sO(H)$ the ideal of the unit section, each quotient $\sO(H)/I^n$ is a finitely generated projective $R$-module. This condition is automatically satisfied if $R$ is a field and $H$ is of finite type. It is also satisfied in case $H$ is smooth over $R$: in fact in this case the immersion $\Spec(R) \hookrightarrow H$ corresponding to the unit section is regular by~\cite[\href{https://stacks.math.columbia.edu/tag/0E9J}{Tag 0E9J}]{stacks-project}, which by~\cite[\href{https://stacks.math.columbia.edu/tag/063E}{Tag 063E} \&~\href{https://stacks.math.columbia.edu/tag/063H}{Tag 063H}]{stacks-project} implies that $H$ is infinitesimally flat. In case $H$ is infinitesimally flat, $\Dist(H)$ is flat over $R$, and for any morphism of commutative rings $R \to R'$ we have
\[
\Dist(H \times_{\Spec(R)} \Spec(R')) = R' \otimes_R \Dist(H),
\]
see~\cite[Eqn.~(1) in~\S I.7.4]{Jan03}. Recall also that if $H$ is infinitesimally flat, noetherian and integral, the functor~\eqref{eqn:functor-Rep-Dist} is fully faithful on modules which are projective over $k$, see~\cite[Lemma~I.7.16]{Jan03}.

\subsubsection{Integrability}
\label{sss:integrability}

We will use the constructions of~\S\ref{sss:Dist} in particular for the $\k$-group schemes $T$, $U$ and $B$ of~\S\ref{sss:reductive-gps-notation}. We will say that a $\Dist(T)$-module $M$ is \emph{integrable} if there exists a (necessarily unique) decomposition $M = \bigoplus_{\lambda \in \X} M_\lambda$ such that $\Dist(T)$ acts on each $M_\lambda$ via the character associated with $\lambda$. Since any $T$-module is the direct sum of its weight spaces, the essential image of the canonical functor $\Rep(T) \to \Mod(\Dist(T))$ consists precisely of the integrable modules.

For the case of $B$, the essential image of the functor above is described by the following standard lemma.

\begin{lem}
\label{lem:Rep-modules-Dist}
An object $M \in \Mod(\Dist(B))$ belongs to the essential image of the canonical functor
$\Rep(B) \to \Mod(\Dist(B))$
if and only if it satisfies the following conditions:
\begin{enumerate}
\item
$M$ is integrable as a $\Dist(T)$-module;
\item
for any $v \in M$ there exists $n \in \Z_{\geq 1}$ such that $(\Dist^+(U))^n \cdot v = 0$.
\end{enumerate}
\end{lem}

\begin{proof}[Sketch of proof]
If $M$ belongs to the essential image of the functor, then the first condition is satisfied by the comments above, and the second one by local finiteness (see~\cite[\S I.2.13]{Jan03}) and weight considerations. Reciprocally, if $M$ satisfies the stated conditions then clearly the action of $\Dist(T)$ can be integrated into an action of $T$, and the action of $\Dist(U)$ can be integrated into an action of $U$ because $\Dist(U)$ is the $\X$-graded dual of the $\X$-graded coalgebra $\sO(U)$.
\end{proof}

\subsubsection{An easy homological algebra lemma}

Below we will use the following variation on standard properties of derived categories (see e.g.~\cite[\href{https://stacks.math.columbia.edu/tag/0FCL}{Tag 0FCL}]{stacks-project}). Let $\mathsf{C}$ be an abelian category, let $\mathsf{A} \subset  \mathsf{C}$ be a Serre subcategory, and let $\mathsf{B} \subset \mathsf{C}$ be a full abelian subcategory containing $\mathsf{A}$.

\begin{lem}
\label{lem:Ext-D-isom}
Assume that for any $X$ in $\mathsf{C}$ and $Y$ in $\mathsf{A}$, and for any surjection $X \to Y$, there exists a subobject $X' \subset X$ which belongs to $\mathsf{A}$ and such that the composition $X' \to X \to Y$ is surjective. Then for any $M_1$ in $D^- (\mathsf{A})$ and $M_2$ in $D^- (\mathsf{B})$, the natural morphism
\[
\Hom_{D^- (\mathsf{B})}(M_1,M_2) \to \Hom_{D^- (\mathsf{C})}(M_1,M_2)
\]
is an isomorphism.
\end{lem}

\begin{proof}
Let us fix a bounded above complex of objects of $\mathsf{A}$, resp.~$\mathsf{B}$, representing $M_1$, resp.~$M_2$. We first prove surjectivity of our map. A morphism from $M_1$ to $M_2$ in $D^- (\mathsf{C})$ can be represented by a diagram
\[
M_1 \xleftarrow{f} M_3 \xrightarrow{g} M_2
\]
where $M_3$ is a bounded above complex of objects of $\mathsf{C}$, $f$ and $g$ are morphisms of complexes, and $f$ is a quasi-isomorphism. By the claim at the beginning of the proof of~\cite[\href{https://stacks.math.columbia.edu/tag/0FCL}{Tag 0FCL}]{stacks-project} (with the objects ``$B_i$'' chosen to be $0$), there exists a subcomplex $M_3' \subset M_3$ whose terms belong to $\mathsf{A}$, and such that the embedding $M_3' \subset M_3$ is a quasi-isomorphism. Then our morphism can be represented by the diagram
\[
M_1 \leftarrow M_3' \to M_2,
\]
hence it belongs to the image of our morphism.

Now we prove injectivity of the map. Consider a morphism $\phi : M_1 \to M_2$ in $D^-(\mathsf{B})$ whose image in $D^-(\mathsf{C})$ vanishes. Then $\phi$ can be represented by a diagram
\[
M_1 \xleftarrow{f} M_3 \xrightarrow{g} M_2
\]
where $M_3$ is a bounded above complex of objects of $\mathsf{B}$, $f$ and $g$ are morphisms of complexes, and $f$ is a quasi-isomorphism. The image of $g$ in $D^-(\mathsf{C})$ vanishes, so that there exists a bounded above complex $M_3'$ of objects of $\mathsf{C}$ and a quasi-isomorphism $M_3' \to M_3$ whose composition with $g$ is homotopic to $0$. As above there exists a subcomplex $M_3'' \subset M_3'$ whose terms belong to $\mathsf{A}$ and such that the embedding $M_3'' \to M_3'$ is a quasi-isomorphism. Then the composition 
\[
M_3'' \to M_3' \to M_3 \xrightarrow{g} M_2
\]
is homotopic to $0$, which shows that $g$ vanishes in $D^-(\mathsf{B})$, and finally that $\phi=0$.
\end{proof}

\subsection{The hybrid algebra} 

By the Poincar\'e--Birkhoff--Witt theorem, multiplication induces an isomorphism of $\C$-vector spaces
\[
\sU(\n^-_\C) \otimes_\C \sU(\b_\C) \simto \sU(\g_\C).
\]
We have the lattice $\sU(\n^-_\Z) \subset \sU(\n^-_\C)$, and the results recalled in~\S\ref{sss:Dist} show that we also have the lattice
\[
\Dist(B_\Z) \subset \C \otimes_{\Z} \Dist(B_\Z) = \Dist(B_\C)=\sU(\b_\C).
\]
In fact, a $\Z$-basis of $\Dist(B_\Z)$ can be obtained in terms of divided powers as in Kostant's $\Z$-form of $\sU(\g_\C)$, see~\cite[\S II.1.12]{Jan03} or~\cite[Theorem~1.3]{haboush}. We therefore obtain that $\sU(\n^-_\Z) \otimes_\Z \Dist(B_\Z)$ is a lattice in $\sU(\g_\C)$. The following result is well known.

\begin{lem}
\label{lem:Uhb-subring}
The lattice $\sU(\n^-_\Z) \otimes_\Z \Dist(B_\Z) \subset \sU(\g_\C)$ is a subring.
\end{lem}

\begin{proof}[Sketch of proof]
It suffices to show that $\Dist(B_\Z) \cdot \sU(\n^-_\Z) \subset \sU(\n^-_\Z) \cdot \Dist(B_\Z)$. Choosing bases of $\n^-_\Z$ and $\n_\Z$ consisting of weight vectors, and a basis of $\t_\Z$, this follows from the usual formulas for commutators in Kostant's $\Z$-form of $\sU(\g_\C)$, as e.g.~in the proof of~\cite[Chap.~VIII, \S 12, Proposition~3]{bourbaki}.
\end{proof}

The ring considered in Lemma~\ref{lem:Uhb-subring} will be called the
\textit{hybrid algebra}, and denoted
$\sU^\hb_\Z$.
We also set
\[
\sU^\hb_\k:= \sU^\hb_\Z\otimes_\Z \k,
\]
so that $\sU^\hb_\k$ contains $\sU(\n^-)$ and $\Dist(B)$ as subalgebras, and moreover multiplication induces an isomorphism of $\k$-vector spaces
\[
 \sU(\n^-) \otimes_\k \Dist(B) \simto \sU^\hb_\k.
\] 
There is a canonical $\k$-algebra morphism $\sU(\g) \rightarrow \sU^\hb_\k$ induced by the inclusion $\sU(\g_\Z) \subset \sU^\hb_\Z$. (This morphism is \emph{not} injective.)

We will also denote by $\sU^{\hb,\leq}_\Z$ the subalgebra of $\sU^{\hb}_\Z$ generated by $\sU(\n^-_\Z)$ and $\Dist(T_\Z)$; then $\sU^{\hb,\leq}_\k := \sU^{\hb,\leq}_\Z \otimes_\Z \k$ identifies in the natural way with a subalgebra in $\sU^{\hb}_\k$, and the multiplication morphism
\[
\sU(\n^-) \otimes_\k \Dist(T) \to \sU^{\hb,\leq}_\k
\]
is an isomorphism of vector spaces.

\subsection{\texorpdfstring{$(\g,B)$}{(g,B)}-modules} 
\label{ss:g-B-mod}

Our main object of study in this paper is the category
\[
\Mod(\g,B)
\]
of $(\g,B)$-modules, i.e.~strongly $B$-equivariant $\g$-modules, i.e.~$B$-equivariant $\sU(\g)$-modules on which the differential of the $B$-action coincides with the restriction of the $\sU(\g)$-action to $\sU(\b)$. 
This category identifies naturally with the full subcategory of $\Mod(\sU^\hb_\k)$ consisting of modules on which the action of $\Dist(B)$ is induced by an action of $B$, or in other words, in view of Lemma~\ref{lem:Rep-modules-Dist}, 
with the category of $\sU^\hb_\k$-modules $M$ such that:
\begin{itemize}
\item 
$M$ is integrable as a $\Dist(T)$-module;
\item
for any $v \in M$ there exists $n \in \Z_{\geq 1}$ such that $(\Dist^+(U))^n \cdot v = 0$. 
\end{itemize}

We will also denote by $\mod(\g,B)$ be the full subcategory of $\Mod(\g,B)$ consisting of objects which are finitely generated as $\sU(\g)$-modules. Since $\sU(\g)$ is noetherian, this is a Serre subcategory in $\Mod(\g,B)$.

\begin{rmk}
\begin{enumerate}
\item
The category $\mod(\g,B)$ is also considered in~\cite{losev}, where it is called ``modular category $\mathcal{O}$.'' The reason for this terminology is that the same definition, replacing the base field $\k$ by an algebraically closed field of characteristic $0$, produces the direct sum of the integral blocks of the Bernstein--Gelfand--Gelfand category $\mathcal{O}$ associated with the corresponding reductive Lie algebra. There are however important differences between the structure of $\mod(\g,B)$ and of the latter category. In particular we will see below that all simple objects in $\mod(\g,B)$ are finite-dimensional, but that this category also contains some infinite-dimensional objects; in particular, it is not the case that every object has finite length.
\item
There is another category attached to a reductive algebraic group that is sometimes called ``modular category $\mathcal{O}$,'' namely the category introduced by Soergel in~\cite{soergel}. Although, in a sense, they contain the same combinatorial information, we do not know of any particularly interesting relation between these categories.
\end{enumerate}
\end{rmk}

The category of $B$-equivariant $\sU(\g)$-modules admits a canonical action of the category $\Rep(B)$ by tensor product, and we have a (fully-faithful) monoidal functor $\Rep(T) \to \Rep(B)$ induced by the projection $B \twoheadrightarrow B/U \cong T$, hence we deduce an action of $\Rep(T)$ on this category. We also have a (fully faithful) monoidal functor $\Rep(T^{(1)}) \to \Rep(T)$ given by pullback under the Frobenius morphism $\mathrm{Fr}_T$, so we obtain an action of $\Rep(T^{(1)})$ on the category of $B$-equivariant $\sU(\g)$-modules. It is easily seen that this action preserves $\Mod(\g,B)$. With the conventions explained in~\S\ref{sss:grading-shift}, and identifying the character lattice of $T^{(1)}$ with $p\X$ (via pullback under $\mathrm{Fr}_T$), we deduce autoequivalences
\[
\langle \mu \rangle : \Mod(\g,B) \simto \Mod(\g,B)
\]
for any $\mu \in p\X$, which stabilize $\mod(\g,B)$.


In $\Mod(\g,B)$, for any $\lambda \in \X$ we have
the associated \emph{standard module} (also called \emph{Verma module}) 
\[
\Delta(\lambda)= \sU^\hb_\k \otimes_{\Dist(B)} \k_B(\lambda),
\]
where $\k_B(\lambda)$ is the 1-dimensional $B$-module associated with $\lambda$.
This module is free of rank $1$ over $\sU(\n^-)$; in particular it belongs to $\mod(\g,B)$.
In fact we have an isomorphism of $(\g,B)$-modules 
\[
\Delta(\lambda)\cong \sU(\g) \otimes_{\sU(\b)} \k_B(\lambda).
\]
It is clear that for any $\lambda \in \X$ and $\mu \in p\X$ we have a canonical isomorphism
\[
\Delta(\lambda) \langle \mu \rangle \cong \Delta(\lambda+\mu).
\]
The dimensions of the $T$-weight spaces in $\Delta(\lambda)$ can be easily expressed in terms of Kostant's partition function, as in the usual Bernstein--Gelfand--Gelfand category $\mathcal{O}$, see e.g.~\cite[\S 1.16]{Hum08}.

Using local finiteness of $B$-modules one easily sees that for any nonzero object $M \in \Mod(\g,B)$ there exist $\lambda \in \X$ and a nonzero morphism $\Delta(\lambda) \to M$. On the other hand, by standard ``highest weight'' considerations one checks that for any $\lambda \in \X$
the object $\Delta(\lambda)$ admits a unique simple quotient $\Simp(\lambda)$ in $\Mod(\g,B)$. Combining these observations we obtain that $(\Simp(\lambda) : \lambda \in \X)$ is a complete set of pairwise non-isomorphic simple objects in $\Mod(\g,B)$. These objects satisfy
\begin{equation}
\label{eqn:Simp-twist}
\Simp(\lambda) \langle \mu \rangle \cong \Simp(\lambda+\mu)
\end{equation}
for any $\lambda \in \X$ and $\mu \in p\X$.

Recall that for any affine group scheme $H$ of finite type over $\k$, the Lie algebra $\mathfrak{h}$ of $H$ is endowed with a restricted $p$-th power operation $x \mapsto x^{[p]}$, and that for any $x \in \mathfrak{h}$ the element $x^p - x^{[p]} \in \sU(\mathfrak{h})$ is central. The $\k$-subalgebra of $\sU(\mathfrak{h})$ generated by such elements is denoted by $\ZFr(\mathfrak{h})$, and called the \emph{Frobenius center} of $\sU(\mathfrak{h})$. By the Poincar\'e--Birkhoff--Witt theorem, this subalgebra naturally identifies with $\sO(\mathfrak{h}^{*(1)})$. We set
\[
\sU_0(\mathfrak{h}) := \sU(\mathfrak{h}) \otimes_{\ZFr(\mathfrak{h})} \k
\]
where $\k$ is the restriction of the trivial $\sU(\mathfrak{h})$-module to $\ZFr(\mathfrak{h})$. (Under the identification with $\sO(\mathfrak{h}^{*(1)})$, this corresponds to the character associated with $0 \in \mathfrak{h}^{*(1)}$.) It is a standard fact that the morphism~\eqref{eqn:morph-U-Dist} factors through an isomorphism
\[
\sU_0(\mathfrak{h}) \simto
\Dist(H_1)
\]
where $H_1$ is the (first) Frobenius kernel of $H$, i.e.~the kernel of $\mathrm{Fr}_H$; see~\cite[\S I.9.6]{Jan03}.

Consider now the subgroup $G_1 B$ of $G$ generated by $B$ and $G_1$. 
Any $G_1B$-module admits a natural action of $\Dist(G_1) \cong \sU_0(\g)$, hence of $\sU(\g)$, and this operation provides an equivalence of categories between $\Rep(G_1B)$ and the category of $B$-equivariant $\sU_0(\g)$-modules on which the differential of the action of $B$ coincides with the restriction of the action of $\sU_0(\g)$ to the subalgebra $\sU_0(\b)$.
In particular we have a canonical exact and fully faithful functor
\begin{equation}
\label{eqn:functor-G1B-gB}
\Rep(G_1 B) \to \Mod(\g,B),
\end{equation}
whose essential image is stable under subquotients, hence which sends simple $G_1B$-modules to simple objects of $\Mod(\g,B)$. The simple $G_1B$-modules are described in~\cite[Proposition~II.9.6]{Jan03}; they are in a canonical bijection with $\X$. More specifically, for any $\lambda \in \X$ we have a (finite-dimensional) baby Verma module
\[
\bV(\lambda) := \sU_0(\g) \otimes_{\sU_0(\b)} \k_B(\lambda),
\]
which admits a unique simple quotient in $\Rep(G_1 B)$, and the assignment sending $\lambda$ to the unique simple quotient of $\bV(\lambda)$ induces the bijection mentioned above. It is clear that for any $\lambda \in \X$ there exists a canonical surjection of $(\g,B)$-modules
\[
\Delta(\lambda) \twoheadrightarrow \bV(\lambda),
\]
so that $\Simp(\lambda)$ is the image of the simple quotient of $\bV(\lambda)$ under~\eqref{eqn:functor-G1B-gB}. The latter $G_1B$-module will therefore also be denoted $\Simp(\lambda)$. (In particular, these considerations show that $\Simp(\lambda)$ is finite-dimensional for any $\lambda \in \X$.)

Recall also that the simple objects in $\Rep(G)$ are in a natural bijection with the subset $\X^+ \subset \X$ of dominant weights; this bijection will be denoted $\lambda \mapsto \widetilde{\Simp}(\lambda)$. Recall also that a dominant weight $\lambda \in \X^+$ is called \emph{restricted} if $\langle \lambda, \alpha^\vee \rangle < p$ for all $\alpha \in \Phis$. If $\lambda$ is a restricted dominant weight, then we have
\[
\widetilde{\Simp}(\lambda)_{| G_1B} \cong \Simp(\lambda),
\]
see~\cite[Proposition~II.9.6(g)]{Jan03}.
In case $G$ has simply connected derived subgroup, the restriction of the surjection $\X \to \X/p\X$ to restricted dominant weights is still surjective; in view of~\eqref{eqn:Simp-twist}, this shows that in this case the simple objects in $\Mod(\g,B)$ can be completely described in terms of those in $\Rep(G)$.

\begin{rmk}
A notable difference between $\mod(\g,B)$ and the Bernstein--Gelfand--Gelfand category $\mathcal{O}$ is that in this setting Verma modules have no simple subobjects.\footnote{This fact, and its proof, were explained to us by G.~Dhillon.} In fact, since simple objects are finite-dimensional, the existence of a simple subobject of $\Delta(\lambda)$ would imply the existence of a noninjective morphism $\Delta(\mu) \to \Delta(\lambda)$ for some $\mu \in \X$. But Verma modules are free of rank $1$ as modules for the domain $\sU(\n^-)$, so any nonzero morphism between them must be injective.
\end{rmk}

\subsection{Some categories of modules for the hybrid algebra} 

Let us introduce more terminology regarding $\Dist(T)$-modules. We will say that an integrable $\Dist(T)$-module $M$ is
\begin{itemize}
\item
\emph{finitely integrable} if $M_\lambda$ is finite-dimensional for any $\lambda\in \X$;
\item
\emph{bounded above} if there exists a finite subset $\Omega\subset \X$ such that
\[
M_\lambda\neq 0 \quad \Rightarrow \quad \lambda\in \Omega-\Z_{\geq 0}\Phi^+;
\]
\item
\emph{bounded below} if there exists a finite subset $\Omega\subset \X$ such that
\[
M_\lambda\neq 0 \quad \Rightarrow \quad \lambda\in \Omega+\Z_{\geq 0}\Phi^+.
\]
\end{itemize}

We will denote by $\mathbf{O}^{\downarrow}$ the full subcategory of $\Mod(\sU^\hb_\k)$ whose objects are the modules which are finitely integrable and bounded above as $\Dist(T)$-modules.
It is clear that 
we have embeddings 
\[
\mod(\g,B) \subset \mathbf{O}^{\downarrow} \subset \Mod(\g,B),
\]
and that $\mathbf{O}^{\downarrow}$ is a Serre subcategory of $\Mod(\g,B)$.
The following claim follows from Lemma~\ref{lem:Ext-D-isom} (whose assumptions are satisfied because $\sU(\g)$ is noetherian and $B$-modules are locally finite).

\begin{lem}
\label{lem:Hom-gB-Odown}
For any $M_1\in \Db(\mod(\g,B))$ and $M_2\in \Db(\mathbf{O}^{\downarrow})$, the natural map 
\[
\Hom_{\Db(\mathbf{O}^{\downarrow})}(M_1,M_2) \to \Hom_{\Db\Mod(\g,B)}(M_1,M_2)
\]
is an isomorphism. 
\end{lem}

We will also denote by $\mathbf{O}^{\uparrow}$ the full subcategory of $\Mod(\sU^\hb_\k)$ whose objects are the modules which are finitely integrable and bounded below as $\Dist(T)$-modules.
There is a (contravariant) equivalence $(-)^\circledast$ from $\mathbf{O}^{\downarrow}$ to $\mathbf{O}^{\uparrow}$ such that 
\[
M^\circledast=\bigoplus\limits_{\lambda \in \X} \Hom(M_{\lambda},\k) \ \subseteq\ \Hom_\k(M,\k),
\]
where we view $\Hom_\k(M,\k)$ as a left $\sU^\hb_\k$-module using the obvious antipode map. (Here we have $(M^\circledast)_\lambda = \Hom_\k(M_{-\lambda},\k)$ for any $\lambda \in \X$.) Its inverse will also be denoted $(-)^\circledast$; it is defined by the same procedure.

\subsection{Costandard modules} 
\label{ss:costandard}


Each $\mu \in \X$ defines a $1$-dimensional $\sU^{\hb,\leq}_\k$-module $\k_\mu$ (on which $\Dist(T)$ acts via the character $\mu$ and $\sU(\n^-)$ acts trivially). 
In $\mathbf{O}^{\uparrow}$ we can therefore consider the ``lowest weight Verma module'' 
\[
\Delta^{\uparrow}(\mu)=\sU^\hb_\k \otimes_{\sU^{\hb,\leq}_\k} \k_{\mu},
\]
and define the \textit{costandard module} in $\mathbf{O}^{\downarrow}$ attached to $\lambda \in \X$ to be 
\[
\nabla(\lambda):=(\Delta^{\uparrow}(-\lambda))^\circledast.
\]
It is easily seen that for any $\lambda,\mu \in \X$ we have $\dim \nabla(\lambda)_\mu=\dim \Delta(\lambda)_\mu$.
It is clear also that for $\lambda \in \X$ and $\mu \in p\X$ we have a canonical isomorphism
\[
\nabla(\lambda) \langle \mu \rangle \cong \nabla(\lambda+\mu).
\]

\begin{lem}
\phantomsection
\label{lem:Ext-Delta-nabla}
\begin{enumerate}
\item 
\label{it:morph-D-N-1}
For any $\lambda,\mu \in \X$ and $n \in \Z$
we have 
\[
\Ext^n_{\Mod(\g,B)}(\Delta(\lambda),\nabla(\mu)) \cong 
\begin{cases}
\k & \text{if $\lambda=\mu$ and $n=0$;} \\
0 & \text{otherwise.}
\end{cases}
\]
\item 
\label{it:morph-D-N-2}
For any $\lambda \in \X$,
the image of any nonzero morphism $\Delta(\lambda)\rightarrow \nabla(\lambda)$ is isomorphic to $\Simp(\lambda)$. 
\end{enumerate}
\end{lem}

\begin{proof}
\eqref{it:morph-D-N-1}
Let  $\lambda,\mu \in \X$ and $n \in \Z$.
In view of Lemma~\ref{lem:Hom-gB-Odown} we have
\[
\Ext^n_{\Mod(\g,B)}(\Delta(\lambda),\nabla(\mu)) \cong 
\Ext^n_{\mathbf{O}^\downarrow}(\Delta(\lambda),\nabla(\mu)).
\]
Consider the full subcategory $\mathbf{O}^{\downarrow}_{\geq}$ of $\Mod(\Dist(B))$ whose objects are the modules which are finitely integrable and bounded above as $\Dist(T)$-modules. Then we have an exact ``forgetful'' functor
\[
\mathbf{O}^\downarrow \to \mathbf{O}^\downarrow_{\geq},
\]
which admits an exact left adjoint given by $\sU^\hb_\k \otimes_{\Dist(B)} (-)$. By construction $\Delta(\lambda)$ is the image under the latter functor of the $1$-dimensional $B$-module $\k_B(\lambda)$; we deduce an isomorphism
\[
\Ext^n_{\mathbf{O}^\downarrow}(\Delta(\lambda),\nabla(\mu)) \cong 
\Ext^n_{\mathbf{O}_{\geq}^\downarrow}(\k_B(\lambda),\nabla(\mu)).
\]

Consider next the full subcategory $\mathbf{O}^{\uparrow}_{\geq}$ of $\Mod(\Dist(B))$ whose objects are the modules which are finitely integrable and bounded below as $\Dist(T)$-modules. Then we have a ``graded duality'' contravariant equivalence from $\mathbf{O}^\downarrow_{\geq}$ to $\mathbf{O}^\uparrow_{\geq}$ which sends $\nabla(\mu)$ to $\Delta^{\uparrow}(-\mu)$; we deduce an isomorphism
\[
\Ext^n_{\mathbf{O}_{\geq}^\downarrow}(\k_B(\lambda),\nabla(\mu))
\cong 
\Ext^n_{\mathbf{O}_{\geq}^\uparrow}(\Delta^{\uparrow}(-\mu), \k_B(-\lambda)).
\]

Finally, consider the full subcategory $\mathbf{O}^{\uparrow}_{=}$ of $\Mod(\Dist(T))$ whose objects are the modules which are finitely integrable and bounded below. Then we have an exact ``forgetful'' functor
\[
\mathbf{O}^{\uparrow}_{\geq} \to \mathbf{O}^{\uparrow}_{=}
\]
which has an exact left adjoint given by $\Dist(B) \otimes_{\Dist(T)} (-)$. The latter functor sends the $1$-dimensional $T$-module $\k_T(-\mu)$ to $\Delta^{\uparrow}(-\mu)$; we deduce an isomorphism
\[
\Ext^n_{\mathbf{O}_{\geq}^\uparrow}(\Delta^{\uparrow}(-\mu), \k_B(-\lambda))
\cong 
\Ext^n_{\mathbf{O}_{=}^\uparrow}(\k_T(-\mu), \k_T(-\lambda)).
\]
Since the category $\mathbf{O}^{\uparrow}_{=}$ is semisimple,
the right-hand side
is $1$-dimensional if $\lambda=\mu$ and $n=0$, and vanishes otherwise. Gathering the various identifications above we deduce the claim.

\eqref{it:morph-D-N-2}
By construction there exists a surjective morphism $\Delta(\lambda) \twoheadrightarrow \Simp(\lambda)$. On the other hand, the object $\Simp(\lambda)^\circledast$ of $\mathbf{O}^{\uparrow}$ admits $-\lambda$ as its lowest weight; hence there exists a nonzero morphism $\Delta^{\uparrow}(-\lambda) \to \Simp(\lambda)^\circledast$, and then a nonzero morphism $\Simp(\lambda) \to \nabla(\lambda)$. Since its domain is simple, this morphism must be injective. The composition
\[
\Delta(\lambda) \twoheadrightarrow \Simp(\lambda) \hookrightarrow \nabla(\lambda)
\]
is a nonzero morphism whose image identifies with $\Simp(\lambda)$. By~\eqref{it:morph-D-N-1} any nonzero morphism from $\Delta(\lambda)$ to $\nabla(\lambda)$ is a multiple of that morphism, hence has the same property.
\end{proof}

\subsection{Costandard modules -- an alternative description} 

In this subsection we give another description of the costandard modules of~\S\ref{ss:costandard}, which will not play any role below but is interesting in our opinion. For this we consider the Frobenius kernel $B_1^-$ of $B^-$, and the subgroup $B_1^- T$ of $B^-$ generated by $B_1^-$ and $T$. For any $\lambda \in \X$ we have a $1$-dimensional module $\k_{B_1^- T}(\lambda)$ (on which $U_1^-$ acts trivially), and the induced $G_1B$-module $\Ind_{B_1^- T}^{G_1B}(\k_{B_1^- T}(\lambda))$.

\begin{lem}
\label{lem:costandard-induction}
For any $\lambda \in \X$ there exists a canonical isomorphism between $\nabla(\lambda)$ and the image under~\eqref{eqn:functor-G1B-gB} of $\Ind_{B_1^- T}^{G_1B}(\k_{B_1^- T}(\lambda))$.
\end{lem}

\begin{proof}
Multiplication induces an isomorphism of $\k$-schemes
\[
U \times B_1^-T \simto G_1 B,
\]
so that as $B$-modules we have $\Ind_{B_1^- T}^{G_1B}(\k_{B_1^- T}(\lambda)) \cong \sO(B/T) \otimes \k_B(\lambda)$. In particular the set of $T$-weights of this module is $\lambda - \Z_{\geq 0} \Phi^+$. As a consequence, the set of $T$-weights of $\Ind_{B_1^- T}^{G_1B}(\k_{B_1^- T}(\lambda))^\circledast$ is $-\lambda + \Z_{\geq 0} \Phi^+$, and its $(-\lambda)$-weight space has a canonical vector, hence there exists a canonical morphism $\Delta^{\uparrow}(-\lambda) \to \Ind_{B_1^- T}^{G_1B}(\k_{B_1^- T}(\lambda))^\circledast$. As $\Dist(U)$-modules, both objects are free of rank $1$, and it is clear that this map is an isomorphism. Applying the functor $(-)^\circledast$, we deduce the desired identification.
\end{proof}

\begin{rmk}
\begin{enumerate}
\item
Recall that for $\lambda \in \X$ one can consider the associated baby co-Verma module $\bV'(\lambda) = \Ind_{B^-}^{G_1 B^-}(\k_{B^-}(\lambda))$, see~\cite[\S II.9.1]{Jan03}.
For any $\lambda \in \X$ we have 
\[
\Ind_{B_1^- T}^{G_1B}(\k_{B_1^- T}(\lambda)) \cong \Ind_{G_1 T}^{G_1B}( \Ind_{B_1^- T}^{G_1 T}(\k_{B_1^- T}(\lambda))),
\]
and $\Ind_{B_1^- T}^{G_1 T}(\k_{B_1^- T}(\lambda))$ identifies with (the restriction to $G_1 T$ of) $\bV'(\lambda)$, see~\cite[\S II.9.1, Eqn.~(4)]{Jan03}. Using this observation one can check that the formula in Lemma~\ref{lem:Ext-Delta-nabla}\eqref{it:morph-D-N-1} is equivalent to the standard formula expressing $\Ext$-groups between baby Verma modules and baby co-Verma modules in $\Rep(G_1T)$, see e.g.~\cite[Lemma~II.9.9]{Jan03}.
\item
It follows in particular from Lemma~\ref{lem:costandard-induction} that the action of $\sU(\g)$ on $\nabla(\lambda)$ factors through an action of the finite-dimensional quotient $\sU_0(\g)$. Since $\nabla(\lambda)$ is not finite-dimensional, it cannot be finitely generated as a $\sU(\g)$-module.
\end{enumerate}
\end{rmk}

\subsection{Extensions from simples to costandard modules} 
\label{ss:Ext-simples-costandards}

The main result of this paper will be a description of some extension groups $\Ext^n_{\Mod(\g,B)}(\Simp(\lambda), \nabla(\mu))$ for $\lambda,\mu \in \X$ (under appropriate assumptions). In this subsection, we prove a formula that will allow to express these spaces in terms of extension groups in $T$-equivariant $\g$-modules. This formula does not require any assumption, and will be the first step in our proof.

Consider the category $\Mod^T(\sU(\n^-))$ of $T$-equivariant $\sU(\n^-)$-modules. Then we have a canonical fully faithful functor
\[
\Mod^T(\sU(\n^-))\to \Mod(\sU^{\hb,\leq}_\k),
\]
whose essential image consists of the $\sU^{\hb,\leq}_\k$-modules that are integrable over $\Dist(T)$. The $1$-dimensional $\sU^{\hb,\leq}_\k$-module $\k_\mu$ considered in~\S\ref{ss:costandard} is the image under this functor of the $T$-module $\k_T(\mu)$, endowed with the trivial action of $\sU(\n^-)$.

Let $\mathbf{O}^\uparrow_{\leq}$ be the full subcategory of $\Mod(\sU^{\hb,\leq}_\k)$ consisting of modules which are
finitely integrable and bounded below as $\Dist(T)$-modules; then we have embeddings $\mathbf{O}^\uparrow_{\leq} \subset \Mod^T(\sU(\n^-)) \subset \Mod(\sU^{\hb,\leq}_\k)$.
We have an exact restriction functor
\[
\mathbf{O}^\uparrow\rightarrow \mathbf{O}^\uparrow_{\leq},
\]
which admits an exact left adjoint
\begin{equation}
\label{eqn:induction-functor-uparrow}
\sU^\hb_\k\otimes_{\sU^{\hb,\leq}_\k}(-) :\ \mathbf{O}^\uparrow_{\leq} \rightarrow \mathbf{O}^\uparrow.
\end{equation}

\begin{lem}
\label{lem:Oup-ModTn}
The natural functor $D^+(\mathbf{O}^\uparrow_{\leq})\rightarrow D^+(\Mod^T(\sU(\n^-)))$ is fully faithful. 
\end{lem} 

\begin{proof}
By a standard result (see e.g.~\cite[Chap.~I, Proposition~4.8]{Hart66}), it is enough to show that any object in $\mathbf{O}^\uparrow_{\leq}$ can be embedded into an object of $\mathbf{O}^\uparrow_{\leq}$ which is injective in $\Mod^T(\sU(\n^-))$. 
Consider the (exact) forgetful functor $\Mod^T(\sU(\n^-))\rightarrow \Rep(T)$; it admits a right adjoint 
\[
\mathrm{Coind} : \Rep(T)\rightarrow \Mod^T(\sU(\n^-))
\]
defined by
\[
\mathrm{Coind}(M) = \bigoplus_{\nu \in \X} \Hom_{\Rep(T)}(\k_T(\nu) \otimes \sU(\n^-), M),
\]
where $\sU(\n^-)$ acts via right multiplication on itself, and $T$ acts on the factor parame\-trized by $\nu$ via the character $\nu$.
Since $\Rep(T)$ is a semisimple category, $\mathrm{Coind}(V)$ is injective in $\Mod^T(\sU(\n^-))$ for any $V\in \Rep(T)$. 

Let now $M$ be an object of $\mathbf{O}^\uparrow_{\leq}$. 
By adjunction we have a natural morphism $M \to \mathrm{Coind}(M)$, which is clearly injective. 
Since $M$ is bounded below and has finite dimensional weight spaces, so does $\mathrm{Coind}(M)$. 
In other words, $\mathrm{Coind}(M)$ is contained in $\mathbf{O}^\uparrow_{\leq}$. As explained above this object is injective, which concludes the proof.
\end{proof} 

We define the \emph{universal Verma module} associated with $\lambda\in \X$ as the $T$-equivariant $\sU(\g)$-module 
\[
\tDelta(\lambda)=\sU(\g) \otimes_{\sU(\n)} \k_T(\lambda),
\]
where $\sU(\n)$ acts on $\k_T(\lambda)$ trivially and $T$ acts diagonally.

\begin{prop}
\label{prop:Ext-gB-gT}
For any $\lambda,\mu \in \X$ and $n \in \Z$ we have an isomorphism 
\begin{equation*}
\Ext^n_{\Mod(\g,B)}(\Simp(\mu),\nabla(\lambda)) \cong \Ext^n_{\Mod^T(\sU(\g))}(\tDelta(\lambda), \Simp(\mu)). 
\end{equation*}
\end{prop}

\begin{proof} 
Fix $\lambda \in \X$ and $n \in \Z$.
For any $M\in \mod(\g,B)$
we have isomorphisms 
\begin{align*}
\Ext^n_{\Mod(\g,B)}(M,\nabla(\lambda)) 
&\cong \Ext^n_{\mathbf{O}^\downarrow}(M,\nabla(\lambda))\\ 
&\cong \Ext^n_{\mathbf{O}^\uparrow}(\Delta^{\uparrow}(-\lambda),M^\circledast)\\ 
&\cong \Ext^n_{\mathbf{O}^\uparrow_{\leq}}(\k_{-\lambda}, M^\circledast)\\ 
&\cong \Ext^n_{\Mod^T(\sU(\n^-))}(\k_T({-\lambda}), M^\circledast), 
\end{align*}
where:
\begin{itemize}
\item
the first identification follows from Lemma~\ref{lem:Hom-gB-Odown};
\item
the second identification is obtained by applying the equivalence $(-)^\circledast$;
\item
the third isomorphism follows from the fact that~\eqref{eqn:induction-functor-uparrow} sends the object $\k_{-\lambda}$ to $\Delta^\uparrow(-\lambda)$;
\item
the fourth isomorphism follows from Lemma~\ref{lem:Oup-ModTn}. 
\end{itemize}

Let now $\tau$ be a Chevalley involution of $G$, i.e.~an involution which satisfies $\tau(t)=t^{-1}$ for any $t \in T$, and denote similarly the induced involution of $\g$. (It is a standard fact that such an involution exists, see e.g.~the proof of~\cite[Corollary~II.1.16]{Jan03}.) Then $\tau$ restricts to an isomorphism from $\n$ to $\n^-$, so that
any $\sU(\n^-)$-module $V$ determines a $\sU(\n)$-module, denoted ${}^\tau V$, with the same underlying vector space. 
If $V$ admits a $T$-equivariant structure, then ${}^\tau V$ has a canonical $T$-equivariant structure, and we have $({}^\tau V)_\lambda=V_{-\lambda}$ for any $\lambda \in \X$. 
For $M$ as above, we therefore have an isomorphism 
\[
\Ext^n_{\Mod^T(\sU(\n^-))}(\k_T({-\lambda}), M^\circledast) 
\cong \Ext^n_{\Mod^T(\sU(\n))}(\k_T({\lambda}), {}^\tau (M^\circledast)).
\]
 Finally, we observe that the forgetful functor $\Mod^T(\sU(\g)) \to \Mod^T(\sU(\n))$ admits an exact left adjoint (given by $\sU(\g) \otimes_{\sU(\n)} (-)$) which sends $\k_T(\lambda)$ to $\tDelta(\lambda)$, which provides an isomorphism
 \[
\Ext^n_{\Mod^T(\sU(\n))}(\k_T({\lambda}), {}^\tau (M^\circledast)) \cong \Ext^n_{\Mod^T(\sU(\g))}(\tDelta(\lambda), {}^\tau (M^\circledast))
\]
where ${}^\tau (M^\circledast)$ is $M^\circledast$ with the action of $\sU(\g)$ and $T$ twisted by $\tau$.

The desired isomorphism is obtained by applying the above series of identifications to $M=\Simp(\mu)$, noting that ${}^\tau (\Simp(\mu)^\circledast) \cong \Simp(\mu)$ since this module is simple of highest weight $\mu$. 
\end{proof} 

The category $\Mod^T(\sU(\g))$ admits a canonical action of $\Rep(T)$ by tensor product. Using the convention of~\S\ref{sss:grading-shift} we therefore have autoequivalences
\[
\langle \lambda \rangle : \Mod^T(\sU(\g)) \simto \Mod^T(\sU(\g))
\]
for $\lambda \in \X$.
In case $\lambda \in p\X$, this autoequivalence is compatible in the natural way with the autoequivalence of $\Mod(\g,B)$ denoted similarly in~\S\ref{ss:g-B-mod}.

\subsection{Reduction to the simply-connected case} 

Starting from Section~\ref{sec:central-completions}, for simplicity (and ease of references) we will assume that the group $G$ is semisimple and simply-connected.
Here we briefly explain why the study of the category $\Mod(\g,B)$ in general can be reduced to this case.

By definition any object in $\Mod(\g,B)$ is the direct sum of its $T$-weight spaces. Moreover, as for $\lambda \in \X$ the $\lambda$-graded component of $\sU^{\hb}_\k$ vanishes unless $\lambda$ belongs to $\Z \Phi$, for any object $M$ and any choice of coset $\gamma \in \X/\Z\Phi$, the direct sum $M^\gamma$ of the weight spaces of $M$ corresponding to weights in $\gamma$ is stable under the action of $\sU^\hb_\k$. This provides a decomposition
\[
\Mod(\g,B) = \prod_{\gamma \in \X/\Z\Phi} \Mod^\gamma(\g,B)
\]
where $\Mod^\gamma(\g,B)$ is the full subcategory consisting of objects $M$ such that $M=M^\gamma$. Here, of course, for any $\lambda \in \X$ the objects $\Delta(\lambda)$ and $\nabla(\lambda)$ belong to $\Mod^{\lambda + \Z\Phi}(\g,B)$.

Now, consider the derived subgroup $G_{\mathrm{der}}$ of $G$. Denote by $\g_{\mathrm{der}}$ the Lie algebra of $G_{\mathrm{der}}$ (which identifies with a Lie subalgebra in $\g$) and set $B_{\mathrm{der}} = B \cap G_{\mathrm{der}}$, which is a Borel subgroup of $G_{\mathrm{der}}$ with maximal torus $T_{\mathrm{der}} = T \cap G_{\mathrm{der}}$. Let also $\X_{\mathrm{der}}$ be the character lattice of $T_{\mathrm{der}}$. Then we have a restriction morphism $\X \to \X_{\mathrm{der}}$, whose restriction to $\Z\Phi$ is injective; in fact it identifies $\Phi$ with the root system of $(G_{\mathrm{der}}, T_{\mathrm{der}})$. One can therefore consider the category $\Mod(\g_{\mathrm{der}},B_{\mathrm{der}})$, and its canonical decomposition parametrized by $\X_{\mathrm{der}}/\Z\Phi$. We have a restriction functor
\[
\Mod(\g,B) \to \Mod(\g_{\mathrm{der}},B_{\mathrm{der}}),
\]
which for any $\gamma \in \X/\Z\Phi$ restricts to an equivalence
\[
\Mod^\gamma(\g,B) \to \Mod^{\gamma_{\mathrm{der}}}(\g_{\mathrm{der}},B_{\mathrm{der}})
\]
where $\gamma_{\mathrm{der}}$ is the image of $\gamma$ in $\X_{\mathrm{der}}/\Z\Phi$.

Next, consider the simply-connected cover $G_\mathrm{sc}$ of $G_{\mathrm{der}}$; so we have a surjective morphism of $\k$-group schemes $G_\mathrm{sc} \to G_{\mathrm{der}}$. Let also $\g_\mathrm{sc}$ be the Lie algebra of $G_\mathrm{sc}$, so that we have a canonical Lie algebra morphism $\g_\mathrm{sc} \to \g_\mathrm{der}$ (which might not be an isomorphism).
Let $B_\mathrm{sc}$, $T_\mathrm{sc}$ be the preimages of $B_{\mathrm{der}}$, $T_{\mathrm{der}}$ in $G_\mathrm{sc}$. 
Let also $\X_\mathrm{sc}$ be the character lattice of $T_\mathrm{sc}$; then pullback under the surjection $T_\mathrm{sc} \to T_\mathrm{der}$ induces an embedding $\X_\mathrm{der} \to \X_\mathrm{sc}$, such that the image of $\Phi$ is the root system of $(G_\mathrm{sc},T_\mathrm{sc})$. Consider the category $\Mod(\g_\mathrm{sc},B_\mathrm{sc})$, and its decomposition parametrized by $\X_{\mathrm{sc}}/\Z\Phi$.
Then pullback under the morphisms $\g_\mathrm{sc} \to \g_\mathrm{der}$ and $B_\mathrm{sc} \to B_\mathrm{der}$ induces a fully faithful functor
\[
\Mod(\g_{\mathrm{der}},B_{\mathrm{der}}) \to \Mod(\g_\mathrm{sc},B_\mathrm{sc}),
\] 
which for any $\gamma \in \X_{\mathrm{der}}/\Z\Phi$ induces an equivalence
\[
\Mod^\gamma(\g_{\mathrm{der}},B_{\mathrm{der}}) \simto \Mod^\gamma(\g_\mathrm{sc},B_\mathrm{sc}),
\]
where in the right-hand side we still denote by $\gamma$ the image of this element in $\X_{\mathrm{sc}}/\Z\Phi$.

These equivalences can be used to reduce any question regarding the category $\Mod(\g,B)$ to the case $G$ is semisimple and simply connected.



\section{Central completions and translation functors}
\label{sec:central-completions}

We continue with the data above, assuming for simplicity that $G$ is semisimple and simply connected. (In particular, we have $\rho \in \X$ in this case.) We also assume that $p>h$, where $h$ is the Coxeter number of $G$.

\subsection{Central completions of \texorpdfstring{$\sU(\g)$}{U(g)}} 

\subsubsection{Geometric description of the center} 

Given a weight $\lambda \in \X$, we will denote by $\overline{\lambda}$ its differential at the origin, which is an element of $\t^*$. We will consider the action, also denoted $\bullet$, of $W$ on $\t^*$ defined by
\[
w \bullet \xi = w(\xi + \overline{\rho})-\overline{\rho}
\]
for $w \in W$ and $\xi \in \t^*$. With this definition the map $\X \to \t^*$ defined by $\lambda \mapsto \overline{\lambda}$ is $W_\ex$-equivariant with respect to the $\bullet$-action on $\X$ (see~\S\ref{sss:reductive-gps-notation}) and the inflation (via the quotient morphism $W_\ex \twoheadrightarrow W$) of the $\bullet$-action of $W$ on $\t^*$, which justifies the use of the same notation.

Under our assumptions the center of the algebra $\sU(\g)$ has the following structure (see~\cite[\S 3.1.5--3.1.6]{BMR08} for details and references). As in~\S\ref{ss:g-B-mod} we have the Frobenius center $\ZFr(\g) \subset \sU(\g)$, which identifies canonically with $\sO({\g}^{*(1)})$. We also have the ``Harish-Chandra center'' $\ZHC(\g) := (\sU(\g))^G$, which forms another central subalgebra. There exists an isomorphism of algebras
\[
\ZHC(\g) \simto \sO(\t^*/(W,\bullet)),
\]
induced by the composition
\[
\ZHC(\g) \hookrightarrow \sU(\g) = (\n^- \cdot \sU(\g) + \sU(\g) \cdot \n) \oplus \mathrm{S}(\t) \twoheadrightarrow \mathrm{S}(\t) = \sO(\t^*).
\]
There exists a canonical isomorphism
\[
\ZFr(\g) \cap \ZHC(\g) \cong \sO(\t^{*(1)}/W)
\]
such that:
\begin{itemize}
\item
the embedding $\ZFr(\g) \cap \ZHC(\g) \hookrightarrow \ZFr(\g)$ corresponds under the identifications above to the (Frobenius twisted) co-adjoint quotient morphism $\g^{*(1)} \to \g^{*(1)} / G^{(1)} \cong \t^{*(1)} / W$ (where the identification here is induced by restriction of linear forms);
\item
the embedding $\ZFr(\g) \cap \ZHC(\g) \hookrightarrow \ZHC(\g)$ corresponds to the morphism $\t^*/(W,\bullet) \to \t^{*(1)}/W$ induced by the Artin--Schreier morphism $\t^* \to \t^{*(1)}$ (see~\cite[\S 2.3]{BMR08}). 
\end{itemize}
It is known that multiplication induces an algebra isomorphism
\[
\ZFr(\g) \otimes_{\ZFr(\g) \cap \ZHC(\g)} \ZHC(\g) \simto \mathrm{Z}(\sU(\g))
\]
where the right-hand side is the center of $\sU(\g)$;
we therefore deduce a canonical identification
\begin{equation}
\label{eqn:center-Ug}
\mathrm{Z}(\sU(\g)) \cong \sO({\g}^{*(1)}\times_{\t^{*(1)}/W} \t^*/(W,\bullet)).
\end{equation}

\subsubsection{Central completions and reductions}
\label{sss:central-reductions}

In view of~\eqref{eqn:center-Ug},
characters of $\mathrm{Z}(\sU(\g))$ are determined by compatible pairs $(\chi,\eta)$ with $\chi \in \g^{*(1)}$ and $\eta \in \t^*/(W,\bullet)$. For the main results of this paper we will only require the case $\chi=0$, but when this does not require more work we will consider the more general case when $\chi$ is nilpotent, i.e.~the image of $\chi$ in $\t^{*(1)}/W$ is the image of $0$. In this case, the compatibility condition mentioned above means that $\eta$ must be ``integral,'' i.e.~belong to the image of the composition $\X \to \t^* \to \t^*/(W,\bullet)$. For $\chi \in \g^{*(1)}$ we will denote by $\k_\chi$ the $1$-dimensional $\sO({\g}^{*(1)})$-module associated with the character $\chi$, and by ${\g}^{*(1)}_{\hchi}$ the spectrum of the completion of $\sO({\g}^{*(1)})$ with respect to the maximal ideal of $\chi$. We also set
\[
\sU_\chi := \sU(\g) \otimes_{\sO({\g}^{*(1)})} \k_\chi, \quad \sU_{\hchi} := \sU(\g) \otimes_{\sO({\g}^{*(1)})} \sO({\g}^{*(1)}_{\hchi}).
\]
Here, since $\sU(\g)$ is free of finite rank over $\sO({\g}^{*(1)})$, $\sU_\chi$ is a finite-dimensional algebra (of dimension $p^{\dim(\g)}$), and $\sU_{\hchi}$ is also the completion of the $\sO({\g}^{*(1)})$-module $\sU(\g)$ with respect to the ideal of $\chi$, see~\cite[\href{https://stacks.math.columbia.edu/tag/00MA}{Tag 00MA}]{stacks-project}. (In case $\chi=0$, $\sU_0$ is therefore the algebra denoted $\sU_0(\g)$ in the conventions of~\S\ref{ss:g-B-mod}.) Note also that $\sU_{\hchi}$ is a noetherian algebra. We set
\[
\mathrm{Z}(\sU(\g))_{\hchi} := \mathrm{Z}(\sU(\g)) \otimes_{\sO({\g}^{*(1)})} \sO({\g}^{*(1)}_{\hchi}).
\]
Then $\mathrm{Z}(\sU(\g))_{\hchi}$ identifies with a central subalgebra in $\sU_{\hchi}$, and is the completion of $\mathrm{Z}(\sU(\g))$ with respect to the ideal generated by the maximal ideal of $\sO({\g}^{*(1)})$ corresponding to $\chi$.

Given $\lambda \in \X$, we will denote by $\k_\lambda$ the $1$-dimensional $\sO(\t^*/(W,\bullet))$-module associated with the image of $\lambda$ in $\t^*/(W,\bullet)$, and set
\[
\sU^\lambda := \sU(\g) \otimes_{\sO(\t^*/(W,\bullet))} \k_\lambda.
\]
For instance, the action of $\sU(\g)$ on the Verma module $\Delta(\lambda)$ (see~\S\ref{ss:g-B-mod}) factors through an action of $\sU^\lambda$.
If $\chi$ is as above, we will also denote by $\mathrm{Z}(\sU(\g))_{\hchi}^{\hla}$ the completion of $\mathrm{Z}(\sU(\g))$ with respect to the ideal associated with the pair consisting of $\chi$ and the image of $\lambda$ in $\t^*/(W,\bullet)$, and set
\[
\sU_{\hchi}^{\hla} := \sU(\g) \otimes_{\mathrm{Z}(\sU(\g))} \mathrm{Z}(\sU(\g))_{\hchi}^{\hla}.
\]
As above, $\sU_{\hchi}^{\hla}$ is also the completion of the $\mathrm{Z}(\sU(\g))$-module $\sU(\g)$ with respect to the ideal associated with $\chi$ and $\lambda$. 
We set
\[
\sU_{\chi}^{\hla} := \sU_\chi \otimes_{\mathrm{Z}(\sU(\g))} \mathrm{Z}(\sU(\g))_{\hchi}^{\hla};
\]
then $\sU_{\chi}^{\hla}$ identifies with the quotient of $\sU_\chi$ by the ideal generated by the $n$-th power of the ideal of $\ZHC(\g)$ corresponding to the image of $\lambda$ for $n \gg 0$.
Finally, we set
\[
\sU_{\hchi}^\lambda := \sU^\lambda \otimes_{\sO({\g}^{*(1)})} \sO({\g}^{*(1)}_{\hchi}).
\]
It is clear that all the algebras introduced in this paragraph do not really depend on $\lambda$, but only on its $(W_\ex, \bullet)$-orbit.


If $\chi$ and $\lambda$ are as above,
we will denote by $\mod_{(\chi,\lambda)}(\sU(\g))$ the abelian category of finitely generated $\sU(\g)$-modules supported set-theoretically at the image of $(\chi,\lambda)$ in $\Spec(\mathrm{Z}(\sU(\g)))$, i.e.~on which a power of the maximal ideal of $\mathrm{Z}(\sU(\g))$ associated with $\chi$ and $\lambda$ acts trivially. Note that such modules are necessarily finite-dimensional. We also denote by
$\mod_{\chi}(\sU^{\lambda})$
the full subcategory of $\mod_{(\chi,\lambda)}(\sU(\g))$ consisting of modules on which the action of $\sU(\g)$ factors through an action of $\sU^\lambda$.
Note that the full subcategory of $\mod_{(\chi,\lambda)}(\sU(\g))$ consisting of modules on which the action of $\sU(\g)$ factors through an action of $\sU_\chi$ identifies with $\mod(\sU_\chi^{\hla})$.


\subsection{Categories of equivariant modules} 
\label{ss:equiv-modules-Ug}

Let us fix a nilpotent element $\chi \in \g^{*(1)}$, and 
a closed subgroup scheme $H \subset G$ which fixes $\chi$. (In practice, the only case which will be used below is when $\chi=0$ and $H=T$.)
The algebra $\sU(\g)$ admits an action of $G$ by conjugation, which we restrict to an action of $H$. This action induces actions of $H$ on $\sU_\chi$, $\sU_\chi^{\hla}$ and $\sU^\lambda$ (for any $\lambda \in \X$).
One can then consider the categories of $H$-equivariant $\sU(\g)$-modules, $\sU_\chi$-modules, $\sU_\chi^{\hla}$-modules or $\sU^\lambda$-modules. Imposing conditions on the action of the center as in~\S\ref{sss:central-reductions}, we define the categories
\[
\mod^H_{(\chi,\lambda)}(\sU(\g)), \quad \mod^H_{\chi}(\sU^{\lambda})
\]
in the obvious way.

The action of $H$ on $\sU(\g)$ induces an action on $\sU_{\hchi}$, but which is not algebraic. This difficulty can be bypassed using the constructions of Appendix~\ref{app:equiv-sheaves}. More specifically, the $\sO({\g}^{*(1)})$-algebra $\sU(\g)$ determines a coherent sheaf of algebras on $\g^{*(1)}$, which is $H$-equivariant with respect to the action of $H$ obtained via the composition $H \hookrightarrow G \xrightarrow{\mathrm{Fr}_G} G^{(1)}$ from the coadjoint action of $G^{(1)}$ on $\g^{*(1)}$. As in Example~\ref{ex:equiv-qcoh-alg} we deduce an $H$-equivariant coherent sheaf of algebras on the affine scheme
$\g^{*(1)}_{\hchi}$
in the sense of~\S\ref{ss:def-equiv-sheaves}, with global sections $\sU_{\hchi}$. One can therefore consider the abelian category $\Mod^H(\sU_{\hchi})$ of $H$-equivariant modules over this algebra, and its full subcategory $\mod^H(\sU_{\hchi})$ of finitely generated modules, see in particular~\S\ref{ss:app-invariants}. One defines in a similar way the categories $\Mod^H(\sU_{\hchi}^\lambda)$, $\mod^H(\sU_{\hchi}^\lambda)$.

\subsection{Translation functors}
\label{ss:translation-functors}

We continue with our nilpotent element $\chi \in \g^{*(1)}$.

\begin{lem}
\label{lem:completion-U-product}
The natural morphism
\[
\sU_{\hchi} \to \bigoplus_{\lambda\in \X/(W_\ex,\bullet)} \sU^{\hla}_{\hchi}
\]
is an isomorphism.
\end{lem}

This lemma will be deduced from the following general statement (where $\k$ can be replaced by an arbitrary field).

\begin{lem}
\label{lem:commutative-algebra}
Let $A$ be a commutative $\k$-algebra of finite type, and $I \subset A$ an ideal of finite codimension. Then
the natural morphism
\[
A_{\widehat{I}} \to \prod_{\mathfrak{m}} A_{\widehat{\mathfrak{m}}}
\]
is an isomorphism, where in the right-hand side $\mathfrak{m}$ runs over the (finitely many) maximal ideals in $A$ containing $I$.
\end{lem}

\begin{proof}
This statement is a variant of~\cite[\href{https://stacks.math.columbia.edu/tag/07N9}{Tag 07N9}]{stacks-project}, whose proof follows from standard properties of artinian rings. Namely, the maximal ideals of $A$ containing $I$ are in bijection with the maximal ideals of the finite-dimensional (hence artinian) $\k$-algebra $A/I$, 
so there exist only finitely many such ideals, see~\cite[\href{https://stacks.math.columbia.edu/tag/00J7}{Tag 00J7}]{stacks-project}. Next, for any $n \geq 0$ the $\k$-algebra $A/I^n$ is also finite-dimensional, hence artinian; it therefore identifies with the product of its localizations at its maximal ideals, see~\cite[\href{https://stacks.math.columbia.edu/tag/00JA}{Tag 00JA}]{stacks-project}. In other words the canonical morphism
\[
A/I^n \to \prod_{\mathfrak{m}} A_{\mathfrak{m}} / I^n \cdot A_{\mathfrak{m}}
\]
is an isomorphism, where $\mathfrak{m}$ runs over the maximal ideals in $A$ containing $I$ and, for such $\mathfrak{m}$, $A_{\mathrm{m}}$ is the localization of $A$ at $\mathfrak{m}$. We deduce an isomorphism
\[
A_{\widehat{I}} \simto \prod_{\mathfrak{m}} (A_{\mathfrak{m}})_{\widehat{I \cdot A_{\mathfrak{m}}}}.
\]

To conclude, we will prove that for any maximal ideal $\mathfrak{m}$ containing $I$ we have a canonical identification $(A_{\mathfrak{m}})_{\widehat{I \cdot A_{\mathfrak{m}}}} \cong A_{\widehat{\mathfrak{m}}}$.
We have $I \cdot A_{\mathfrak{m}} \subset \mathfrak{m} \cdot A_{\mathfrak{m}}$, and $A_{\mathfrak{m}} / I \cdot A_{\mathfrak{m}}$ is a local artinian ring, so that there exists $m$ such that $\mathfrak{m}^m \cdot A_{\mathfrak{m}} \subset I \cdot A_{\mathfrak{m}}$, see~\cite[\href{https://stacks.math.columbia.edu/tag/00J8}{Tag 00J8}]{stacks-project}. As a consequence we have $(A_{\mathfrak{m}})_{\widehat{I \cdot A_{\mathfrak{m}}}} \cong (A_{\mathfrak{m}})_{\widehat{\mathfrak{m} \cdot A_{\mathfrak{m}}}}$, see~\cite[\href{https://stacks.math.columbia.edu/tag/0319}{Tag 0319}]{stacks-project}. Finally, since $\mathfrak{m}$ is a maximal ideal we have $(A_{\mathfrak{m}})_{\widehat{\mathfrak{m} \cdot A_{\mathfrak{m}}}} \cong A_{\widehat{\mathfrak{m}}}$, which finishes the proof.
\end{proof}

\begin{proof}[Proof of Lemma~\ref{lem:completion-U-product}]
Applying Lemma~\ref{lem:commutative-algebra} to the ring $\mathrm{Z}(\sU(\g))$
and the ideal generated by the maximal ideal of $\sO(\g^{*(1)})$ corresponding to $\chi$, and remarking that the maximal ideals under consideration correspond to the $\k$-points of $\Spec(\mathrm{Z}(\sU(\g)))$ over $\chi$, i.e.~to the images of the pairs $(\chi,\lambda)$ with $\lambda \in \X$, we obtain an isomorphism
\[
\mathrm{Z}(\sU(\g))_{\hchi} \to \bigoplus_{\lambda\in \X/(W_\ex,\bullet)} \mathrm{Z}(\sU(\g))_{\hchi}^{\hla}.
\]
Tensoring with $\sU\g$, we deduce the desired isomorphism.
\end{proof}

The isomorphism in Lemma~\ref{lem:completion-U-product} shows that, for any $\lambda \in \X$, $\sU^{\hla}_{\hchi}$ is a finite algebra over $\sO({\g}^{*(1)}_{\hchi})$. 
It is clear that, if $H$ is as in~\S\ref{ss:equiv-modules-Ug}, the $H$-equivariant structure on $\sU_{\hchi}$ restricts to an $H$-equivariant structure on $\sU^{\hla}_{\hchi}$, so that one can consider the category $\mod^H(\sU_{\hchi}^{\hla})$, and the isomorphism in Lemma~\ref{lem:completion-U-product} is $H$-equivariant.

The isomorphism in Lemma~\ref{lem:completion-U-product} provides decompositions of categories
\begin{equation}
\label{eqn:decomposition-mod-central-char}
\Mod^H(\sU_{\hchi}) = \bigoplus_{\lambda\in \X/(W_\ex,\bullet)} \Mod^H(\sU_{\hchi}^{\hla}), \quad \mod^H(\sU_{\hchi}) = \bigoplus_{\lambda\in \X/(W_\ex,\bullet)} \mod^H(\sU_{\hchi}^{\hla}),
\end{equation}
which we will use to identify $\Mod^H(\sU_{\hchi}^{\hla})$, resp.~$\mod^H(\sU_{\hchi}^{\hla})$, with a full subcategory in $\Mod^H(\sU_{\hchi})$, resp.~$\mod^H(\sU_{\hchi})$.
The functor of projection to the direct summand $\Mod^H(\sU_{\hchi}^{\hla})$ will be denoted
\[
\pr_\lambda : \Mod^H(\sU_{\hchi}) \to \Mod^H(\sU_{\hchi}^{\hla}).
\]
These decompositions restrict to similar decompositions
\begin{equation}
\label{eqn:decomposition-mod-Uchi}
\Mod^H(\sU_{\chi}) = \bigoplus_{\lambda\in \X/(W_\ex,\bullet)} \Mod^H(\sU_{\chi}^{\hla}), \quad
\mod^H(\sU_{\chi}) = \bigoplus_{\lambda\in \X/(W_\ex,\bullet)} \mod^H(\sU_{\chi}^{\hla}).
\end{equation}


Recall from~\S\ref{ss:g-B-mod} that
for $\nu \in \X^+$ we have an associated simple object $\widetilde{\Simp}(\nu) \in \Rep(G)$.
The action of $\sU(\g)$ on this module obtained by differentiation factors through an action of $\sU_0$, so that tensoring with it defines a functor
 \[
 \widetilde{\Simp}(\nu) \otimes (-) : \Mod^H(\sU_{\hchi}) \to \Mod^H(\sU_{\hchi}).
 \] 
 Using the fact that a $\sU_{\hchi}$-module is finitely generated if and only if it is finitely generated as an $\sO(\g^{*(1)}_{\hchi})$-module, one sees that this functor restricts to a endofunctor of $\mod^H(\sU_{\hchi})$.

Let
\[
\mathscr{A}_0:= \{\eta\in \X \mid \forall \alpha\in \Phi^+, \, 0 \leq \langle \eta+\rho, \alpha^\vee \rangle \leq p\}.
\]
Under our assumption that $p>h$, this set contains the weight $0$.
Given $\lambda,\mu \in \mathscr{A}_0$, denoting by 
$\nu$ the unique element in $W(\mu-\lambda) \cap \X^+$, the attached \emph{translation functor} is defined as 
\[
\sfT^\mu_\lambda:= \pr_\mu \bigl(
\widetilde{\Simp}(\nu) \otimes (-) \bigr) : \Mod^H(\sU^{\hla}_{\hchi}) \to \Mod^H(\sU^{\hmu}_{\hchi}).
\]
This functor is exact, and restricts to a functor from $\mod^H(\sU^{\hla}_{\chi})$ to $\mod^H(\sU^{\hmu}_{\chi})$.

For any $\lambda,\mu \in \mathscr{A}_0$, we fix an isomorphism
$\widetilde{\Simp}(\nu)^* \cong \widetilde{\Simp}(-w_\circ \nu)$ for $\nu$ as above. After making such a choice, it is well known that
$(\sfT^\mu_\lambda,\sfT_\mu^\lambda)$ forms a biadjoint pair of functors. 

\begin{rmk}
\label{rmk:translation-G1T}
In case $\chi=0$ and $H=T$,
the relation between the translation functors considered above and those for $G_1T$-modules considered e.g.~in~\cite[\S II.9.22]{Jan03} is as follows. For any $\lambda,\mu \in \mathscr{A}_0$, we have the associated blocks $\Rep_\lambda(G_1T)$ and $\Rep_\mu(G_1T)$ of $\Rep(G_1T)$, and natural (fully faithful) functors
\[
\Rep_\lambda(G_1T) \to  \Mod^T(\sU^{\hla}_{0}), \quad
\Rep_\mu(G_1T) \to \Mod^T(\sU^{\hmu}_{0}).
\]
Then the obvious diagram involving these functors, the functor $\sfT^\mu_\lambda$ and the corresponding translation functor for $G_1T$-modules, commutes. (The proof is similar to that of~\cite[Lemma~6.1]{br-Hecke}.) Note that there is a subtlety here, in that ``blocks'' of $\Rep(G_1T)$ are parametrized by $(W_\aff,\bullet)$-orbits, while our subcategories only depend on the $(W_\ex,\bullet)$-orbit of the given weight. In fact, for any $\chi, H$, for any $\omega \in \Omega$ and $\lambda,\mu \in \mathscr{A}_0$ the weights $\omega \bullet \lambda$ and $\omega \bullet \mu$ also belong to $\mathscr{A}_0$, and we have
\[
\Mod^H(\sU^{\widehat{\omega \bullet \lambda}}_{\hchi}) = \Mod^H(\sU^{\hla}_{\hchi}), \quad \Mod^H(\sU^{\widehat{\omega \bullet \mu}}_{\hchi}) = \Mod^H(\sU^{\hmu}_{\hchi}), \quad \sfT^{\omega \bullet \mu}_{\omega \bullet \lambda}=\sfT^\mu_\lambda.
\]
So, for $\chi=0$ and $H=T$, the functor $\sfT^\mu_\lambda$ records all translation functors from $\omega \bullet \lambda$ to $\omega \bullet \mu$ ($\omega \in \Omega$) in the sense of $G_1T$-modules.

See also Remark~\ref{rmk:G1T} for more details about the relation between $\Rep(G_1T)$ and $\Mod^T(\sU_0)$. 
\end{rmk}

Below we will use a slightly different perspective on translation functors, introduced in~\cite{br-Hecke}. Consider the algebra $\sU(\g) \otimes_{\ZFr(\g)} \sU(\g)^{\op}$. There exists a canonical morphism from
\[
\mathrm{Z}(\sU(\g)) \otimes_{\ZFr(\g)} \mathrm{Z}(\sU(\g)) \cong \sO(\g^{*(1)} \times_{\t^{*(1)}/W} \t^*/(W,\bullet) \times_{\t^{*(1)}/W} \t^*/(W,\bullet))
\]
to this algebra, with central image. In~\cite{br-Hecke}, for $\lambda,\mu \in \mathscr{A}_0$, Bezrukavnikov and the first author consider the algebra $\sU^{\hla,\hmu}$ obtained by tensoring $\sU(\g) \otimes_{\ZFr(\g)} \sU(\g)^{\op}$ with the completion of $\sO(\t^*/(W,\bullet) \times_{\t^{*(1)}/W} \t^*/(W,\bullet))$ with respect to the ideal corresponding to the image of $(\lambda,\mu)$. This algebra admits an (algebraic!) action of $G$, and one can consider the category $\HC^{\hla,\hmu}$ of $G$-equivariant finitely generated $\sU^{\hla,\hmu}$-modules such that the differential of the $G$-action coincides with the antidiagonal action of $\g$, see~\cite[\S 3.5]{br-Hecke}. (Such objects are called Harish-Chandra $\sU^{\hla,\hmu}$-modules.) As explained in~\cite[\S 3.7]{br-Hecke}, for $\lambda,\mu,\nu \in \mathscr{A}_0$ there exists a canonical bifunctor
\[
\HC^{\hla,\hmu} \times \HC^{\hmu,\hnu} \to \HC^{\hla,\hnu}
\]
induced by tensor product of (completed) $\sU(\g)$-bimodules, which here we will denote by $\star$. When $\lambda=\mu$, the tensor product $\sU^{\hla}$ of $\sU(\g)$ with the completion of $\sO(\t^*/(W,\bullet))$ with respect to the ideal corresponding to the image of $\lambda$ (over $\ZHC(\g)$) defines naturally an object of $\HC^{\hla,\hla}$, which is a unit object for this structure in the sense that the functors $\sU^{\hla} \star (-)$ and $(-) \star \sU^{\hla}$ coincide with the identity functors.

If $\chi,H$ are as above, similar considerations show that there exists a canonical bifunctor
\[
\HC^{\hla,\hmu} \times \Mod^H(\sU^{\hmu}_{\hchi}) \to \Mod^H(\sU^{\hla}_{\hchi}),
\]
which we will also denote by $\star$,
and which restricts to a bifunctor $\HC^{\hla,\hmu} \times \mod^H(\sU^{\hmu}_{\hchi}) \to \mod^H(\sU^{\hla}_{\hchi})$. In case $\lambda=\mu$, the functor $\sU^{\hla} \star (-)$ coincides again with the identity functor. In~\cite[\S 3.5]{br-Hecke}, the authors introduce a certain object $\mathsf{P}^{\lambda,\mu} \in \HC^{\hla,\hmu}$.\footnote{Note that the object denoted $\Simp(\nu)$ in~\cite{br-Hecke} corresponds to $\widetilde{\Simp}(\nu)$ with our present conventions. Note also that the object $\mathsf{P}^{\lambda,\mu}$ is denoted $\mathsf{P}^{\hla,\hmu}$ in~\cite{br-two}.} It follows from the definitions that the functor $\mathsf{P}^{\lambda,\mu} \star (-)$ identifies with the functor $\sfT_{\mu}^\lambda$.

\subsection{Wall crossing functors}
\label{ss:wall-crossing}


We continue with our nilpotent element $\chi \in \g^{*(1)}$, and $H \subset G$ a closed subgroup scheme fixing $\chi$.

For any $s \in S_\aff$,
we choose a weight $\mu_s \in \mathscr{A}_0$ whose stabilizer in $W_\ex$ (for the $\bullet$-action) is  $\{1, s\}$ and set
\[
\Theta_s := \sfT_{\mu_s}^{0} \circ \sfT^{\mu_s}_{0} : \mod^H(\sU^{\widehat{0}}_{\hchi}) \to \mod^H(\sU^{\widehat{0}}_{\hchi}).
\]
(It is well known that such a weight exists under our assumptions; more specifically, by~\cite[\S II.6.3]{Jan03} there exists a weight whose stabilizer in $W_\aff$ is $\{1,s\}$, and then the stabilizer in $W_\ex$ is automatically the same because $p \X \cap \Z\Phi=p\Z\Phi$.) From the perspective of Harish-Chandra $\sU^{\ho,\ho}$-modules discussed in~\S\ref{ss:translation-functors}, $\Theta_s$ is therefore given by the functor $\sfR_s \star (-)$, where
\[
\sfR_s := \mathsf{P}^{0,\mu_s} \star \mathsf{P}^{\mu_s,0}.
\]
(This notation is consistent with that used in~\cite[\S 2.11]{br-two}.)

It is easily seen (using e.g.~the formula at the end of~\cite[\S 3.7]{br-Hecke}) that the choices we have fixed for the definition of the biadjunction $(\sfT_0^{\mu_s}, \sfT_{\mu_s}^0)$ also determine morphisms
\begin{equation}
\label{eqn:morphisms-Rs}
\sU^{\ho} \to \sfR_s \quad \text{and} \quad \sfR_s \to \sU^{\ho}.
\end{equation}
in $\HC^{\ho,\ho}$, which induce the adjunction morphisms $\id \to \Theta_s$ and $\Theta_s \to \id$ upon acting on $\mod^H(\sU^{\widehat{0}}_{\hchi})$.

\begin{lem}
\label{lem:morphisms-Rs-adjunction}
The morphisms in~\eqref{eqn:morphisms-Rs} are generators of the spaces $\Hom_{\HC^{\ho,\ho}}(\sU^{\ho}, \sfR_s)$ and $\Hom_{\HC^{\ho,\ho}}(\sfR_s, \sU^{\ho})$ respectively as modules over the completion of $\sO(\t^*/(W,\bullet))$ with respect to the ideal corresponding to the image of $0$.
\end{lem}

\begin{proof}
As explained in~\cite[\S 6.6]{br-Hecke}, the morphism spaces under consideration are free of rank $1$. Hence to prove the statement it suffices to remark that there exists an object $M \in \mod^G(\sU^{\ho}_{\ho})$ on which the action of $\sO(\t^*/(W,\bullet))$ factors through the character corresponding to $0$, and such that the induced morphisms $M \to \sfR_s \star M$ and $\sfR_s \star M \to M$ are nonzero. The latter condition is equivalent to requiring that $\sfT_0^{\mu_s}(M) \neq 0$, so that $M$ can e.g.~be chosen as a simple $G$-module not killed by $\sfT_0^{\mu_s}$.
\end{proof}

\subsection{Intertwining functors and the braid group action}
\label{ss:intertwining-RT}

Recall that given abe\-lian categories $\mathsf{A}$ and $\mathsf{B}$, exact functors $F_1, F_2 : \mathsf{A} \to \mathsf{B}$, and a morphism of functors $\varphi : F_1 \to F_2$, the functor sending a bounded complex $M$ of objects of $\mathsf{A}$ to the 
cocone
of the morphism of complexes $\varphi_M : F_1(M) \to F_2(M)$ sends quasi-isomorphisms to quasi-isomorphisms; it therefore induces a functor from $\Db(\mathsf{A})$ to $\Db(\mathsf{B})$. Considering this construction for the identity functor of $\mod^H(\sU^{\widehat{0}}_{\hchi})$, $\Theta_s$, and the adjunction morphism $\Theta_s \to \mathrm{id}$, we obtain a functor
\[
\mathbb{S}_s^- :
\Db \mod^H(\sU^{\widehat{0}}_{\hchi}) \to \Db \mod^H(\sU^{\widehat{0}}_{\hchi}).
\]
Similarly, using cones instead of cocones and the morphism $\id \to \Theta_s$ we obtain a functor
\[
\mathbb{S}_s^+ :
\Db \mod^H(\sU^{\widehat{0}}_{\hchi}) \to \Db \mod^H(\sU^{\widehat{0}}_{\hchi}).
\]

Below we also need a description of these functors in terms of Harish-Chandra $\sU^{\ho,\ho}$-modules, as follows. As explained in~\cite[\S 2.6]{br-two}, taking a derived tensor product instead of a tensor product in the definition of $\star$ one obtains a bifunctor
\[
D^- \HC^{\ho,\ho} \times D^- \HC^{\ho,\ho} \to D^- \HC^{\ho,\ho},
\]
which defines a monoidal structure on $D^- \HC^{\ho,\ho}$.
For simplicity this bifunctor will also be denoted $\star$; this should not lead to any confusion since for all the objects of $\HC^{\ho,\ho}$ considered above (in particular, the objects $\sfR_s$) the two possible meanings of $\star$ lead to the same objects (see the discussion in~\cite[\S 2.6]{br-two}). Similarly, we have a canonical bifunctor
\[
D^- \HC^{\ho,\ho} \times D^- \mod^H(\sU^{\widehat{0}}_{\hchi}) \to D^- \mod^H(\sU^{\widehat{0}}_{\hchi})
\]
which defines an action of $D^- \HC^{\ho,\ho}$ on $D^- \mod^H(\sU^{\widehat{0}}_{\hchi})$. This bifunctor will again be denoted $\star$, and in this setting also the functor $\sU^{\ho} \star (-)$, resp.~$\sfR_s \star (-)$, identifies with (the derived functor of) the identity functor, resp.~$\Theta_s$.

Define the objects $\sfD_s$, $\sfN_s$ in $\Db \HC^{\ho,\ho}$ so that they fit in distinguished triangles 
\[
\sfD_s\rightarrow \sfR_s\rightarrow \sU^{\ho} \xrightarrow{[1]} 
\quad \text{and}\quad 
\sU^{\ho}\rightarrow \sfR_s \rightarrow \sfN_s \xrightarrow{[1]}
\]
where the morphisms $\sfR_s\rightarrow \sU^{\ho}$ and $\sU^{\ho}\rightarrow \sfR_s$ are those of~\eqref{eqn:morphisms-Rs}. Then by construction the functor $\sfD_s \star (-)$, resp.~$\sfN_s \star (-)$, stabilizes $\Db \mod^H(\sU^{\widehat{0}}_{\hchi})$, and its restriction identifies with $\mathbb{S}_s^-$, resp.~$\mathbb{S}_s^+$.

It follows from Lemma~\ref{lem:morphisms-Rs-adjunction} that our present notation is consistent with that introduced in~\cite[\S 2.11]{br-two}. In particular, as explained in~\cite[\S 2.12]{br-two}, there exists a morphism from the extended affine braid group $\Br_\ex$ (see~\S\ref{sss:reductive-gps-notation}) to the group of invertible objects in $D^- \HC^{\ho,\ho}$ which sends $\rH_s$ to $\sfN_s$ and $(\rH_s)^{-1}$ to $\sfD_s$ for any $s \in S_\aff$. In fact, this morphism is uniquely determined by the property that it sends $\rH_s$ to $\sfN_s$ for any $s \in S_\aff$ and $\rH_\omega$ to $\mathsf{P}^{0, \omega \bullet 0}$ for any $\omega \in \Omega$. The object associated with $b \in \Br_\ex$ will be denoted $\sfN_b$, and the corresponding autoequivalence of $\Db \mod^H(\sU^{\widehat{0}}_{\hchi})$ will be denoted $\mathbb{S}_b$; we therefore have
\[
\mathbb{S}_{\rH_s} = \mathbb{S}^+_s, \quad \mathbb{S}_{(\rH_s)^{-1}} = \mathbb{S}_s^- \quad \text{for $s \in S_\aff$}
\]
and
\[
\mathbb{S}_{\rH_\omega} = \sfT_{0}^{\omega^{-1} \bullet 0} \quad \text{for $\omega \in \Omega$.}
\]

By~\cite[Lemma~6.1.2]{Ric10}, for any $s \in S_\aff \smallsetminus S$ there exist $b \in \Br_\ex$ and $t \in S$ such that $b \cdot \rH_t \cdot b^{-1} = \rH_s$. We fix such a pair $(b,t)$ once and for all.

\begin{lem}
\label{lem:conjugation-wall-crossing}
With the notation above, there exists an isomorphism
\[
\mathbb{S}_b \circ \Theta_t \circ \mathbb{S}_{b^{-1}} \cong \Theta_s.
\]
of endofunctors of $\Db \mod^H(\sU^{\widehat{0}}_{\hchi})$.
\end{lem}

\begin{proof}
This follows from~\cite[Lemma~2.24]{br-two} and the constructions above.
\end{proof}

\begin{rmk}
A statement very close to Lemma~\ref{lem:conjugation-wall-crossing} appears in~\cite[Corollary~6.1.3]{Ric10}. Unfortunately, the proof given there is incomplete, which is corrected by the arguments above.
\end{rmk}

%
%

\subsection{Completed universal Verma modules}
\label{ss:completed-Verma}

Let $\chi$ and $H$ be as in \S\ref{ss:equiv-modules-Ug}, and let $\h=\mathrm{Lie}(H)$. 
Any $H$-equivariant $\sU(\g)$-module admits a canonical action of $\sU(\h)^\op$ commuting with the actions of $\g$ and $H$, 
and defined as follows. Let $M\in \Mod^H(\sU(\g))$, and denote by $\sigma_1: \h \rightarrow \End(M)$ the restriction of the $\sU(\g)$-action to $\h$ and by $\sigma_2: \h \rightarrow \End(M)$ the differential of the $H$-action. Then for $x \in \h$ and $m \in M$ we set
\[
x \odot m = \sigma_1(x)(m) - \sigma_2(x)(m).
\]
It is easily seen that this assignment defines a right action of the Lie algebra $\h$ on $M$ 
(i.e.~that it extends to an action of $\sU(\h)^\op$),
which makes $M$ an $H$-equivariant $\sU(\g) \otimes \sU(\h)^\op$-module. It is clear that this construction is compatible with tensor products, in the sense that for $M,N$ in $\Mod^H(\sU(\g))$, the ``extra'' action of $\sU(\h)^\op$ on the tensor product $M \otimes N$ (equipped with the diagonal actions of $\g$ and $H$) is the diagonal action with respect to the similar actions on $M$ and $N$. It is also functorial, in the sense that any morphism of $H$-equivariant $\sU(\g)$-modules commutes with these extra actions.

\begin{rmk}
\label{rmk:action-t-tensor-prod}
By construction, if the actions $\sigma_1$ and $\sigma_2$ coincide, then the action of $\sU(\h)^\op$ constructed above vanishes. This observation has the following consequence. Let $V \in \Rep(G)$, and $M \in \Mod^H(\sU(\g))$. We consider the tensor product $V \otimes M$, which we equip with the diagonal actions of $\g$ and $H$ to see it as an $H$-equivariant $\sU(\g)$-module. Then the restriction to $\ZFr(\g) \otimes \sU(\h)^\op$ of the action of $\sU(\g) \otimes \sU(\h)^\op$ on $V \otimes M$ is induced by the action on $M$.
\end{rmk}

In the rest of this subsection we assume that $\chi=0$ and $H=T$. 
We identify $\mathrm{S}(\t)$ with $\sU(\t)^\op$ via the isomorphism that is identical on $\t$. 

For $\nu \in \X$
we will denote by $\sO(\g^{*(1)} \times \t^*)_{\widehat{(0,\nu)}}$ the completion of $\sO(\g^{*(1)} \times \t^*)$ with respect to the maximal ideal corresponding to the point $(0,\overline{\nu})$.
Recall the universal Verma modules $\tDelta(\lambda)$ introduced in~\S\ref{ss:Ext-simples-costandards}. For $\lambda \in \X$ we now set
\[
\hDelta(\lambda) := \tDelta(\lambda) \otimes_{\sO(\g^{*(1)} \times \t^*)} \sO(\g^{*(1)} \times \t^*)_{\widehat{(0,0)}},
\]
where the action of $\sO(\g^{*(1)} \times \t^*)$ is the restriction of the above action of $\sU(\g) \otimes \mathrm{S}(\t)$. Note that $\hDelta(\lambda)$ has a canonical structure of $T$-equivariant $\sU_{\ho}$-module. 

Below we will use the following properties of the completed Verma modules,
where we denote by $\sO(\t^*)_{\ho}$ the completion of $\sO(\t^*)$ with respect to the ideal corresponding to $0 \in \t^*$. (See~\S\ref{sss:reductive-gps-notation} for the definition of the order $\preceq$.)

\begin{lem}
\label{lem:properties-hDelta}
\begin{enumerate}
\item
\label{it:central-char-hDelta}
For any $\lambda \in \X$, $\hDelta(\lambda)$ belongs to the direct summand in the decomposition~\eqref{eqn:decomposition-mod-central-char} corresponding to the image of $\lambda$.
We have an isomorphism of $T$-equivariant $\sU_{\widehat{0}}^{\hla}$-modules 
\begin{equation}
\label{eqn:completion-tDelta-1}
\tDelta(\lambda)\otimes_{\mathrm{Z}(\sU(\g))} \mathrm{Z}(\sU(\g))_{\widehat{0}}^{\hla}
\ \cong\ 
\bigoplus_{\mu\in W\bullet \lambda} \hDelta(\mu)\langle \lambda-\mu \rangle
\end{equation}
where in the right-hand side $\sU_{\widehat{0}}^{\hla}$ acts via its actions on the modules $\hDelta(\mu)$.
In particular, $\hDelta(\lambda)$ is finitely generated over $\sU_{\widehat{0}}^{\hla}$.
\item
\label{it:morph-hDelta}
For $\lambda,\mu \in \X$ we have
\[
\Hom_{\Mod^T(\sU_{\ho})}(\hDelta(\lambda), \hDelta(\mu)) \cong \begin{cases}
\sO(\t^*)_{\ho} & \text{if $\lambda=\mu$;} \\
0 & \text{if $\lambda \not\preceq \mu$.}
\end{cases}
\]
\end{enumerate}
\end{lem}

\begin{proof}
\eqref{it:central-char-hDelta}
For the first assertion, the statement amounts to the property that the action of $\ZHC(\g)$ on $\hDelta(\lambda)$ factors through an action of its completion with respect to the maximal ideal corresponding to the image of $\lambda$. Now, from the definitions we see that this action is the composition of the ``extra'' action $\odot$ with the composition 
\[
\ZHC(\g) \cong \mathrm{S}(\t)^{(W,\bullet)} \hookrightarrow \mathrm{S}(\t) \simto \mathrm{S}(\t)
\]
where the rightmost map is the isomorphism sending $x \in \t$ to $x + \overline{\lambda}(x)$, so that the property is clear.

From this observation and considerations similar to those in the proof of Lemma~\ref{lem:completion-U-product} we also deduce an isomorphism 
\begin{equation}
\label{eqn:completion-tDelta-2}
\tDelta(\lambda)\otimes_{\mathrm{Z}(\sU(\g))} \mathrm{Z}(\sU(\g))_{\widehat{0}}^{\hla}\cong 
\bigoplus_{\mu\in W\bullet\lambda} \tDelta(\lambda)\otimes_{\sO(\g^{*(1)} \times \t^*)} (\sO(\g^{*(1)} \times \t^*))_{\widehat{(0,\mu-\lambda)}}.
\end{equation}
To proceed, note that for $\eta \in \X$ the tautological isomorphism of $T$-equivariant $\sU(\g)$-modules 
\[
\tDelta(\lambda)=\tDelta(\lambda+\eta) \langle -\eta \rangle
\]
intertwines the $\mathrm{S}(\t)$-action on the right-hand side induced by the $\odot$-action on $\tDelta(\lambda+\eta)$ with the action on the left-hand side given by the composition of the $\odot$-action with the isomorphism $\mathrm{S}(\t) \simto \mathrm{S}(\t)$ sending $x \in \t$ to $x - \overline{\eta}(x)$. 
Hence the right-hand side of \eqref{eqn:completion-tDelta-2} is isomorphic to 
\[
\bigoplus_{\mu\in W\bullet\lambda} \tDelta(\mu)\otimes_{\sO(\g^{*(1)} \times \t^*)} (\sO(\g^{*(1)} \times \t^*))_{\widehat{(0,0)}} \langle \lambda-\mu \rangle \\
=\bigoplus_{\mu\in W\bullet\lambda} \hDelta(\mu) \langle \lambda-\mu \rangle,
\]
which completes the proof of~\eqref{eqn:completion-tDelta-1}.
 
For the final assertion, by \eqref{eqn:completion-tDelta-1} the module $\hDelta(\lambda)$ is a direct summand of the module $\tDelta(\lambda)\otimes_{\mathrm{Z}(\sU(\g))} \mathrm{Z}(\sU(\g))_{\widehat{0}}^{\hla}$, which is clearly finitely generated over $\sU_{\widehat{0}}^{\hla}$. 

\eqref{it:morph-hDelta}
As explained above, each $\hDelta(\lambda)$ is finitely generated over $\sU_{\ho}$, hence over $\sO(\g^{*(1)}_{\ho})$. As a consequence, denoting by $I \subset \sO(\g^{*(1)})$ the maximal ideal corresponding to $0$, the datum of a morphism of $\sU_{\ho}$-modules from $\hDelta(\lambda)$ to $\hDelta(\mu)$ is equivalent to the datum of a projective system of morphisms of $\sU(\g)$-modules
\[
\hDelta(\lambda) \otimes_{\sO(\g_{\ho}^{*(1)})} (\sO(\g^{*(1)}_{\ho}) / I^n) \to \hDelta(\mu) \otimes_{\sO(\g_{\ho}^{*(1)})} (\sO(\g^{*(1)}_{\ho}) / I^n).
\]
Moreover, the former morphism is $T$-equivariant if and only if each morphism in the corresponding system is $T$-equivariant (see Remark~\ref{rmk:IndMod}\eqref{it:comments-formal-schemes}).

Now, recall that any morphism of $T$-equivariant $\sU(\g)$-modules is equivariant with respect to the ``extra'' $\mathrm{S}(\t)$-action. It follows from~\eqref{eqn:completion-tDelta-1} that the module $\hDelta(\lambda) \otimes_{\sO(\g_{\ho}^{*(1)})} (\sO(\g^{*(1)}_{\ho}) / I^n)$ is a quotient (in fact a direct summand) of $\tDelta(\lambda) \otimes_{\sO(\g^{*(1)})} (\sO(\g^{*(1)}) / I^n)$, hence it is generated as a $\sU(\g)$-module by a vector of $T$-weight $\lambda$. In case $\lambda \not\preceq \mu$, $\hDelta(\mu) \otimes_{\sO(\g_{\ho}^{*(1)})} (\sO(\g^{*(1)}_{\ho}) / I^n)$ has no nonzero vector of $T$-weight $\lambda$, which shows that
\[
\Hom_{\Mod^T(\sU_{\ho})}(\hDelta(\lambda), \hDelta(\mu))=0
\]
in this case.

In case $\lambda=\mu$, we similarly have
\begin{multline*}
\Hom_{\Mod^T(\sU(\g))} \bigl( \tDelta(\lambda) \otimes_{\sO(\g^{*(1)})} (\sO(\g^{*(1)}) / I^n), \tDelta(\lambda) \otimes_{\sO(\g^{*(1)})} (\sO(\g^{*(1)}) / I^n) \bigr) \\
\cong \sO(\t^*) \otimes_{\sO(\t^{*(1)})} \sO(\t^{*(1)}) / (I^n \cap \sO(\t^{*(1)})),
\end{multline*}
and $\Hom_{\Mod^T(\sU(\g))}(\hDelta(\lambda) \otimes_{\sO(\g_{\ho}^{*(1)})} (\sO(\g^{*(1)}_{\ho}) / I^n), \hDelta(\lambda) \otimes_{\sO(\g_{\ho}^{*(1)})} (\sO(\g^{*(1)}_{\ho}) / I^n))$
corresponds to the completion of the right-hand side with respect to the maximal ideal corresponding to $0 \in \t^*$. We deduce the identification $\Hom_{\Mod^T(\sU_{\ho})}(\hDelta(\lambda), \hDelta(\lambda)) \cong \sO(\t^*)_{\ho}$.
\end{proof}

The following statement is a variant in our setting of standard properties of translation functors, see~\cite[Theorems~7.6 and~7.14]{Hum08} or~\cite[Propositions~II.7.11 and~II.7.12]{Jan03}.


\begin{prop}
\label{prop:translation-hDelta}
Let $w \in W_\aff$ and $s \in S_\aff$.

\begin{enumerate}
\item
\label{it:translation-hDelta-1}
We have $\sfT^{\mu_s}_{0}(\hDelta(w \bullet 0)) \cong \hDelta(w \bullet \mu_s)$.
\item
\label{it:translation-hDelta-2}
Assume that $w \bullet 0 \prec ws \bullet 0$. Then there exists an exact sequence
\[
0 \to \hDelta(ws \bullet 0) \to \sfT_{\mu_s}^{0}(\hDelta(w \bullet \mu_s)) \to \hDelta(w \bullet 0) \to 0,
\]
and an isomorphism
\[
\mathbb{S}_s^-(\hDelta(w \bullet 0)) \cong \hDelta(ws \bullet 0).
\]
\end{enumerate}
\end{prop}

\begin{proof}
The proof is essentially the same as in~\cite{Hum08, Jan03}. Namely, for~\eqref{it:translation-hDelta-1}, we denote by $\nu_s$ the unique dominant weight in $W\mu_s$. Then we observe that we have a canonical isomorphism of $T$-equivariant $\sU(\g)$-modules
\[
\widetilde{\Simp}(\nu_s) \otimes \tDelta(w \bullet 0) \cong \sU(\g) \otimes_{\sU(\n)} (\widetilde{\Simp}(\nu_s) \otimes \k_T(w \bullet 0)).
\]
Using also Remark~\ref{rmk:action-t-tensor-prod}, we deduce an isomorphism
\[
\widetilde{\Simp}(\nu_s) \otimes \hDelta(w \bullet 0) \cong \bigl( \sU(\g) \otimes_{\sU(\n)} (\widetilde{\Simp}(\nu_s) \otimes \k_T(w \bullet 0)) \bigr) \otimes_{\sO(\g^{*(1)} \times \t^*)} \sO(\g^{*(1)} \times \t^*)_{\widehat{(0,0)}},
\]
where the action of $\sO(\g^{*(1)} \times \t^*)$ is obtained using the ``extra'' action of $\t$ on the $T$-equivariant $\sU(\g)$-module $\sU(\g) \otimes_{\sU(\n)} (\widetilde{\Simp}(\nu_s) \otimes \k_T(w \bullet 0))$. Now $\widetilde{\Simp}(\nu_s) \otimes \k_T(w \bullet 0)$ has an increasing filtration, as a $T$-equivariant $\sU(\n)$-module, with subquotients of the form $\k_T(w \bullet 0+\eta)$ where $\eta$ runs over the $T$-weights of $\widetilde{\Simp}(\nu_s)$ (counted with multiplicities). We deduce a filtration of $\widetilde{\Simp}(\nu_s) \otimes \hDelta(w \bullet 0)$ with subquotients of the form $\hDelta(w \bullet 0 + \eta)$ with $\eta$ as above. In view of Lemma~\ref{lem:properties-hDelta}\eqref{it:central-char-hDelta}, the terms in this filtration which contribute to $\sfT^{\mu_s}_{0}(\hDelta(w \bullet 0))$ are those such that $w \bullet 0 + \eta \in W_\ex \bullet \mu_s$, which is equivalent to $w \bullet 0 + \eta \in W_\aff \bullet \mu_s$. Now by~\cite[Lemma~II.7.7]{Jan03} there is exactly one such $\eta$, and moreover it has multiplicity $1$ and satisfies $w \bullet 0 + \eta = w \bullet \mu_s$, which finishes the proof.

For the first assertion in~\eqref{it:translation-hDelta-2} we proceed in the same way to compute $\sfT_{\mu_s}^{0}(\hDelta(w \bullet \mu_s))$. In this case, again by~\cite[Lemma~II.7.7]{Jan03} there are two weights which contribute, corresponding to $w \bullet 0$ and $ws \bullet 0$, and they again have multiplicity $1$. Under our assumption that $w \bullet 0 \prec ws \bullet 0$, the weight corresponding to $ws \bullet 0$ occurs earlier in the filtration, so we obtain the desired exact sequence.

To prove the second assertion in~\eqref{it:translation-hDelta-2}, we note that using~\eqref{it:translation-hDelta-1} we have $\Theta_s(\hDelta(w \bullet 0)) \cong \sfT_{\mu_s}^{0}(\hDelta(w \bullet \mu_s))$, hence an exact sequence
\[
0 \to \hDelta(ws \bullet 0) \to \Theta_s(\hDelta(w \bullet 0)) \to \hDelta(w \bullet 0) \to 0,
\]
and to conclude it suffices to prove that the surjection in this exact sequence, which we will denote $\varpi$, differs from the adjunction morphism by composition with an automorphism of $\hDelta(w \bullet 0)$. Now, by adjunction we have
\[
\Hom(\Theta_s(\hDelta(w \bullet 0)), \hDelta(w \bullet 0)) \cong \End(\sfT_{0}^{\mu_s} \hDelta(w \bullet 0)).
\]
Using~\eqref{it:translation-hDelta-1} and Lemma~\ref{lem:properties-hDelta}\eqref{it:morph-hDelta} we obtain that we have identifications
\[
\sO(\t^*)_{\ho} \simto \End(\hDelta(w \bullet 0)) \simto \End(\sfT_{0}^{\mu_s} \hDelta(w \bullet 0));
\]
hence $\Hom(\Theta_s(\hDelta(w \bullet 0)), \hDelta(w \bullet 0))$ is free of rank $1$ as a module over the algebra $\End(\hDelta(w \bullet 0))$, with a generator given by the adjunction morphism.
On the other hand, applying $\Hom(-,\hDelta(w \bullet 0))$ to the exact sequence above and using the other claim in Lemma~\ref{lem:properties-hDelta}\eqref{it:morph-hDelta}, we see that we have an isomorphism
\[
\End(\hDelta(w \bullet 0)) \xrightarrow[\sim]{(-)\circ\varpi} \Hom(\Theta_s(\hDelta(w \bullet 0)), \hDelta(w \bullet 0)).
\]
In other words, $\varpi$ is also a generator of $\Hom(\Theta_s(\hDelta(w \bullet 0)), \hDelta(w \bullet 0))$ as a module over $\End(\hDelta(w \bullet 0))$, which proves the claim.
\end{proof}

\begin{rmk}
\label{rmk:intertwining-bV}
Similar considerations show that in the setting of Proposition \ref{prop:translation-hDelta}\eqref{it:translation-hDelta-2} we have $\mathbb{S}_s^-(\bV(w \bullet 0)) \cong \bV(ws \bullet 0)$.
\end{rmk}

%
%


\subsection{A decomposition} 
\label{ss:U0-modules}

Let $\chi$ and $H$ be as in \S\ref{ss:equiv-modules-Ug}, assuming in addition that
$H$ is a subtorus in $G$. 
We set $\X_H=X^*(H)$ and $\h=\mathrm{Lie}(H)$. 
Since $\chi\in \g^{*(1)}$ is nilpotent and fixed by $H^{(1)}$, its restriction to $\h^{*(1)}$ vanishes. 

Since $H$ is abelian, we have an
identification $\mathrm{S}(\h) = \sU(\h)^\op$ which is the identity on $\h$. 
Via this identification, the construction recalled in~\S\ref{ss:completed-Verma} provides an extra action, denoted $\odot$, of $\mathrm{S}(\h)$ on any $H$-equivariant $\sU(\g)$-module. 
If $M\in \Mod^H(\sU_\chi)$, the $\odot$-action on $M$ factors through an action of $\sO(\h^*\times_{\h^{*(1)}} \{ 0 \})$, where the map $\h^* \to \h^{*(1)}$ is the Artin--Schreier map, and the scheme $\h^*\times_{\h^{*(1)}} \{ 0\} $ identifies with the discrete scheme $\X_H/p\X_H$. 
We deduce decompositions 
\begin{equation}
\label{eqn:decomp-U0mod}
\Mod^H(\sU_\chi)= \bigoplus_{\gamma\in \X_H/p\X_H} \Mod^{H,\gamma}(\sU_\chi), \quad
\mod^H(\sU_\chi)= \bigoplus_{\gamma\in \X_H/p\X_H} \mod^{H,\gamma}(\sU_\chi), 
\end{equation}
such that a module is contained in the $\gamma$-summand if and only if the $\odot$-action of $\mathrm{S}(\h)$ is by the character corresponding to $\gamma$. 
For any $\gamma\in \X_H/p\X_H$, for any choice of a representative $\widetilde{\gamma}$ of $\gamma$ in $\X_H$
we have an equivalence 
\[
-\otimes \k_H(\widetilde{\gamma}): \Mod^{H,\gamma}(\sU_\chi) \xs \Mod^{H,0}(\sU_\chi).
\]
For $\lambda \in \X$ we set 
\begin{align*}
\mod^{H,\gamma}(\sU_{\chi}^{\hla}) &= \mod^{H,\gamma}(\sU_{\chi}) \cap \mod^{H}(\sU_{\chi}^{\hla}), \\
\Mod^{H,\gamma}(\sU_{\chi}^{\hla}) &= \Mod^{H,\gamma}(\sU_{\chi}) \cap \Mod^{H}(\sU_{\chi}^{\hla}).
\end{align*}

\begin{rmk}
\label{rmk:G1T}
In case $\chi=0$ and $H=T$, since $\sU_0=\mathrm{Dist}(G_1)$ we have a natural equivalence 
\[ 
\Rep(G_1\rtimes T) \simto \Mod^T(\sU_0),
\]
which restricts to an equivalence between the subcategories $\Rep(G_1T) \subset \Rep(G_1\rtimes T)$ and $\Mod^{T,0}(\sU_0) \subset \Mod^T(\sU_0)$. In the setting of Remark~\ref{rmk:translation-G1T}, this equivalence restricts, for any $\lambda \in \mathscr{A}_0$, to a decomposition
\[
\Mod^{T,0}(\sU_0^{\hla}) \cong \bigoplus_{\omega \in \Omega} \Rep_{\omega \bullet \lambda}(G_1T).
\]
\end{rmk}

%

One can ``lift''~\eqref{eqn:decomp-U0mod} to a decompositions of the category 
$\mod^H(\sU_{\hchi})$ as follows. 
By the considerations of~\S\ref{ss:infinitesimal}, for any $M\in \Mod^H(\sU_{\hchi})$ we have a canonical $\h$-action on $M$.
Thus as in \S\ref{ss:completed-Verma} we can consider an extra $\mathrm{S}(\h)$-action, denoted $\odot$, on any $H$-equivariant $\sU_{\hchi}$-module, which commutes with the actions of $H$ and $\sU_{\hchi}$. 
Let $\m_\chi\subseteq \sO(\g^{*(1)}_{\hchi})$ be the maximal ideal associated with $\chi\in \g^{*(1)}$. 
For any $M\in \mod^H(\sU_{\hchi})$ there is a natural isomorphism of $H$-equivariant $\sU_{\hchi}$-modules 
\[
M \cong \varprojlim_{n\geq 1} M/(\m_\chi)^n M. 
\]
Since each $M/(\m_\chi)^n M$ is an extension of $H$-equivariant $\sU_\chi$-modules, by the discussion above the $\odot$-action of $\mathrm{S}(\h)$ on $M$ factors through an action of the direct sum of completions 
$\bigoplus_{\gamma\in \X_H/p\X_H} \mathrm{S}(\h)_{\widehat{{\gamma}}}$. 
We deduce the desired decomposition 
\begin{equation}\label{eqn:decomp-Uhat0mod}
\mod^H(\sU_{\hchi})= \bigoplus_{\gamma\in \X_H/p\X_H} \mod^{H,\gamma}(\sU_{\hchi}). 
\end{equation}
For $\lambda\in \X$ and $\gamma\in \X_H$, we set
\[
\mod^{H,\gamma}(\sU_{\hchi}^{\hla})= \mod^{H,\gamma}(\sU_{\hchi}) \cap \mod^{H}(\sU_{\hchi}^{\hla}),
\]
which is a direct summand in the category $\mod^{H}(\sU_{\hchi}^{\hla})$. 
We similarly define the categories
$\mod_{(\chi,\lambda)}^{H,\gamma}(\sU(\g))$ and $\mod^{H,\gamma}_\chi(\sU^{\lambda})$.
%

\section{Localization theorems}
\label{sec:localization}


In this subsection we explain how to adapt the constructions of~\cite{BMR08, BMR06, Ric10} to obtain ``localization theorems'' describing some categories of equivariant $\sU(\g)$-modules in terms of equivariant coherent sheaves. Once again, we consider a more general setting whenever possible for completeness, but for our main results only the case $\chi=0$ and $H=T$ will be required.

\subsection{Generalized and completed Frobenius characters}
\label{ss:localization-comp} 

We consider a nilpotent element $\chi \in \g^{*(1)}$ and a subtorus $H \subset G$ stabilizing $\chi$. 
Recall the affine scheme $\g^{*(1)}_{\hchi}$ defined in~\S\ref{sss:central-reductions}, and set
\[
\sB^{(1)}_\chi = \Groth^{(1)}\times_{\g^{*(1)}} \{\chi\} = \Spr^{(1)}\times_{\g^{*(1)}} \{\chi\}, \qquad
\Groth^{(1)}_{\hchi}=\Groth^{(1)}\times_{\g^{*(1)}} \g^{*(1)}_{\hchi}.
\] 
Considering the actions of $H$ on $\Groth^{(1)}$ and $\g^{*(1)}$ we are in the setting studied in
Appendix~\ref{app:equiv-sheaves}, so we can consider the category $\Coh^H(\Groth^{(1)}_{\widehat{\chi}})$ of $H$-equivariant coherent sheaves on $\Groth^{(1)}_{\widehat{\chi}}$.

Below we will also consider the category $\Coh^H_{\sB_\chi^{(1)}}(\Spr^{(1)})$ of $H$-equivariant coherent sheaves on $\Spr^{(1)}$ which are set-theoretically supported on $\sB^{(1)}_\chi$, i.e.~whose restriction to the complement vanishes. The comments in Example~\ref{ex:infinitesimal-neighborhoods} show that the composition
\[
\Coh^H_{\sB_\chi^{(1)}}(\Spr^{(1)}) \hookrightarrow \Coh^H(\Spr^{(1)}) \xrightarrow{i_*} \Coh^H(\Groth^{(1)}) \xrightarrow{(-)_{\hchi}} \Coh^H(\Groth_{\hchi}^{(1)})
\]
is fully faithful, where $i : \Spr^{(1)} \to \Groth^{(1)}$ is the natural closed immersion and the rightmost functor is given by pullback along the projection $\Groth_{\hchi}^{(1)} \to \Groth^{(1)}$. All the functors here are exact, and we will denote similarly their derived versions.





The following theorem proposes variants of the main results of~\cite{BMR08}. Some versions of this statement appear in~\cite{BM13}, but we include a sketch of proof for the reader's convenience. Here we say that a weight $\lambda \in \X$ is \emph{regular} if its stabilizer for the $\bullet$-action of $W_\aff$ (see~\S\ref{sss:reductive-gps-notation}) is trivial.

\begin{thm}
\label{thm:localization-comp} 
Let $\lambda\in \X$ be regular.
There exist equivalences of triangulated categories 
\begin{align}
\label{eqn:localization-comp}
\Db(\mod^H(\sU_{\hchi}^{\hla})) &\xs \Db\Coh^H(\Groth^{(1)}_{\widehat{\chi}}), \\
\label{eqn:localization-fixed-HC}
\Db(\mod^H_\chi(\sU^{\lambda})) &\xs \Db\Coh^H_{\sB_\chi^{(1)}}(\Spr^{(1)}), 
\end{align}
which are compatible in the sense that the diagram 
\begin{equation*}
\begin{tikzcd}[column sep=large]
\Db(\mod^H_\chi(\sU^{\lambda})) \arrow[r,"\eqref{eqn:localization-fixed-HC}", "\sim"'] \arrow[d] & \Db\Coh^H_{\sB_\chi^{(1)}}(\Spr^{(1)}) \arrow[d, "(i_*(-))_{\hchi}"] \\ 
\Db(\mod^H(\sU^{\hla}_{\widehat{\chi}})) \arrow[r,"\eqref{eqn:localization-comp}", "\sim"'] & \Db\Coh^H(\Groth^{(1)}_{\hchi})
\end{tikzcd}
\end{equation*}
is commutative,
where the left vertical functor is induced by the inclusion $\mod^H_\chi(\sU^{\lambda})\subset \mod^H(\sU^{\widehat{\lambda}}_{\hchi})$.
\end{thm} 

\begin{proof}[Sketch of proof]
We only explain the construction of~\eqref{eqn:localization-comp}; that of~\eqref{eqn:localization-fixed-HC} is similar, and easier since it does not require the setting of Appendix~\ref{app:equiv-sheaves}. Note that our assumption on $\lambda$ guarantees that $\overline{\lambda}$ is regular in the sense considered in~\cite{BMR08}, see e.g.~\cite[Lemma~3.1]{br-Hecke}.

Consider the $T$-torsor $p:G/U^-\rightarrow \sB=G/B^-$, where the action of $T$ on $G/U^-$ is given by $gU^- \cdot t=gtU^-$ for $t\in T$ and $g \in G$. 
Consider the sheaf of algebras $\sD_{G/U^-}$ of differential operators on $G/U^-$, and set $\tD=p_*(\sD_{G/U^-})^T$.
Consider the fiber product $\Groth^{(1)}\times_{\t^{*(1)}} \t^*$, where the map $\t^* \to \t^{*(1)}$ is the Artin--Schreier map.
Then (see~\cite[\S2.3]{BMR08}) there exists a canonical sheaf of $\sO_{\Groth^{(1)}\times_{\t^{*(1)}} \t^*}$-algebras on $\Groth^{(1)}\times_{\t^{*(1)}} \t^*$ whose pushforward under the natural (affine) map to $\sB^{(1)}$ is $(\mathrm{Fr}_{\sB})_* \tD$; this sheaf of algebras will also be denoted $\tD$ for simplicity.

The $G$-action on $G/U^-$ induces a homomorphism of algebras
\[
\sU(\g) \rightarrow \Gamma \bigl( \Groth^{(1)}\times_{\t^{*(1)}} \t^*, \tD \bigr),
\]
which is compatible with the natural morphism  
\[
\widetilde{\g}^{(1)}\times_{\t^{*(1)}} \t^* \rightarrow \g^{*(1)}\times_{\t^{*(1)}/W} \t^*/(W,\bullet)
\] 
and the identification~\eqref{eqn:center-Ug} in the natural way. This morphism induces an isomorphism
\begin{equation}
\label{eqn:global-sections-D}
\sU(\g) \otimes_{\sO(\t^*/(W,\bullet))} \sO(\t^*) \xs \Gamma \bigl( \Groth^{(1)}\times_{\t^{*(1)}} \t^*, \tD \bigr),
\end{equation}
see~\cite[Proposition~3.4.1]{BMR08}.

Fix $\lambda$ as in the theorem, and consider the setting of Example~\ref{ex:equiv-qcoh-alg} for the sheaf of algebras $\tD$ on $\widetilde{\g}^{(1)}\times_{\t^{*(1)}} \t^*$, the natural morphism
\begin{equation}
\label{eqn:morph-Groth-g}
\Groth^{(1)}\times_{\t^{*(1)}} \t^* \to \g^{*(1)} \times_{\t^{*(1)}/W} \t^*
\end{equation}
and the ideal in $\sO(\g^{*(1)} \times_{\t^{*(1)}/W} \t^*)$ corresponding to the point $(\chi,\overline{\lambda})$. (In fact, since the morphism~\eqref{eqn:morph-Groth-g} is projective, equivariant coherent sheaves in this setting can be described as in Remark~\ref{rmk:IndMod}\eqref{it:comments-formal-schemes}.)
Denote by $(\g^{*(1)} \times_{\t^{*(1)}/W} \t^*)_{\widehat{(\chi,\lambda)}}$ the spectrum of the completion of $\sO(\g^{*(1)} \times_{\t^{*(1)}/W} \t^*)$ with respect to the maximal ideal of $(\chi,\overline{\lambda})$, and by $\tD^{\hla}_{\hchi}$ the pullback of $\tD$ to 
\begin{equation}
\label{eqn:fiber-product-proof-localization}
(\Groth^{(1)}\times_{\t^{*(1)}} \t^*) \times_{\g^{*(1)} \times_{\t^{*(1)}/W} \t^*} (\g^{*(1)} \times_{\t^{*(1)}/W} \t^*)_{\widehat{(\chi,\lambda)}}.
\end{equation}

Our assumption on $\lambda$ implies that the quotient morphism $\t^* \to \t^*/(W,\bullet)$ is \'etale at $\overline{\lambda}$, see e.g.~\cite[Lemma~3.2]{br-Hecke}. As a consequence, we have an identification
\[
\sO((\g^{*(1)} \times_{\t^{*(1)}/W} \t^*)_{\widehat{(\chi,\lambda)}}) \cong
\mathrm{Z}(\sU(\g))_{\hchi}^{\hla}.
\]
Using this fact and the flat base change theorem, from~\eqref{eqn:global-sections-D} we obtain a canonical isomorphism
\[
\sU_{\hchi}^{\hla} \simto \Gamma(\tD^{\hla}_{\hchi}).
\]


Taking this identification into account, Corollary~\ref{cor:derived-pushforward} provides a derived pushforward functor 
\[
R\Gamma: \Db(\mod^H(\tD^{\hla}_{\hchi}))\rightarrow \Db(\mod^H(\sU^{\hla}_{\hchi})).
\]
Similarly, by Corollary~\ref{cor:derived-pullback}
we have a derived pullback functor 
\[
\tD^{\hla}_{\hchi}\otimes^L_{\sU^{\hla}_{\hchi}}- : D^-(\mod^H(\sU^{\hla}_{\hchi})) \rightarrow D^-(\mod^H(\tD^{\hla}_{\hchi})). 
\]
This functor preserves bounded derived categories; in fact this property does not involve equivariance, so that we can drop this structure, and then the claim follows from the considerations in~\cite[\S 5.4]{BMR08}.\footnote{In this reference the authors work with the formalism of formal schemes. However, as in Remark~\ref{rmk:IndMod}\eqref{it:comments-formal-schemes} the corresponding category of coherent sheaves is equivalent to the one we consider here; see also the discussion in~\cite[p.~842]{BM13}.} (See~\cite[\S 3.4]{br-two} for similar constructions using a different completion.)
Once this fact is established, the considerations of~\S\ref{ss:adjointness} imply that the two functors above are adjoint to each other. Then one proves as above (by reduction to the nonequivariant setting, treated in~\cite[\S 5.4]{BMR08}) that the associated adjunction morphisms are isomorphisms, and one obtains an equivalence of triangulated categories
\[
\Db(\mod^H(\tD^{\hla}_{\hchi})) \simto \Db(\mod^H(\sU^{\hla}_{\hchi})).
\]

Next by~\cite[\S\S 5.1--5.2]{BMR08}, there exists a vector bundle $\sV^\lambda_{\chi}$ on the scheme in~\eqref{eqn:fiber-product-proof-localization} and an isomorphism of sheaves of algebras
\[
\tD^{\hla}_{\hchi} \cong \sEnd(\sV^\lambda_{\chi}).
\]
By the discussion in~\cite[\S 5.2.4]{BM13}, this vector bundle can be endowed with an $H$-equivariant structure such that the isomorphism above is $H$-equivariant; fixing such a structure we obtain an equivalence of categories
\[
\mod^H(\tD^{\hla}_{\hchi}) \cong \Coh^H \left( (\Groth^{(1)}\times_{\t^{*(1)}} \t^*) \times_{\g^{*(1)} \times_{\t^{*(1)}/W} \t^*} (\g^{*(1)} \times_{\t^{*(1)}/W} \t^*)_{\widehat{(\chi,\lambda)}} \right),
\]
induced by the functor $\sHom_{\tD^{\hla}_{\hchi}}(\sV^\lambda_\chi,-)$.

Finally, we claim that the natural morphism
\begin{equation}
\label{eqn:isom-completions-Groth}
(\Groth^{(1)}\times_{\t^{*(1)}} \t^*) \times_{\g^{*(1)} \times_{\t^{*(1)}/W} \t^*} (\g^{*(1)} \times_{\t^{*(1)}/W} \t^*)_{\widehat{(\chi,\lambda)}} \to \Groth^{(1)}_{\chi}
\end{equation}
is an isomorphism, which will complete the proof. In fact, the left-hand side identifies with
\[
\Groth^{(1)} \times_{\g^{*(1)} \times_{\t^{*(1)}/W} \t^{*(1)}} (\g^{*(1)} \times_{\t^{*(1)}/W} \t^*)_{\widehat{(\chi,\lambda)}}.
\]
Now the Artin--Schreier map is \'etale, and sends $\overline{\lambda}$ to $0$; it therefore induces an isomorphism
\[
(\g^{*(1)} \times_{\t^{*(1)}/W} \t^*)_{\widehat{(\chi,\lambda)}} \simto (\g^{*(1)} \times_{\t^{*(1)}/W} \t^{*(1)})_{\widehat{(\chi,0)}}
\]
where the right-hand side is the spectrum of the completion of $\sO(\g^{*(1)} \times_{\t^{*(1)}/W} \t^{*(1)})$ with respect to the ideal of $(\chi,0)$. Next, $0 \in \t^{*(1)}$ is the only $\k$-point in the preimage of its image in $\t^{*(1)}/W$; applying Lemma~\ref{lem:commutative-algebra} we deduce that this completion is also the completion of $\sO(\g^{*(1)} \times_{\t^{*(1)}/W} \t^{*(1)})$ with respect to the ideal generated by the maximal ideal of $\sO(\g^{*(1)})$ corresponding to $\chi$, i.e.~the completion of this ring, considered as an $\sO(\g^{*(1)})$-module, with respect to this maximal ideal. Geometrically, using the fact that $\g^{*(1)} \times_{\t^{*(1)}/W} \t^{*(1)}$ is finite over $\g^{*(1)}$, this provides an isomorphism
\[
(\g^{*(1)} \times_{\t^{*(1)}/W} \t^{*(1)})_{\widehat{(\chi,0)}} \cong \g_{\hchi}^{*(1)} \times_{\t^{*(1)}/W} \t^{*(1)},
\]
which finishes the proof.
\end{proof}

%

\begin{rmk}
\label{rmk:restriction-loc-neighborhood}
We have natural ``pullback'' functors
\[
\Db(\mod_{(\chi,\lambda)}^H(\sU(\g))) \to \Db(\mod^H(\sU_{\hchi}^{\hla})), \quad
\Db\Coh_{\sB_\chi^{(1)}}^H(\Groth^{(1)}) \to \Db\Coh^H(\Groth^{(1)}_{\hchi}),
\]
and it follows from Lemma~\ref{lem:completion-Ext-inf-neighborhoods} that these functors are fully faithful. 
It is easily seen that the equivalence~\eqref{eqn:localization-comp} restricts to an equivalence between these full subcategories, which is compatible with the equivalence of~\cite[Theorem~5.3.1]{BMR08} via forgetting the equivariance condition.
Similarly,~\eqref{eqn:localization-fixed-HC} is the restriction of an equivalence of categories $\Db(\mod^H(\sU_{\hchi}^{\lambda})) \simto \Db\Coh^H(\Spr^{(1)} \times_{\g^{*(1)}} \g^{*(1)}_{\widehat{\chi}})$. 
\end{rmk}

\subsection{The splitting bundle} 
\label{ss:splitting-bundle}

As should be clear from the proof of Theorem~\ref{thm:localization-comp}, the equivalences \eqref{eqn:localization-comp}--\eqref{eqn:localization-fixed-HC} depend on a choice of ``splitting vector bundle'' $\sV^\lambda_{\chi}$, and of an appropriate $H$-equivariant structure on it. We will choose the vector bundle constructed by the procedure explained in~\cite[Remark~1.3.5]{BMR06}, and the equivariant structure on it obtained as in~\cite[\S 5.2.4]{BM13}. 
The corresponding equivalences will be denoted\footnote{Note that the notational conventions of the present paper are different from those in~\cite{Ric10}, in that the named functor goes from modules to sheaves and not from sheaves to modules.}
\begin{align*}
\gamma_{\hchi}^{\hla} : 
\Db(\mod^H(\sU_{\hchi}^{\hla})) &\xs \Db\Coh^H(\Groth^{(1)}_{\widehat{\chi}}), \\
\gamma_{\hchi}^\lambda :
\Db(\mod^H_\chi(\sU^{\lambda})) &\xs \Db\Coh^H_{\sB_\chi^{(1)}}(\Spr^{(1)}).
\end{align*}

In this subsection we recall some aspects of the construction of $\sV^\lambda_{\chi}$ that will be required below.
It follows from~\cite[Corollary 3.11]{BG01} (see also \cite[Proposition 5.2.1(a)]{BMR08}) that $\sU^{\widehat{{-\rho}}}_{\hchi}$ is an Azumaya algebra over $\mathrm{Z}(\sU(\g))_{\hchi}^{\widehat{-\rho}}$.
Since the latter ring is complete local with separably closed residue field, this Azumaya algebra splits, 
and thus there is a free $\mathrm{Z}(\sU(\g))_{\hchi}^{\widehat{-\rho}}$-module  $V^{-\rho}_\chi$ such that 
\begin{equation}
\label{eqn:U-End}
\sU_{\hchi}^{\widehat{{-\rho}}} \cong \End_{\mathrm{Z}(\sU(\g))_{\hchi}^{\widehat{-\rho}}} (V^{-\rho}_\chi).
\end{equation}
Now, consider the natural morphism
\[
f : (\Groth^{(1)}\times_{\t^{*(1)}} \t^*) \times_{\g^{*(1)} \times_{\t^{*(1)}/W} \t^*} (\g^{*(1)} \times_{\t^{*(1)}/W} \t^*)_{\widehat{(\chi,-\rho)}} \to \Spec(\mathrm{Z}(\sU(\g))_{\hchi}^{\widehat{-\rho}}).
\]
By \cite[Proposition 5.2.1(b)]{BMR08} there is a canonical isomorphism 
\begin{equation}
\label{eqn:D-pullback-U}
\tD^{\widehat{-\rho}}_{\hchi}\cong f^* \sU_{\hchi}^{\widehat{{-\rho}}}. 
\end{equation} 
On the other hand we have an identification
\begin{equation}
\label{eqn:D-tensor-product}
\tD^{\hla}_{\hchi} \cong \sO_{\sB}(\lambda+\rho) \otimes_{\sO_\sB} \tD^{\widehat{-\rho}}_{\hchi} \otimes_{\sO_\sB} \sO_{\sB}(-\lambda-\rho),
\end{equation}
see~\cite[Lemma~2.3.1]{BMR08}. Hence, setting $\sV^\lambda_\chi :=\sO_{\sB}(\lambda+\rho) \otimes_{\sO_\sB} f^*V_\chi^{-\rho}$,
we obtain an isomorphism $\tD^{\hla}_{\hchi} \cong \sEnd(\sV^\lambda_\chi)$; in other words $\sV^\lambda_\chi$ is a splitting bundle as desired.

Recall (see e.g.~\cite[Proof of Lemma~4.7]{br-Hecke}) that the morphism $\t^*/(W,\bullet) \to \t^{*(1)}/W$ is \'etale at the image of $-\rho$. Hence this morphism induces an isomorphism between $\mathrm{Z}(\sU(\g))_{\hchi}^{\widehat{-\rho}}$ and $\sO(\g^{*(1)}_{\hchi})$, so that the latter ring can be replaced by $\sO(\g^{*(1)}_{\hchi})$ in the discussion above.

As explained above, the construction of the equivariant structure on $\sV_\chi^\lambda$ is discussed in~\cite[p.~882]{BM13}. It is induced by an equivariant structure on $V_\chi^{-\rho}$ for which $V^{-\rho}_{\chi}\in \mod^{H,0}(\sU^{\widehat{-\rho}}_{\hchi})$ in the notation of~\S\ref{ss:U0-modules}. (More details on this construction, in the special case $\chi=0$ and $H=T$, will be given in the proof of Proposition~\ref{prop:M-Z-w0} below.)

For the next statement we will use the following notation.
Given $\sM_1\in \Mod(\tD^{\hla}_{\hchi})$, we denote by $a_1:\h \rightarrow \End(\sM_1)$ the restriction of the action of $\tD^{\hla}_{\hchi}$ along the composition $\sU(\h)\rightarrow \sU^{\hla}_{\hchi}\rightarrow \Gamma(\tD^{\hla}_{\hchi})$. 
Given $\sM_2\in \QCoh^H(\Groth^{(1)}_{\widehat{\chi}})$, we denote by $a_2:\h \rightarrow \End(\sM_2)$ the infinitesimal action obtained from the $H$-equivariant structure on $\sM_2$, see~\S\ref{ss:infinitesimal}. 
Given $\sF\in \Mod^H(\tD^{\hla}_{\hchi})$, we denote by $a:\mathrm{S}(\h)\rightarrow \End(\sF)$ the homomorphism induced by $a_1-a_2$. 

\begin{lem}
\label{lem:action-V^lambda_chi}
The homomorphism $a:\mathrm{S}(\h)\rightarrow \End(\sV^{\lambda}_{\chi})$ factors through the completion $\mathrm{S}(\h)_{\widehat{0}}$. 
\end{lem}

\begin{proof}
By construction we have $\sV^{\lambda}_{\chi}\cong \sO_{\sB}(\lambda+\rho) \otimes_{\sO_\sB} f^* V^{-\rho}_{\chi}$, where $\tD^{\hla}_{\hchi}$ acts on $\sV^{\lambda}_{\chi}$ via the isomorphism~\eqref{eqn:D-tensor-product}.
From the construction
one sees that the $\sU(\g)$-action on $\sV^{\lambda}_{\chi}$ induced by the action of $\tD^{\hla}_{\hchi}$ is the sum of two $\sU(\g)$-actions: one on $\sO_{\sB}(\lambda+\rho)$ induced by the natural $G$-equivariant structure, and one on $f^* V^{-\rho}_{\chi}$ through the action of $\tD^{\widehat{-\rho}}_{\hchi}$. 
The equivariant structure has a similar diagonal description, so that the claim follows from the fact that $V^{-\rho}_{\chi}\in \mod^{H,0}(\sU^{\widehat{-\rho}}_{\hchi})$.
\end{proof}

\subsection{Equivalence for direct summands} 

Let $\chi,H,\h$ be as in \S\ref{ss:U0-modules}, and recall the direct summands $\mod^{H,0}(\sU_{\hchi}^{\hla})\subset \mod^{H}(\sU_{\hchi}^{\hla})$ and $\mod^{H,0}_\chi(\sU^{\lambda})\subset \mod^{H}_\chi(\sU^{\lambda})$. On the geometric side, by construction the action of $H$ on $\Groth^{(1)}$ and $\Spr^{(1)}$ factors through actions of $H^{(1)}$, so that we can consider the categories $\Coh^{H^{(1)}}(\Groth^{(1)}_{\hchi})$ and $\Coh^{H^{(1)}}_{\sB_\chi^{(1)}}(\Spr^{(1)})$.


\begin{lem}
\label{lem:localization-summands}
Let $\lambda\in \X$ be regular. 
The equivalences $\gamma_{\hchi}^{\hla}$ and $\gamma_{\hchi}^\lambda$ restrict to equivalences 
\begin{align*}
\gamma_{\hchi}^{\hla,0} : 
\Db(\mod^{H,0}(\sU_{\hchi}^{\hla})) &\xs \Db\Coh^{H^{(1)}}(\Groth^{(1)}_{\hchi}), \\
\gamma_{\hchi}^{\lambda,0} :
\Db(\mod^{H,0}_\chi(\sU^{\lambda})) &\xs \Db\Coh^{H^{(1)}}_{\sB_\chi^{(1)}}(\Spr^{(1)}).
\end{align*}
\end{lem}

\begin{proof}
We prove the claim for the equivalence $\gamma_{\hchi}^{\hla}$; the case of $\gamma_{\hchi}^{\lambda}$ is similar. Since the Frobenius kernel $H_1$ (which is a diagonalizable group scheme, with group of characters $\X_H/p\X_H$) acts trivially on $\Groth^{(1)}$, we have a decomposition
\[
\Coh^{H}(\Groth^{(1)}_{\hchi}) = \bigoplus_{\gamma \in \X_H/p\X_H} \Coh^{H,\gamma}(\Groth^{(1)}_{\hchi})
\]
according to the action of $H_1$ on the coherent sheaves, such that $\Coh^{H,0}(\Groth^{(1)}_{\hchi}) = \Coh^{H^{(1)}}(\Groth^{(1)}_{\hchi})$. Hence, to prove the lemma it suffices to remark that for any $\sF \in \Coh^{H,\gamma}(\Groth^{(1)}_{\hchi})$ the complex
\[
(\gamma_{\hchi}^{\hla})^{-1}(\sF)= R\Gamma(\sV^{\lambda}_{\chi}\otimes_{\sO_{\Groth^{(1)}_{\hchi}}} \sF)
\]
belongs to $\Db(\mod^{H,\gamma}(\sU_{\hchi}^{\hla}))$, which is an immediate consequence of Lemma~\ref{lem:action-V^lambda_chi}.
\end{proof}

\begin{rmk}
As in Remark~\ref{rmk:restriction-loc-neighborhood}, we similarly have natural equivalences
\begin{align*}
\Db(\mod_{(\chi,\lambda)}^{H,0}(\sU(\g))) &\simto \Db\Coh_{\sB_\chi^{(1)}}^{H^{(1)}}(\Groth^{(1)}), \\
\Db(\mod^{H,0}(\sU_{\hchi}^{\lambda})) &\simto \Db\Coh^{H^{(1)}}(\Spr^{(1)} \times_{\g^{*(1)}} \g^{*(1)}_{\widehat{\chi}}).
\end{align*}
\end{rmk}

\subsection{Fixed Frobenius characters}
\label{ss:localization-fixed} 


We will also need a variant of Theorem~\ref{thm:localization-comp} describing some derived categories of $\sU_\chi$-modules. Here, the nonequivariant version is~\cite[Theorem 3.4.1]{Ric10}. The answer involves (equivariant) quasi-coherent sheaves on a certain dg-scheme. Namely, consider the Koszul resolution $K_\chi$ of the $1$-dimensional $\sO(\g^{*(1)})$-module $\k_\chi$ associated with $\chi$, i.e.~the graded-symmetric algebra of the complex whose nonzero terms are $\g^{(1)} \otimes \sO(\g^{*(1)})$ in degree $-1$ and $\sO(\g^{*(1)})$ in degree $0$, with the differential sending $x \otimes f$ to $(x - \chi(x)) \cdot f$ for $x \in \g^{(1)}$ and $f \in \sO(\g^{*(1)})$. Then $K_\chi$ is a differential graded $\sO(\g^{*(1)})$-algebra, free over $\sO(\g^{*(1)})$, and quasi-isomorphic to $\k_\chi$. We can consider $K_\chi$ as a sheaf of $\sO_{\g^{*(1)}}$-dg-algebras on $\g^{*(1)}$, and then define $\sA_\chi$ as the pullback of this sheaf of dg-algebras to $\Groth^{(1)}$. We will denote by
\[
\DGQCoh^H(\Groth^{(1)}\times^R_{{\g}^{*(1)}} \{\chi\})
\]
the localization at quasi-isomorphisms of the homotopy category of the category of $H$-equivariant quasi-coherent (over $\sO_{\Groth^{(1)}}$) sheaves of $\sA_\chi$-dg-modules on $\Groth^{(1)}$, and by
\[
\DGCoh^H(\Groth^{(1)}\times^R_{{\g}^{*(1)}} \{\chi\})
\]
its full subcategory whose objects have bounded and coherent (over $\sO_{\Groth^{(1)}}$) cohomology.

We have an obvious ``forgetful'' functor $\DGCoh^H(\Groth^{(1)}\times^R_{{\g}^{*(1)}} \{\chi\}) \to \Db \Coh^H(\Groth^{(1)})$, which takes values in $\Db \Coh^H_{\sB_{\chi}^{(1)}}(\Groth^{(1)})$ Composing this functor with the functor considered in
Remark~\ref{rmk:restriction-loc-neighborhood}
we deduce a canonical functor
\begin{equation}
\label{eqn:forget-dg-completion}
\DGCoh^H(\Groth^{(1)}\times^R_{{\g}^{*(1)}} \{\chi\}) \to \Db \Coh^H(\Groth_{\hchi}^{(1)}).
\end{equation}



\begin{thm}
\label{thm:localization-fixed-Fr}
Let $\lambda \in \X$ be regular. There exists an equivalence of triangulated categories 
\begin{equation*}
\gamma^{\hla}_\chi : \Db(\mod^H(\sU_{\chi}^{\hla})) \simto \DGCoh^H(\Groth^{(1)}\times^R_{{\g}^{*(1)}} \{\chi\}), 
\end{equation*}
such that the diagram
\begin{equation*}
\begin{tikzcd}
\Db(\mod^H(\sU_\chi^{\hla})) \arrow[r,"\gamma^{\hla}_\chi"] \arrow[d] & \DGCoh^H(\widetilde{\g}^{(1)}\times^R_{{\g}^{*(1)}} \{\chi\}) \arrow[d, "\eqref{eqn:forget-dg-completion}"] \\ 
\Db(\mod^H(\sU^{\hla}_{\hchi})) \arrow[r,"\gamma_{\hchi}^{\hla}"] & \Db\Coh^H(\Groth^{(1)}_{\hchi}) 
\end{tikzcd}
\end{equation*}
is commutative,
where the left vertical functor is induced by the inclusion $\mod^H(\sU_{\chi}^{\hla}) \subset \mod^H(\sU^{\hla}_{\hchi})$.
\end{thm}

\begin{proof}
The proof can be essentially copied from that of~\cite[Theorem 3.4.1]{Ric10}. This requires a little bit of theory of derived categories of (equivariant) dg-modules over dg-algebras; here we will just outline the main steps. None of the details of this proof will be required below.

\emph{Step~1.}
We start with a general result. Let $X$ be a noetherian scheme over some base commutative ring $k$, and let $K$ be a flat affine group scheme over $k$ acting on $X$. Let $\sA$ be a sheaf of $\sO_X$-dg-algebras on $X$, which is concentrated in degrees $[-N,0]$ for some $N \geq 0$, coherent in each degree, and endowed with a compatible $K$-equivariant structure. In this setting we denote by $\DGMod^K(\sA)$ the derived category of sheaves of $\sA$-dg-modules which are quasi-coherent as $\sO_X$-modules and $K$-equivariant. Let also $Z \subset X$ be a closed subscheme, and let $\DGMod^K_{\mathrm{fg},Z}(\sA)$ be the full subcategory of $\DGMod^K(\sA)$ consisting of objects whose cohomology is bounded, coherent and set-theoretically supported on $Z$. Let also $\mathsf{C}^K_{\mathrm{fg},Z}(\sA)$ be the category of $\sA$-dg-modules which are bounded, coherent as $\sO_X$-modules and set-theoretically supported on $Z$, and let $D(\mathsf{C}^K_{\mathrm{fg},Z}(\sA))$ be the localization at quasi-isomorphisms of the associated homotopy category. Then the natural functor
\[
D(\mathsf{C}^K_{\mathrm{fg},Z}(\sA)) \to \DGMod^K_{\mathrm{fg},Z}(\sA)
\]
is an equivalence. The non-equivariant version of this statement is~\cite[Proposition~3.3.4]{Ric10}, and the same considerations apply here.\footnote{We take this opportunity to correct a minor mistake in~\cite[\S 3.3]{Ric10}: the assumption that the dg-algebra $\mathcal{Y}$ is coherent (on top of p.~69) should be imposed from the beginning of that subsection; it is necessary to ensure that the coinduction functor takes values in quasi-coherent dg-modules.} (For generalities about equivariant quasi-coherent sheaves supported on a subscheme, see also~\cite[Appendix~A]{br-two}.)

\emph{Step~2.} Consider the dg-algebra $\sU_{\mathrm{dg},\chi} := \sU(\g) \otimes_{\sO(\g^{*(1)})} K_\chi$, which is quasi-isomor\-phic to $\sU_\chi$ since $\sU(\g)$ is flat over $\sO(\g^{*(1)})$.
Let $\DGMod^H(\sU_{\mathrm{dg},\chi})$ be the derived category of $H$-equivariant $\sU_{\mathrm{dg},\chi}$-dg-modules, and consider the full subcategory
\[
\DGMod^H_{\mathrm{fg},\lambda}(\sU_{\mathrm{dg},\chi})
\]
consisting of objects whose cohomology is bounded, finitely generated (over $\sU(\g)$), and supported set-theoretically on the preimage in $\g^{*(1)} \times_{\t^{*(1)}/W} \t^*/(W,\bullet)$ of the image of $\lambda$ in $\t^*/(W,\bullet)$. (In other words, for the second condition we require that $\ZHC(\g)$ acts with generalized character determined by the image of $\lambda$.)
Then there exists an equivalence of triangulated categories
\[
\Db(\mod^H(\sU_{\chi}^{\hla})) \simto \DGMod^H_{\mathrm{fg},\lambda}(\sU_{\mathrm{dg},\chi}).
\]
This equivalence is obtained by restriction from an equivalence of triangulated categories
\[
D^-(\Mod^H(\sU_{\chi})) \simto \DGMod^{H,-}(\sU_{\mathrm{dg},\chi})
\]
(where the right-hand side denotes the full subcategory of $\DGMod^H(\sU_{\mathrm{dg},\chi})$ consisting of objects with bounded above cohomology) which follows from the quasi-isomorphism $\sU_{\mathrm{dg},\chi} \to \sU_\chi$. (Note that the standard proof of invariance of derived categories of dg-modules under quasi-isomorphisms of dg-algebras, see e.g.~\cite[Theorem~10.12.5.1]{bernstein-lunts}, involves the construction of K-flat resolutions. Here we restricted to bounded above complexes because the existence of resolutions in this context is much easier.)

\emph{Step~3.}
Consider now the natural morphism
\begin{equation}
\label{eqn:projection-Groth-g}
\Groth^{(1)} \times_{\t^{*(1)}} \t^* \to \g^{*(1)},
\end{equation}
and denote by $\tD_{\mathrm{dg},\chi}$ the tensor product of $\tD$ with the pullback of $K_\chi$. Consider also the derived category $\DGMod^H(\tD_{\mathrm{dg},\chi})$ of sheaves of quasi-coherent dg-modules over this sheaf of dg-algebras, and the full subcategory
$\DGMod^H_{\mathrm{fg},\lambda}(\tD_{\mathrm{dg},\chi})$
consisting of complexes whose cohomology is bounded, coherent (over $\sO_{\Groth^{(1)} \times_{\t^{*(1)}} \t^*}$) and supported on the preimage of $\overline{\lambda}$. Then we have an equivalence of triangulated categories
\[
\DGMod^H_{\mathrm{fg},\lambda}(\sU_{\mathrm{dg},\chi})
\simto \DGMod^H_{\mathrm{fg},\lambda}(\tD_{\mathrm{dg},\chi}).
\]
For this one constructs a derived global sections functor (from the right-hand side to the left-hand side), and a derived localization functor (in the other direction), which are adjoint, and one proves that they are quasi-inverse equivalences by reducing to the case treated in~\cite{BMR08}. (The construction of derived global sections can e.g.~be treated using K-injective resolutions. The construction of such resolutions for arbitrary dg-modules in the equivariant setting is unknown to us, but here one only needs to consider bounded below complexes, which is easier; see e.g.~\cite[Proposition~2.8]{mr}.)

\emph{Step~4.}
Let now $\sA_\chi'$ be the pullback of the quasi-coherent sheaf of dg-algebras $\sA_\chi$ under the projection map~\eqref{eqn:projection-Groth-g}. Consider the derived category $\DGMod^H(\sA_\chi')$ of $H$-equivariant quasi-coherent $\sA_\chi'$-dg-modules, and the full subcategory $\DGMod^H_{\mathrm{fg},\lambda}(\sA_\chi')$ consisting of objects whose cohomology is coherent and supported on the preimage of $\overline{\lambda}$. Then we have an equivalence of triangulated categories
\[
\DGMod^H_{\mathrm{fg},\lambda}(\tD_{\mathrm{dg},\chi}) \cong
\DGMod^H_{\mathrm{fg},\lambda}(\sA_\chi').
\]
This equivalence is obtained using the statement proved in Step~1, and then the splitting vector bundle $\sV_\chi^\lambda$ from~\S\ref{ss:localization-comp}.

\emph{Step~5.}
To conclude the construction of $\gamma_\chi^{\hla}$ it remains to construct an equivalence of triangulated categories
\[
\DGMod^H_{\mathrm{fg},\lambda}(\sA_\chi') \cong \DGCoh^H(\Groth^{(1)}\times^R_{{\g}^{*(1)}} \{\chi\}).
\]
This is obtained using the equivalence of Step~1 and the isomorphism~\eqref{eqn:isom-completions-Groth}.

\emph{Step~6.}
To prove compatibility of $\gamma_\chi^{\hla}$ with the equivalence $\gamma_{\hchi}^{\hla}$, by construction and Remark~\ref{rmk:restriction-loc-neighborhood} it suffices to prove compatibility with the induced equivalence
\[
\Db(\mod_{(\chi,\lambda)}^H(\sU(\g))) \simto
\Db\Coh_{\sB_\chi^{(1)}}^H(\Groth^{(1)})
\]
constructed as in~\cite{BMR08}, which is clear.
\end{proof}

\begin{rmk}
As in Lemma~\ref{lem:localization-summands}, we have a naturally defined triangulated category $\DGCoh^{H^{(1)}}(\Groth^{(1)}\times^R_{{\g}^{*(1)}} \{\chi\})$, which is a direct summand in $\DGCoh^{H}(\Groth^{(1)}\times^R_{{\g}^{*(1)}} \{\chi\})$, and the equivalence $\gamma^{\hla}_\chi$ restricts to an equivalence 
\[
\gamma^{\hla,0}_\chi: \Db(\mod^{H,0}(\sU_{\chi}^{\hla})) \simto \DGCoh^{H^{(1)}}(\Groth^{(1)}\times^R_{{\g}^{*(1)}} \{\chi\}). 
\]
\end{rmk}

\subsection{Standard objects}
\label{ss:standard-notation}

In this subsection we assume that $\chi=0$, $H=T$.

For $x\in W_\ex$, we abbreviate 
\[
\Delta_x=\Delta(x\bullet 0) \quad \text{and}\quad \Simp_x=\Simp(x\bullet 0),
\]
and define $\tDelta_x$, $\hDelta_x$ and $\bV_x$ similarly. 
For $x \in W_\ex$ we will denote by
$\Pro_x$ the injective hull of $\Simp_x$ as a $G_1T$-module; this module is finite-dimensional, and is also the projective cover of $\Simp_x$ as a $G_1T$-module, see~\cite[Eqn.~(3) in \S II.11.5]{Jan03}.
By the comments in~\S\ref{ss:U0-modules}, $\Pro_x$ is the projective cover and injective hull of $\Simp_x$ in $\mod^T(\sU_0)$ and in $\mod^T(\sU_0^{\ho})$. 
For any $x\in W_\ex$ and $\lambda\in \X$, we have $\Delta_{t_\lambda x}\cong \Delta_x \langle p\lambda \rangle$ as $T$-equivariant $\sU(\g)$-modules, and similar formulas hold for $\Simp_x$, $\tDelta_x$, $\hDelta_x$, $\bV_x$, and $\Pro_x$. 

For $x \in W_\ex$
we set
\[
\hM_x = \gamma^{\ho}_{\ho} (\hDelta_x), \quad \sL_x = \gamma^0_{\ho}(\Simp_x), \quad
\sZ_x = \gamma_0^{\ho}(\bV_x), \quad \sP_x = \gamma_0^{\ho}(\Pro_x).
\]
(All these objects are complexes of $T^{(1)}$-equivariant sheaves on the respective schemes or dg-schemes.)
By compatibility of the equivalences $\gamma^0_{\ho}$ and $\gamma_{\ho}^{\ho}$ (see Theorem~\ref{thm:localization-comp}), we have
\begin{equation}
\label{eqn:L-Groth}
\gamma^{\ho}_{\ho}(\Simp_x) = (i_* \sL_x)_{\ho},
\end{equation}
where we use the notation of~\S\ref{ss:localization-comp}.


For $w\in W$ we have a corresponding point $wB^-/B^- \in \sB$; we will
identify $W=\{wB^-/B^- : w\in W\}=\sB^T$. 
The preimage in $\Groth$ of $w\in \sB^T$ identifies with $(\g/w(\n^-))^*$. 
We also set $w^{(1)} = \mathrm{Fr}_{\sB}(w) \in \sB^{(1)}$. 

The following lemma is closely related to~\cite[Proposition 3.1.4]{BMR08}, and its proof is somewhat similar.

\begin{lem}
\label{lem:hDelta-Dmod}
For any $\lambda\in \X$, there is an isomorphism of $T$-equivariant $\sU^{\widehat{\lambda}}_{\widehat{0}}$-modules 
\begin{equation*}
\hDelta(w_\circ \bullet\lambda) \cong R\Gamma(\sB, \tD^{\hla}_{\widehat{0}} \otimes_{\sO_{\sB}}\sO_{\{w_\circ\}}) \langle w_\circ \bullet\lambda \rangle.
\end{equation*}
\end{lem}

\begin{proof}
Let 
\[
a_l:\sU(\g)\rightarrow \Gamma(\sB,\tD), \quad \text{resp.}\quad a_r:\sU(\t)\rightarrow \Gamma(\sB,\tD),
\]
be the algebra homomorphism induced by the left regular $G$-action, resp. $T$-torsor action, on $G/U^-$. 
Consider the $T$-equivariant $(\sU(\g)\otimes \mathrm{S}(\t))$-module 
\[
\sU(\g)\otimes_{\sU(\b)} (\k_{-2\overline{\rho}}\otimes \sU(\t)),
\] 
where 
\begin{itemize}
\item $\sU(\b)$ acts on $\k_{-2\overline{\rho}}\otimes \sU(\t)$ diagonally: it acts on $\k_{-2\overline{\rho}}=\k$ by the character corresponding to $-2\overline{\rho}\in \t^*$ and on $\sU(\t)$ via the surjection $\sU(\b)\twoheadrightarrow \sU(\b) / (\n \cdot \sU(\b)) = \sU(\t)$; 
\item the $\mathrm{S}(\t)$-action is by $w_\circ$ composed with multiplication on the factor $\sU(\t)$; 
\item $1\otimes 1\otimes 1$ is $T$-invariant. 
\end{itemize}

We firstly claim that $a_l\otimes a_r$ induces an isomorphism of $T$-equivariant $(\sU(\g)\otimes \mathrm{S}(\t))$-modules 
\begin{equation}
\label{eqn:sections-fiber-tD}
\sU(\g)\otimes_{\sU(\b)} (\k_{-2\overline{\rho}}\otimes \sU(\t)) \cong R\Gamma(\sB, \tD \otimes_{\sO_{\sB}}\sO_{\{w_\circ\}}).
\end{equation}
Here, in the right-hand side we consider a quasi-coherent sheaf supported on a point, so derived global sections are concentrated in degree $0$ (and the symbol will be omitted).
Fix a lift $\dot{w}_\circ$ of $w_\circ$ in the normalizer of $T$, and consider the commutative diagram 
\[
\begin{tikzcd}
B^- \arrow[r,"\sim"] \arrow[d,"p'"] & B^- w_\circ U^-/U^- \arrow[d,"p"] \\ 
U^- \arrow[r,"\sim"] & B^- w_\circ B^-/B^-, 
\end{tikzcd}
\] 
where the horizontal maps are induced by multiplication on $\dot{w}_\circ U^-/U^-$ and $w_\circ B^-/B^-$, respectively. 
The $T$-torsor structure on $p':B^-\rightarrow U^-$ is given by $b \cdot t=bw_\circ(t)$ for $t\in T$ and $b\in B^-$. 
Recall that for any smooth group scheme $H$ over $\k$, we have a canonical identification $\sD_{H}=\sU(\mathfrak{h})\otimes \sO_H$, where $\mathfrak{h}$ is the Lie algebra of $H$ and the embedding $\sU(\mathfrak{h})\hookrightarrow \Gamma(H,\sD_H)$ is induced by the left regular action on $H$. 
We thus identify $\sD_{B^-}=\sD_{U^-}\boxtimes \sD_T=\sD_{U^-}\boxtimes (\sU(\t)\otimes \sO_T)$, which provides identifications 
\begin{equation}
\label{eqn:tD-big-cell}
\tD_{|{U^-w_\circ B^-/B^-}}= (p')_*(\sD_{B^-})^T= \sD_{U^-}\otimes \sU(\t)
\end{equation}
and 
\begin{equation}
\label{eqn:fiber-tD-w0}
\tD\otimes_{\sO_{\sB}}\sO_{\{w_\circ\}}= (p')_*(\sD_{B^-})^T\otimes_{\sO_{U^-}}\sO_{\{e\}}= \sU(\n^-)\otimes \sU(\t).
\end{equation}
These isomorphisms respect the $T$-equivariant structures on both sides induced by left multiplication on $B^-\cong B^- w_\circ U^-/U^-$. 

We now write $U^-=\prod_{\alpha\in \Phi^+} U^-_\alpha$, where $U^-_\alpha$ is the one parameter subgroup associated with the root $-\alpha$ (for some choice of order on $\Phi^+$), and choose for any $\alpha$ a coordinate function $y_\alpha : U^-_\alpha \simto \Ga$. 
It is straightforward to compute that under the identification \eqref{eqn:tD-big-cell}, for any $x\in \t$ we have
\begin{multline}
\label{eqn:formula-al-t}
a_l(x)_{|{U^-w_\circ B^-/B^-}}=\sum_{\alpha\in \Phi^+}\alpha(x)y_\alpha\partial_{y_\alpha}\otimes1 +1\otimes x \\
=\sum_{\alpha\in \Phi^+}\alpha(x) \partial_{y_\alpha} y_\alpha \otimes1 +1\otimes (-2\overline{\rho}(x)+x), 
\end{multline}
and
\begin{equation}\label{eqn:formula-ar-t}
a_r(x)_{|{U^-w_\circ B^-/B^-}}=1\otimes w_\circ(x),
\end{equation}
and for any $u\in \sU(\n^-)$ we have 
\begin{equation}\label{eqn:formula-al-n}
a_l(u)_{|{U^-w_\circ B^-/B^-}}= u\otimes 1. 
\end{equation}
By 
\eqref{eqn:formula-al-t}--\eqref{eqn:formula-al-n}, $a_l$ induces an isomorphism 
\begin{equation}
\label{eqn:action-b-fiber-tD}
\sU(\b^-) \xs \tD\otimes_{\sO_{\sB}}\sO_{\{w_\circ\}} \overset{\eqref{eqn:fiber-tD-w0}}{=} \sU(\n^-)\otimes \sU(\t),
\end{equation}
which, for $u\in \sU(\n^-)$ and $x_1, \dots, x_r \in \t$, sends $u \cdot (x_1 \cdots x_r)$ to $u \otimes \prod_i (-2\overline{\rho}(x_i)+x_i)$.
By \eqref{eqn:fiber-tD-w0} and weight considerations, the $\sU(\b)$-action (via $a_l$) on $1\otimes 1\in \tD\otimes_{\sO_{\sB}}\sO_{\{w_\circ\}}$ factors through $\sU(\t)$. 
By \eqref{eqn:formula-al-t} and \eqref{eqn:formula-ar-t}, the $\sU(\b)\otimes \mathrm{S}(\t)$-action (via $a_l\otimes a_r$) on $1\otimes 1\in \tD\otimes_{\sO_{\sB}}\sO_{\{w_\circ\}}$ induces a homomorphism of $T$-equivariant $(\sU(\b)\otimes \mathrm{S}(\t))$-modules 
\[
\k_{-2\overline{\rho}}\otimes \sU(\t)\rightarrow \tD\otimes_{\sO_{\sB}}\sO_{\{w_\circ\}}.
\]
By adjunction this provides a homomorphism of $T$-equivariant $(\sU(\g)\otimes \mathrm{S}(\t))$-modules 
\[
\sU(\g)\otimes_{\sU(\b)} (\k_{-2\overline{\rho}}\otimes \sU(\t)) \rightarrow \tD \otimes_{\sO_{\sB}}\sO_{\{w_\circ\}},
\]
which is an isomorphism by~\eqref{eqn:action-b-fiber-tD}. 
This concludes the proof of~\eqref{eqn:sections-fiber-tD}. 

Now we prove the isomorphism of the lemma.
By \eqref{eqn:sections-fiber-tD} we have isomorphisms of $T$-equivariant $\sU_{\widehat{0}}$-modules 
\begin{align*}
R\Gamma(\sB, \tD^{\hla}_{\widehat{0}} \otimes_{\sO_{\sB}}\sO_{\{w_\circ\}})
&=R\Gamma(\sB, \tD\otimes_{\sO_{\sB}}\sO_{\{w_\circ\}})\otimes_{\sO(\g^{*(1)}\times \t^*)} \sO (\g^{*(1)}\times \t^*)_{\widehat{(0,\lambda)}} \\
&\cong \big(\sU(\g)\otimes_{\sU(\b)} (\k_{-2\overline{\rho}}\otimes \sU(\t)\big) \otimes_{\ZFr(\g)\otimes \mathrm{S}(\t)} (\ZFr(\g)\otimes \mathrm{S}(\t))_{\widehat{(0,\lambda)}},
\end{align*}
where the subscripts $\widehat{(0,\lambda)}$ mean completion with respect to the maximal ideal corresponding to $(0,\overline{\lambda})$.
Consider the isomorphism of $T$-equivariant $\sU(\g)$-modules 
\[
\sU(\g)\otimes_{\sU(\b)} (\k_{-2\overline{\rho}}\otimes \sU(\t))\langle w_\circ \bullet\lambda \rangle \simto \sU(\g)\otimes_{\sU(\n)} \k_{T}(w_\circ \bullet\lambda)= \tDelta(w_\circ \bullet\lambda)
\]
sending, for $u\in \sU(\g)$ and $x_1, \dots, x_r \in \t$, the element $u \otimes (1 \otimes (x_1 \cdots x_r)) \otimes 1$ to $u \cdot \prod_i (2\overline{\rho}(x_i)+x_i) \otimes 1$.
This isomorphism intertwines the $\mathrm{S}(\t)$-action on the left-hand side considered above with the $\mathrm{S}(\t)$-action on the right-hand side given by $(u\otimes 1) \cdot x=u(w_\circ(x)-2\overline{\rho}(x))\otimes 1$ for $u\in \sU(\g)$, $x\in \t$. This is \emph{not} the $\mathrm{S}(\t)$-action on $\tDelta(w_\circ \bullet\lambda)$ defined in \S\ref{ss:completed-Verma}; but comparing them we see that
the completion of $\tDelta(w_\circ \bullet\lambda)$ at the maximal ideal of $\ZFr(\g)\otimes \mathrm{S}(\t)$ corresponding to $(0,\lambda)$ is exactly $\hDelta(w_\circ \bullet\lambda)$. 
\end{proof}

Using Lemma~\ref{lem:hDelta-Dmod} one can describe the objects $\hM_{w_\circ}$ and $\sZ_{w_\circ}$ explicitly, as follows. Here, we identify $(\g/\n)^{*(1)}$ with the fiber of $\Groth^{(1)}$ over $w_\circ^{(1)}$, and set
\[
(\g/\n)^{*(1)}_{\ho} = (\g/\n)^{*(1)} \times_{\g^{*(1)}} \g^{*(1)}_{\ho} \subset \Groth^{(1)}_{\ho}.
\]
We also consider the point $(w_\circ^{(1)},0) \in \Groth^{(1)} \times_{\g^{*(1)}} \{0\}$.

\begin{prop}
\label{prop:M-Z-w0}
We have isomorphisms of $T^{(1)}$-equivariant sheaves 
\[
\hM_{w_\circ} \cong \sO_{(\g/\n)^{*(1)}_{\ho}} \quad \text{and} \quad 
\sZ_{w_\circ} \cong \sO_{\{(w_\circ^{(1)},0)\}}.
\]
\end{prop}

\begin{proof}
We start with the case of $\hM_{w_\circ}$. The proof will be divided into 3 steps.

\textit{Step 1.} 
Recall the construction of the splitting bundle $\sV^\lambda_0$ described in~\S\ref{ss:splitting-bundle}.
The $T$-equivariant structure on $\sV^\lambda_0$ is induced from the one on $V_0^{-\rho}$ constructed as follows.
By~\cite[Corollary 3.11]{BG01} 
$\bV(-\rho)$ is a simple module over $\sU(\g) \otimes_{\mathrm{Z}(\sU(\g))} \k_{(0,-\rho)}$, where $\k_{(0,-\rho)}$ is the $1$-dimensional $\mathrm{Z}(\sU(\g))$-module with action given by the character corresponding to the image of $(0,-\rho)$ in $\g^{*(1)} \times_{\t^{*(1)}/W} \t^*/(W,\bullet)$. On the other hand this algebra identifies with the endomorphism algebra of $V^{-\rho}_0 \otimes_{\mathrm{Z}(\sU(\g))_{\ho}^{\widehat{-\rho}}} \k_{(0,-\rho)}$; hence there exists an isomorphism of $\sU(\g)$-modules (unique up to scalar)
\begin{equation}
\label{eqn:splitting-bundle-0}
\varphi: \bV(-\rho) \cong V^{-\rho}_0 \otimes_{\mathrm{Z}(\sU(\g))_{\ho}^{\widehat{-\rho}}} \k_{(0,-\rho)},
\end{equation}
which equips the right-hand side with a $T$-equivariant structure.
As explained in~\cite[\S 5.2.4]{BM13}, this $T$-equivariant structure can be lifted to $V_0^{-\rho}$ in such a way that the isomorphism~\eqref{eqn:U-End} is $T$-equivariant. 

\textit{Step 2.} 
We will now explain how to ``lift'' the isomorphism $\varphi$ above to an isomorphism of $T$-equivariant $\sU_{\ho}^{\widehat{{-\rho}}}$-modules 
\begin{equation}
\label{eqn:hDelta-V}
\hDelta(-\rho)\cong V_0^{-\rho} \otimes_{\sO(\g^{*(1)}_{\ho})} \sO((\g/\n)^{*(1)}_{\ho}). 
\end{equation}
For $n\geq 1$, denote by $R_n$ the quotient of $\sO((\g/\n)^{*(1)})$ (equivalently, of $\sO((\g/\n)_{\ho}^{*(1)})$) by the ideal generated by $((\g/\n)^{(1)})^n$, and set ${V}_{0,n}^{-\rho}=(V^{-\rho}_0) \otimes_{\sO(\g_{\ho}^{*(1)})} R_n$. 
By construction of the $T$-equivariant structure, the $(-\rho)$-weight space $({V}^{-\rho}_{0,1})_{-\rho}$ has dimension $1$, and $({V}^{-\rho}_{0,1})_{\lambda}\neq 0$ only if $\lambda\preceq -\rho$. We also have $(R_n)_{\lambda}\neq 0$ only if $\lambda\preceq 0$. 
For $n \geq 1$, since ${V}^{-\rho}_{0,n}$ is a $T$-equivariant free $R_n$-module and ${V}^{-\rho}_{0,n} \otimes_{R_n} R_1={V}^{-\rho}_{0,1}$, we deduce that $({V}^{-\rho}_{0,n})_{-\rho}$ has dimension $1$ and $({V}^{-\rho}_{0,n})_{\lambda}\neq 0$ only if $\lambda\preceq -\rho$. 
Hence there exists a unique homomorphism of $T$-equivariant $\sU(\g)$-modules $\tDelta(-\rho) \rightarrow {V}^{-\rho}_{0,n}$ that is compatible with $\varphi$, and these morphisms are compatible with each other (for varying $n$) in the obvious way.

Since $\bV(-\rho)$ is a $G_1T$-module, the ``extra''
$\mathrm{S}(\t)$-action from~\S\ref{ss:completed-Verma} is trivial. 
Since ${V}^{-\rho}_{0,1} \cong \bV(-\rho)$ is indecomposable as a $\sU(\g)$-module, so is ${V}^{-\rho}_{0,n}$, hence the ``extra'' $\mathrm{S}(\t)$-action on this module factors through $\mathrm{S}(\t)/\t^N\mathrm{S}(\t)$ for some $N\geq 1$. 
In this way we obtain a homomorphism of $T$-equivariant $\sU(\g)$-modules 
\[
\varphi_n: \tDelta(-\rho)\otimes_{\sO(\g^{*(1)})\otimes \mathrm{S}(\t)} (R_n\otimes \mathrm{S}(\t)/\t^N\mathrm{S}(\t)) \rightarrow {V}^{-\rho}_{0,n}.
\]
The family $(\varphi_n : n\geq 1)$ yields a homomorphism 
\[
\widetilde{\varphi}:\hDelta(-\rho)\rightarrow V^{-\rho}_0 \otimes_{\sO(\g^{*(1)}_{\ho})} \sO((\g/\n)^{*(1)}_{\ho})
\]
of $T$-equivariant modules over $\sU_{\widehat{0}}^{\widehat{{-\rho}}}\otimes_{\sO(\g^{*(1)}_{\ho})} \sO((\g/\n)_{\ho}^{*(1)})$. 
Since the base change of $\widetilde{\varphi}$ to $R_1$ is surjective, by Nakayama's lemma $\widetilde{\varphi}$ is a surjection. 
Since $V_0^{-\rho} \otimes_{\sO(\g^{*(1)}_{\ho})} \sO((\g/\n)^{*(1)}_{\ho})$ is a projective module over $\sU_{\widehat{0}}^{\widehat{{-\rho}}}\otimes_{\sO(\g^{*(1)}_{\ho})} \sO((\g/\n)_{\ho}^{*(1)})$, this surjection splits. Finally since $\hDelta(-\rho)$ is indecomposable by Lemma \ref{lem:properties-hDelta}\eqref{it:morph-hDelta}, we conclude that $\widetilde{\varphi}$ is an isomorphism. 

\textit{Step 3.}
We are now in a position to prove the isomorphism regarding $\hM_{w_\circ}$. In fact, by construction of $\gamma_{\ho}^{\ho}$ and Lemma~\ref{lem:hDelta-Dmod}
we have 
\begin{multline*}
\hM_{w_\circ}
\cong \sHom_{\tD^{\widehat{0}}_{\widehat{0}}}(\sV^0_0, \tD^{\widehat{0}}_{\widehat{0}}\otimes_{\sO_\sB} \sO_{\{w_\circ\}} \otimes \k_T(-2\rho)) \\
\cong \sHom_{\tD^{\widehat{-\rho}}_{\widehat{0}}}(\sV^{-\rho}_0, \tD^{\widehat{-\rho}}_{\widehat{0}}\otimes_{\sO_\sB} \sO_\sB(-\rho) \otimes_{\sO_\sB} \sO_{\{w_\circ\}} \otimes \k_T(-2\rho)) \\
\cong \sHom_{\tD^{\widehat{-\rho}}_{\widehat{0}}}(\sV_0^{-\rho}, \tD^{\widehat{-\rho}}_{\widehat{0}} \otimes_{\sO_\sB} \sO_{\{w_\circ\}} \otimes \k_T(-{\rho})).
\end{multline*}
%
Since $\tD^{\widehat{-\rho}}_{\widehat{0}} \otimes_{\sO_\sB} \sO_{\{w_\circ\}}$ is supported on the affine subset $\{w_\circ\} \times (\g/\n)^{*(1)}_{\widehat{0}}$ in $\sB\times_{\sB^{(1)}} \widetilde{\g}^{(1)}$, the last term here
is equal to the global sections of the sheaf
\[
\sHom_{\tD^{\widehat{-\rho}}_{\ho} \otimes_{ \sO_{\Groth_{\ho}^{(1)}}}  \sO_{(\g/\n)^{*(1)}_{\ho}}} (\sV^{-\rho}_0 \otimes_{ \sO_{\Groth_{\ho}^{(1)}}}  \sO_{(\g/\n)^{*(1)}_{\ho}}, 
\tD^{\widehat{-\rho}}_{\ho}  \otimes_{\sO_\sB} \sO_{\{w_\circ\}}  \otimes \k_T({-\rho}) ) .
\]
By Lemma~\ref{lem:hDelta-Dmod},
\eqref{eqn:D-pullback-U} and~\eqref{eqn:hDelta-V}, this identifies with
\begin{equation*}
\End_{\sU_{\ho}^{\widehat{{-\rho}}}  \otimes_{\sO(\g_{\ho}^{*(1)})} \sO((\g/\n)^{*(1)}_{\ho})} 
\big( V_0^{-\rho} \otimes_{\sO(\g^{*(1)}_{\ho})} \sO((\g/\n)^{*(1)}_{\ho}) \big). 
\end{equation*}
Finally, by~\eqref{eqn:U-End}, this identifies with
$\sO((\g/\n)^{*(1)}_{\ho})$, which provides the desired isomorphism
$\hM_{w_\circ}\cong \sO_{(\g/\n)^{*(1)}_{\ho}}$. 
All the isomorphisms above respect the $T$-equivariant structures. 

To conclude we prove the description of $\sZ_{w_\circ}$. As a consequence of the previous isomorphism, or using similar arguments, we see that
$\gamma_{\ho}^{\ho}(\bV_{w_\circ}) \cong \sO_{\{(w_\circ^{(1)},0)\}}$. 
By compatibility of the equivalences $\gamma_{\ho}^{\ho}$ and $\gamma_0^{\ho}$ (see Theorem~\ref{thm:localization-fixed-Fr}) and fully faithfulness of the functor $\mod^T(\sU_0^{\ho}) \to \mod^T(\sU_{\ho}^{\ho})$,
we deduce the desired isomorphism $\sZ_{w_\circ} \cong \sO_{\{(w_\circ^{(1)},0)\}}$. 
\end{proof}

\section{Geometric wall-crossing and intertwining functors}
\label{sec:geometric-wall-crossing}

\subsection{Geometric wall-crossing functors -- completed Frobenius character} 
\label{ss:geometric-wall-crossing}

We continue with $\chi$, $H$ as in~\S\ref{ss:localization-comp}.
From now on, in the constructions of Section~\ref{sec:localization} we will mainly consider the case $\lambda=0$, which is regular under our assumption that $p>h$. Recall, for $s \in S_\aff$, the functor $\Theta_s$ considered in~\S\ref{ss:wall-crossing}. In fact, in this subsection only the case $s \in S$ will be considered. In this case, we denote by $P_s^- \subset G$ the parabolic subgroup containing $B^-$ and attached to the subset $\{s\} \subset S$, by $U_s^-$ its unipotent radical, and by $\p_s^-$, $\n_s^-$ their respective Lie algebras. We have an attached parabolic Grothendieck resolution
\[
\Groth_s := G \times^{P_s^-} (\g/\n^-_s)^*,
\]
which is a vector bundle over $\sP_s := G/P_s^-$,
and the canonical morphism $\pi : \Groth \to \g^*$ (see~\S\ref{sss:geometry-G}) factors through a projective morphism $\tpi_s : \Groth \to \Groth_s$. We set
\[
\Groth^{(1)}_{s,\hchi} := \Groth_s^{(1)} \times_{\g^{*(1)}} \g^{*(1)}_{\hchi},
\]
and denote also by $\tpi_s$ the morphism $\Groth^{(1)}_{\hchi} \to \Groth^{(1)}_{s,\hchi}$ induced by $\tpi_s$. Considering the natural action of $H$ on $\Groth_s^{(1)}$ and $\g^{*(1)}$ we can consider the category of $H$-equivariant coherent sheaves on $\Groth^{(1)}_{s,\hchi}$ in the sense of Appendix~\ref{app:equiv-sheaves}. Using Corollary~\ref{cor:derived-pushforward} we have a derived pushforward functor
\[
R(\tpi_s)_* : \Db \Coh^H(\Groth^{(1)}_{\hchi}) \to \Db \Coh^H(\Groth^{(1)}_{s,\hchi}),
\]
and using Corollary~\ref{cor:derived-pullback} we have a derived pullback functor
\[
L(\tpi_s)^* : D^- \Coh^H(\Groth^{(1)}_{s,\hchi}) \to D^- \Coh^H(\Groth^{(1)}_{\hchi}).
\]
This functor preserves bounded derived categories; in fact this can be checked using the factorization $\Groth \hookrightarrow G \times^{B^-} (\g/\n_s)^* \twoheadrightarrow \Groth_s$ of $\tpi_s$, where the second morphism is smooth (hence gives rise to an exact pullback functor) and the first one is the embedding of a subvector bundle, so that the corresponding pullback functor can be described (at least at the level of cohomology) by a Koszul complex construction. By the considerations of~\S\ref{ss:adjointness}, the restriction of $L(\tpi_s)^*$ to bounded derived categories is left adjoint to $R(\tpi_s)_*$. We set
\[
\Xi_s := L(\tpi_s)^* \circ R(\tpi_s)_* : \Db \Coh^H(\Groth^{(1)}_{\hchi}) \to \Db \Coh^H(\Groth^{(1)}_{\hchi}).
\]
By adjointness, there is therefore a canonical morphism $\Xi_s \to \id$. Below we will also consider the similarly defined endofunctors of $\Db \Coh^H(\Groth^{(1)})$ and $\Db \Coh^{H \times \Gm}(\Groth^{(1)})$ (where the $\Gm$-action on $\Groth^{(1)}$ is defined by the same recipe as in~\S\ref{sss:geometry-G}, not involving any Frobenius morphism), which will also be denoted $\Xi_s$. These functors also restrict to $H^{(1)}$-equivariant objects. (This comment applies to all the functors considered in this section, although we will not repeat it.)

The following statement is an equivariant and completed version of~\cite[Lemma in~\S 2.2.5]{BMR06}.

\begin{prop}
\label{prop:wall-crossing}
For any $s \in S$, 
there exists a canonical isomorphism of functors
\[
\gamma_{\hchi}^{\widehat{0}} \circ \Theta_s \cong \Xi_s \circ \gamma_{\hchi}^{\widehat{0}}
\]
which intertwines the morphisms of functors $\Theta_s \to \id$ and $\Xi_s \to \id$ induced by adjunction.
\end{prop}

\begin{proof}[Sketch of proof]
The proof is similar to that of~\cite[Lemma in~\S 2.2.5]{BMR06}. For this one needs to construct a singular variant of the equivalence~\eqref{eqn:localization-comp}, in the form of an equivalence of triangulated categories
\[
\Db(\mod^H(\sU^{\widehat{\mu_s}}_{\widehat{\chi}})) \cong \Db\Coh^H(\Groth^{(1)}_{s,\hchi}).
\]
(The proof is similar to that of~\eqref{eqn:localization-comp}, relying on the results of~\cite{BMR06} rather than~\cite{BMR08}.) This equivalence also depends on a choice of splitting bundle and equivariant structure; we choose them in a way compatible with those for the weight $0$, as in~\cite[\S 2.2.5]{BMR06}. Then we establish an isomorphism relating the functors $\sfT_{\mu_s}^{0}$ and $L(\tpi_s)^*$, as in the proof of~\cite[Lemma in~\S 2.2.5]{BMR06}. Finally, by adjunction we deduce an isomorphism relating the functors $\sfT^{\mu_s}_{0}$ and $R(\tpi_s)_*$, which finishes the proof.
\end{proof}

\subsection{Intertwining functors -- case of \texorpdfstring{$\Groth$}{the Grothendieck resolution}} 
\label{ss:intertwining}

We continue with some fixed $s \in S$. The scheme $\Groth \times_{\Groth_s} \Groth$ is known to be reduced, and to have two irreducible components (both of dimension $\dim(\g)$): one is the diagonal copy $\Delta \Groth$ of $\Groth$, and the other one is a smooth closed subscheme $Z_s$; see~\cite[\S 1.10]{BR12} for details and references. Taking into account the action of $\Gm$ on $\Groth$ considered in~\S\ref{sss:geometry-G}, there exist natural exact sequences of $(G \times \Gm)$-equivariant coherent sheaves
\begin{gather}
\label{eqn:ses1-braid-gp}
\sO_{\Delta \Groth} \langle 2 \rangle \hookrightarrow \sO_{\Groth \times_{\Groth_s} \Groth} \twoheadrightarrow \sO_{Z_s}, \\
\label{eqn:ses2-braid-gp}
\sO_{Z_s}(-\rho,\rho-\alpha_s) \hookrightarrow \sO_{\Groth \times_{\Groth_s} \Groth} \twoheadrightarrow \sO_{\Delta \Groth},
\end{gather}
where we omit pushforward functors associated with closed immersions, and in both cases the surjections are induced by restriction of functions. Here, $\alpha_s$ is the simple root corresponding to $s$, and $\sO_{Z_s}(-\rho,\rho-\alpha_s)$ is the line bundle on $Z_s$ obtained by pullback of the line bundle on $\sB \times \sB$ attached to the pair of weights $(-\rho,\rho-\alpha_s)$. 

Denoting by $p_s^1, p_s^2 : Z_s \to \Groth$ the (restrictions of the) natural projections, and denoting similarly the induced morphisms between Frobenius twists, we consider the functors
\[
\mathbb{I}_s^+, \mathbb{I}_s^- : \Db\Coh^H(\Groth^{(1)}) \to \Db\Coh^H(\Groth^{(1)})
\]
defined by
\begin{gather*}
\mathbb{I}_s^+(\sF) = R(p_s^1)_* ( L(p_s^2)^* \sF), \\
\mathbb{I}_s^-(\sF) = R(p_s^1)_* ( L(p_s^2)^* \sF \otimes_{\sO_{Z_s}} \sO_{Z_s}(-\rho,\rho-\alpha_s)).
\end{gather*}
(In other words, these functors are given by convolution with the sheaves $\sO_{Z_s}$ and $\sO_{Z_s}(-\rho,\rho-\alpha_s)$ respectively. Note that $\sO_{Z_s}(-\rho,\rho-\alpha_s) \cong \sO_{Z_s}(\rho-\alpha_s, -\rho)$ by~\cite[Lemma~1.5.1]{Ric08}, so that the order of the factors is irrelevant.) We will also consider versions of these functors with added $\Gm$-equivariance. 
In this setting the functors $\mathbb{I}_s^+, \mathbb{I}_s^-$ are given by
\begin{gather*}
\mathbb{I}_s^+(\sF) = R(p_s^1)_* ( L(p_s^2)^* \sF) \langle -1 \rangle, \\
\mathbb{I}_s^-(\sF) = R(p_s^1)_* ( L(p_s^2)^* \sF \otimes_{\sO_{Z_s}} \sO_{Z_s}(-\rho,\rho-\alpha_s)) \langle -1 \rangle.
\end{gather*}

It is proved in~\cite[\S 1]{BR12} that these functors are quasi-inverse equivalences of categories, and that they generate a right action of $\Br_\ex$, in the sense that there exists a group morphism from $(\Br_\ex)^{\op}$ to the group of isomorphism classes of autoequivalences of $\Db\Coh^H(\Groth^{(1)})$ or $\Db\Coh^{H \times \Gm}(\Groth^{(1)})$, denoted $b \mapsto \mathbb{I}_b$, such that for $s \in S$ we have $\mathbb{I}_{\rH_s} = \mathbb{I}_s^+$ (hence also $\mathbb{I}_{(\rH_s)^{-1}} = \mathbb{I}_s^-$), and for $\lambda \in \X$ the equivalence $\mathbb{I}_{\theta_\lambda}$ is given by tensor product with the line bundle $\sO_{\Groth^{(1)}}(\lambda)$ obtained by pullback from the line bundle on $\sB^{(1)}$ associated with $\lambda$. (See~\cite[Remark~6.3]{br-two} for a discussion of our choice of convention.)


Using the base change theorem as in~\cite[Proposition~5.2.2]{Ric08}, one sees that the functor $\Xi_s$ identifies with convolution with the kernel $\sO_{\Groth \times_{\Groth_s} \Groth}$. As a consequence,
for any $\sF$ in $\Db \Coh^{H \times \Gm}(\Groth^{(1)})$ there exist functorial distinguished triangles
\begin{gather}
\label{eqn:triangle-functors-braid-gp-1}
\sF \langle 1 \rangle \to \Xi_s(\sF) \langle -1 \rangle \to \mathbb{I}_s^+(\sF) \xrightarrow{[1]} \\
\label{eqn:triangle-functors-braid-gp-2}
\mathbb{I}_s^-(\sF) \to \Xi_s(\sF) \langle -1 \rangle \to \sF \langle -1 \rangle \xrightarrow{[1]}
\end{gather}
induced by the exact sequences~\eqref{eqn:ses1-braid-gp}--\eqref{eqn:ses2-braid-gp}. Here, in~\eqref{eqn:triangle-functors-braid-gp-2} the second morphism is induced by adjunction. We also have similar triangles in the non-$\Gm$-equivariant setting, for any $\sF$ in $\Db \Coh^{H}(\Groth^{(1)})$ (with no grading shift).


Considering the fiber product of $Z_s^{(1)}$ with $\g^{*(1)}_{\hchi}$ and using the same considerations as in~\cite[\S 4--5]{BR12} one constructs similarly endofunctors of $\Db\Coh^H(\Groth^{(1)}_{\hchi})$, which will also be denoted $\mathbb{I}_s^{+}$ and $\mathbb{I}_s^{-}$,
such that for any $\sF$ in $\Db\Coh^H(\Groth^{(1)}_{\hchi})$ we have canonical distinguished triangles
\[
\sF \to \Xi_s(\sF) \to \mathbb{I}_s^{+}(\sF) \xrightarrow{[1]}, \qquad
\mathbb{I}_s^{-}(\sF) \to \Xi_s(\sF) \to \sF \xrightarrow{[1]},
\]
where in the second case the second morphism is induced by adjunction.
Moreover, the pullback functor $(-)_{\hchi}$ considered in~\S\ref{ss:localization-comp}
is compatible in the natural way with the two versions of the functors $\mathbb{I}_s^{\pm}$.
Again by the considerations in~\cite[\S 5]{BR12}, we also have a right $\Br_\ex$-action in this setting: there exists a group morphism from $(\Br_\ex)^{\op}$ to the group of isomorphism classes of autoequivalences of $\Db\Coh^H(\Groth_{\hchi}^{(1)})$, again denoted $b \mapsto \mathbb{I}_b$, such that for $s \in S$ we have $\mathbb{I}_{\rH_s} = \mathbb{I}_s^+$ (hence also $\mathbb{I}_{(\rH_s)^{-1}} = \mathbb{I}_s^-$), and for $\lambda \in \X$ the equivalence $\mathbb{I}_{\theta_\lambda}$ is given by tensor product with the line bundle $\sO_{\Groth_{\hchi}^{(1)}}(\lambda)$ obtained by pullback from $\sO_{\Groth^{(1)}}(\lambda)$.

Now, recall the functor $\mathbb{S}_s^-$ constructed in~\S\ref{ss:wall-crossing}.
The following statement follows from Proposition~\ref{prop:wall-crossing} 
and standard properties of triangulated categories.

\begin{corollary}
\label{cor:compatibility-bbI-bbS}
Let $s \in S$. For any $M$ in $\Db(\mod^H(\sU^{\ho}_{\widehat{\chi}}))$ there exists an isomorphism
\[
\mathbb{I}^{-}_s \circ \gamma_{\hchi}^{\widehat{0}}(M) \cong \gamma_{\hchi}^{\widehat{0}} \circ \mathbb{S}_s^-(M).
\]
\end{corollary}

\begin{rmk}
In the nonequivariant and noncompleted setting, it is proved in~\cite[Theorem~5.4.1]{Ric08} that the isomorphism similar to that in Corollary~\ref{cor:compatibility-bbI-bbS} can be promoted to an isomorphism of functors. Since this is not needed for us here, we will not consider this question in the present setting. A similar comment applies to Proposition~\ref{prop:comparison-Br-actions} below.
\end{rmk}

\subsection{Comparison of the braid group actions}

We continue with $\chi$, $H$ as above.
Let us denote by $\imath$ the involutive anti-auto\-morphism of $\Br_\ex$ defined by $\imath(\rH_w) = \rH_{w^{-1}}$ for any $w \in W_\ex$.
The goal of this subsection is to explain the proof of the following statement.

\begin{prop}
\label{prop:comparison-Br-actions}
Let $b \in \Br_\ex$. For any $M$ in $\Db(\mod^H(\sU^{\ho}_{\hchi}))$ there exists an isomorphism
\[
\mathbb{I}_{\imath(b)} \circ \gamma_{\hchi}^{\widehat{0}}(M) \cong \gamma_{\hchi}^{\widehat{0}} \circ \mathbb{S}_{b}(M).
\]
\end{prop}

This statement is an analogue in our present setting of~\cite[Theorem~5.4.1]{Ric08}, and the proof will be similar. Namely, it clearly suffices to prove the proposition for $b$ running over a set of generators of $\Br_\ex$. We will choose the set $\{(\rH_s)^{-1} : s \in S\} \cup \{\rH_{t_{-\nu}} : \nu \in \X^+\}$. When $b=(\rH_s)^{-1}$, the claim has been established in Corollary~\ref{cor:compatibility-bbI-bbS}. Hence, to conclude the proof it suffices to prove the following claim.

\begin{prop}
\label{prop:comparison-actions-dom}
Let $\nu \in \X^+$.
For any $M$ in $\Db(\mod^H(\sU^{\ho}_{\hchi}))$ there exists an isomorphism
\[
\gamma_{\hchi}^{\widehat{0}} \circ \mathbb{S}_{\rH_{t_{-\nu}}}(M) \cong \gamma_{\hchi}^{\widehat{0}}(M) \otimes_{\sO_{\Groth_{\hchi}^{(1)}}} \sO_{\Groth_{\hchi}^{(1)}}(\nu).
\]
\end{prop}

This proposition is an analogue in our present setting of~\cite[Proposition~2.3.3]{BMR06}. Our proof will be similar, except that we explain a detail in the proof of~\cite[Lemma~2.2.3]{BMR06} which was unclear to us.

Since the weight $0 \in \X$ is regular, the map $w \mapsto w \bullet 0$ induces a bijection from $W_\ex$ to $W_\ex \bullet 0 \subset \X$. We can therefore consider the right action $*$ of $W_\ex$ on $W_\ex \bullet 0$ such that for $y,w \in W_\ex$ we have $(y \bullet 0) * w = yw \bullet 0$. For any $\lambda \in W_\ex \bullet 0$ we consider the equivalence
\[
R\Gamma_\lambda : \Db(\mod^H(\tD^{\hla}_{\hchi})) \simto \Db(\mod^H(\sU^{\hla}_{\hchi}))
\]
established in the course of the proof of Theorem~\ref{thm:localization-comp}. Note that the codomain of this functor does not depend on $\lambda$: it always coincides with $\Db(\mod^H(\sU^{\ho}_{\hchi}))$. (On the other hand the domain does depend on $\lambda$, or more specifically on its image $\overline{\lambda}$.) Let us note also that if $\lambda,\mu \in W_\ex \bullet 0$ there is an equivalence of categories
\[
\mathrm{Tens}_{\lambda}^{\mu} : \Db(\mod^H(\tD^{\hla}_{\hchi})) \simto \Db(\mod^H(\tD^{\hmu}_{\hchi}))
\]
induced by the tensor product operation $\sO_{\sB}(\mu-\lambda) \otimes_{\sO_{\sB}} (-)$. More generally the functor $R\Gamma_\lambda$ is defined for any $\lambda \in \X$ (although this is not an equivalence in general), and the equivalence $\mathrm{Tens}_{\lambda}^{\mu}$ is defined for any $\lambda,\mu \in \X$.

The main ingredient of the proof of Proposition~\ref{prop:comparison-actions-dom} is the following claim.

\begin{lem}
\label{lem:localization-braid-action}
Let $\lambda \in W_\ex \bullet 0$ and $w \in W_\ex$. Assume that one can write $w = s_1 \cdots s_r \omega$ with $s_1, \dots, s_r \in S_\aff$, $\omega \in \Omega$ and $r = \ell(w)$, in such a way that for any $i \in \{0, \dots, r-1\}$ we have $\lambda * (s_1 \cdots s_i s_{i+1}) \prec \lambda * (s_1 \cdots s_i)$. Then for any $\sF \in \Db(\mod^H(\tD^{\hla}_{\hchi}))$ there exists an isomorphism
\[
\mathbb{S}_{(\rH_w)^{-1}} \circ R\Gamma_\lambda(\sF) \cong R\Gamma_{\lambda * w} (\mathrm{Tens}_\lambda^{\lambda * w}(\sF)).
\]
\end{lem}

\begin{proof}
Of course it suffices to prove the claim when $r =0$, and when $r=1$ and $\omega=1$. The proof in these cases will be based on the following observations.
\begin{enumerate}
\item
\label{it:translation-Dmod-1}
Let $y \in W_\ex$, and let $\mu \in \mathscr{A}_0$. Then for $\sF$ in $\Db(\mod^H(\tD^{\widehat{y \bullet 0}}_{\hchi}))$ we have a functorial isomorphism
\[
\sfT_0^\mu(R\Gamma_{y \bullet 0}(\sF)) \cong R\Gamma_{y \bullet \mu}(\mathrm{Tens}_{y \bullet 0}^{y \bullet \mu} ( \sF)).
\]
\item
\label{it:translation-Dmod-2}
Let $y \in W_\ex$ and $s \in S_\aff$. If $ys \bullet 0 \prec y \bullet 0$, then for $\sG$ in $\Db(\mod^H(\tD^{\widehat{y \bullet \mu_s}}_{\hchi}))$ there exists a functorial distinguished triangle
\[
R\Gamma_{ys \bullet 0}( \mathrm{Tens}_{y \bullet \mu_s}^{ys \bullet 0}(\sG) ) \to \sfT_{\mu_s}^0(R\Gamma_{y \bullet \mu_s}(\sG)) \to R\Gamma_{y \bullet 0}(\mathrm{Tens}_{y \bullet \mu_s}^{y \bullet 0}(\sG)) \xrightarrow{[1]}.
\] 
\end{enumerate}
The proof of these claims follows a standard pattern (already encountered in the proof of Proposition~\ref{prop:translation-hDelta})
and can be copied from~\cite[Lemma~6.1.2]{BMR08} (see also~\cite[Lemma~2.2.3]{BMR06}, \cite[Lemma~5.5]{br-Hecke} or~\cite[Lemma~7.1]{br-two}).

Using~\eqref{it:translation-Dmod-1} with $\mu=\omega \bullet 0$ (where $\omega \in \Omega$) we obtain the desired statement when $r=0$. (In this case we have $\mathbb{S}_{(\rH_\omega)^{-1}} = \mathbb{S}_{\rH_{\omega^{-1}}} = \sfT_0^{\omega \bullet 0}$, see~\S\ref{ss:intertwining-RT}.)

Now we consider the case $r=1$ and $\omega=1$, and write $\lambda = y \bullet 0$ ($y \in W_\ex$), $s=s_1$. Combining~\eqref{it:translation-Dmod-1} for $\mu=\mu_s$ and~\eqref{it:translation-Dmod-2}, we obtain for $\sF$ in $\Db(\mod^H(\tD^{\widehat{y \bullet 0}}_{\hchi}))$ a distinguished triangle
\[
R\Gamma_{ys \bullet 0}( \mathrm{Tens}_{y \bullet 0}^{ys \bullet 0}(\sF) ) \to \Theta_s(R\Gamma_{y \bullet 0}(\sF)) \to R\Gamma_{y \bullet 0}(\sF) \xrightarrow{[1]}.
\]
To conclude, it therefore suffices to show that, in this triangle, the second morphism coincides with the map induced by our morphism of functors $\Theta_s \to \id$, up to an automorphism of $R\Gamma_{y \bullet 0}(\sF)$.

This can be justified as follows. One can consider $\sF$ as a sheaf of $\sU(\g)$-modules on $\sB$, and hence apply translation functors to it. From this point of view, by construction the triangle above is obtained by taking (derived) global sections of a functorial distinguished triangle
\[
\mathrm{Tens}_{y \bullet 0}^{ys \bullet 0}(\sF) \to \Theta_s(\sF) \to \sF \xrightarrow{[1]},
\]
and it suffices to prove the similar claim for the morphism $\Theta_s(\sF) \to \sF$ here. 
In fact we will prove that this morphism is the composition of the map induced by our morphism $\Theta_s \to \id$ with the action of an invertible element of the completion $\sO(\t^*)_{\widehat{y \bullet 0}}$ of $\sO(\t^*)$ with respect to the ideal determined by $y \bullet 0$ (independent of $\sF$). 

Let us first consider the case $\sF=\tD^{\widehat{y \bullet 0}}_{\hchi}$. 
Note that, for $\lambda\in \X$, the sheaf $\tD^{\hla}_{\hchi}$ is the the base change of $\tD^{\hla}$ from $\sO(\g^{*(1)})$ to $\sO(\g_{\hchi}^{*(1)})$, where $\tD^{\hla}$ is the pullback of $\tD$ to 
\[
(\Groth^{(1)}\times_{\t^{*(1)}} \t^*) \times_{\t^*} \Spec(\sO(\t^*)_{\hla}),
\]
and as above $\sO(\t^*)_{\hla}$ is the completion of $\sO(\t^*)$ with respect to the ideal corresponding to $\overline{\lambda}\in \t^*$. 
This reduces this proof of our claim to that of the similar claim for the similarly defined morphism $\Theta_s(\tD^{\widehat{y \bullet 0}}) \to \tD^{\widehat{y \bullet 0}}$.
Since this is a morphism of $G$-equivariant quasi-coherent sheaves, it suffices to analyze the morphism after tensor product on the right with $\sO_{\{w_\circ\}}$. 
By the proof of Lemma~\ref{lem:hDelta-Dmod}, the sheaf under consideration is $\widetilde{\Delta}(w_\circ y \bullet 0)\otimes_{\mathrm{S}(\t)} \mathrm{S}(\t)_{\ho}$ concentrated at $\{w_\circ\}$ (up to $T$-shift), where $\mathrm{S}(\t)$ acts on $\widetilde{\Delta}(w_\circ y \bullet 0)$ via the ``extra" $\odot$-action in \S\ref{ss:completed-Verma}. 
Here, since $ys \bullet 0 \prec y \bullet 0$ we have $w_\circ y \bullet 0 \prec w_\circ ys \bullet 0$, and the desired claim for $\widetilde{\Delta}(w_\circ y \bullet 0)\otimes_{\mathrm{S}(\t)} \mathrm{S}(\t)_{\ho}$ can be proved similarly as the one for $\widehat{\Delta}(w_\circ y \bullet 0)$ that has been established in the course of the proof of Proposition~\ref{prop:translation-hDelta}.


For a general bounded complex $\sF$, we observe that there exists a bounded complex $\sG$ of coherent sheaves on $\sB$ and a surjection of complexes $\tD^{\widehat{y \bullet 0}}_{\hchi} \otimes_{\sO_{\sB}} \sG \twoheadrightarrow \sF$; since all the morphisms considered here are functorial, and all the functors exact, the general case therefore follows from the special case treated above.
\end{proof}

\begin{proof}[Proof of Proposition~\ref{prop:comparison-actions-dom}]
We apply Lemma~\ref{lem:localization-braid-action} for $\lambda=0$ and $w=t_{-\nu}$. It is well known that the assumption of the lemma is satisfied with these data, and moreover we have $0 * t_{-\nu} = -p\nu$. Using the fact that $\sO_{\sB}(-p\nu) = (\mathrm{Fr}_{\sB})^* \sO_{\sB^{(1)}}(-\nu)$, we deduce for any $\sG$ in $\Db \Coh^T(\Groth^{(1)}_{\hchi})$ an isomorphism
\[
\mathbb{S}_{(\rH_{t_{-\nu}})^{-1}} \circ (\gamma_{\hchi}^{\ho})^{-1}(\sG) \cong (\gamma_{\hchi}^{\ho})^{-1}(\sG \otimes_{\sO_{\Groth^{(1)}_{\hchi}}} \sO_{\Groth^{(1)}_{\hchi}}(-\nu)),
\]
i.e.~an isomorphism
\[
\sG \cong \gamma_{\hchi}^{\ho} \circ \mathbb{S}_{\rH_{t_{-\nu}}} \circ (\gamma_{\hchi}^{\ho})^{-1}(\sG \otimes_{\sO_{\Groth^{(1)}_{\hchi}}} \sO_{\Groth^{(1)}_{\hchi}}(-\nu)).
\]
Setting $\sG = \gamma_{\hchi}^{\widehat{0}}(M) \otimes_{\sO_{\Groth_{\hchi}^{(1)}}} \sO_{\Groth_{\hchi}^{(1)}}(\nu)$ we obtain the desired isomorphism.
\end{proof}

\subsection{Geometric wall-crossing functors for affine simple reflections} 
\label{ss:geometric-wall-crossing-aff}

Recall from~\S\ref{ss:intertwining-RT} that for any $s \in S_\aff \smallsetminus S$, we have fixed $b \in \Br_\ex$ and $t \in S$ such that $b \rH_t b^{-1} = \rH_s$ in $\Br_\ex$. We set
\[
\Xi_s := \mathbb{I}_{\imath(b)} \circ \Xi_t \circ \mathbb{I}_{\imath(b)^{-1}} : \Db \Coh^H(\Groth^{(1)}_{\hchi}) \to \Db \Coh^H(\Groth^{(1)}_{\hchi}).
\]
Then as above, for any $\sF$ in $\Db\Coh^H(\Groth^{(1)}_{\hchi})$ we have canonical distinguished triangles
\[
\sF \to \Xi_s(\sF) \to \mathbb{I}_{\rH_s}(\sF) \xrightarrow{[1]}, \qquad
\mathbb{I}_{(\rH_s)^{-1}}(\sF) \to \Xi_s(\sF) \to \sF \xrightarrow{[1]}.
\]
We will use similar notation for the corresponding endofunctor of $\Db \Coh^H(\Groth^{(1)})$ or $\Db \Coh^{H \times \Gm}(\Groth^{(1)})$.

The following statement is an immediate consequence of Proposition~\ref{prop:wall-crossing}, Proposition~\ref{prop:comparison-Br-actions} and Lemma~\ref{lem:conjugation-wall-crossing}.

\begin{prop}
\label{prop:wall-crossing-aff}
For any $s \in S_\aff$ and $M \in \Db(\mod^H(\sU_{\hchi}^{\ho}))$, 
there exists an isomorphism
\[
\gamma_{\hchi}^{\widehat{0}} \circ \Theta_s(M) \cong \Xi_s \circ \gamma_{\hchi}^{\widehat{0}}(M).
\]
\end{prop}

\subsection{Intertwining functors -- case of \texorpdfstring{$\Spr$}{the Springer resolution}} 
\label{ss:intertwining-Springer}

Consider again some $s \in S$.
Replacing, in the constructions of~\S\ref{ss:intertwining}, $Z_s$ by  $Z_s':=Z_s \cap (\Spr \times \Spr)$,
 one defines similarly functors
\[
\mathbb{J}^+_s, \mathbb{J}^-_s : \Db\Coh^H(\Spr^{(1)}) \to \Db\Coh^H(\Spr^{(1)}).
\]
We will also consider versions of these functors in the $\Gm$-equivariant setting; here, given our conventions in~\S\ref{sss:geometry-G}, in the definition of these functors we incorporate a shift $\langle 1 \rangle$ rather than $\langle -1 \rangle$. As above, by~\cite[\S 1]{BR12} these functors are quasi-inverse equivalences of categories, and they generate a right action of $\Br_\ex$, in the sense that there exists a group morphism from $(\Br_\ex)^{\op}$ to the group of isomorphism classes of autoequivalences of $\Db\Coh^H(\Spr^{(1)})$ or $\Db\Coh^{H \times \Gm}(\Spr^{(1)})$, denoted $b \mapsto \mathbb{J}_b$, such that for $s \in S$ we have $\mathbb{J}_{\rH_s} = \mathbb{J}_s^+$ (hence also $\mathbb{J}_{(\rH_s)^{-1}} = \mathbb{J}_s^-$), and for $\lambda \in \X$ the equivalence $\mathbb{J}_{\theta_\lambda}$ is given by tensor product with the line bundle $\sO_{\Spr^{(1)}}(\lambda)$ obtained by pullback from the line bundle on $\sB^{(1)}$ associated with $\lambda$.

Considering the map $i$ introduced in~\S\ref{ss:localization-comp}, in the non-$\Gm$-equivariant setting, for any $s \in S$
there exist canonical isomorphisms of functors
\[
\mathbb{I}_s^{\pm} \circ i_* \cong i_* \circ \mathbb{J}_s^{\pm}, \quad \mathbb{J}_s^{\pm} \circ Li^* \cong Li^* \circ \mathbb{I}_s^{\pm},
\]
see~\cite[\S 1.6]{BR12}. (Here we do not consider $\Gm$-equivariance because the map $i$ is not $\Gm$-equivariant with our conventions.) We have similar isomorphisms for the functors of tensoring with line bundles; hence, more generally, for any $b \in \Br_\ex$ we have
\[
\mathbb{I}_b \circ i_* \cong i_* \circ \mathbb{J}_b, \quad \mathbb{J}_b \circ Li^* \cong Li^* \circ \mathbb{I}_b.
\]

\subsection{Semisimple functors} 
\label{ss:ss-functors}

Recall that $\Spr = G \times^{B^-} (\g/\b^-)^*$. For any $s \in S$ we also have a parabolic version of this variety, namely
\[
\Spr_s := G \times^{P_s^-} (\g/\p_s^-)^*,
\]
and a natural closed immersion $\jmath_s : \sB \times_{\sP_s} \Spr_s \hookrightarrow \Spr$. We set
\[
Y_s := (\sB \times_{\sP_s} \Spr_s) \times_{\Spr_s} (\sB \times_{\sP_s} \Spr_s) = \sB \times_{\sP_s} \Spr_s \times_{\sP_s} \sB,
\]
a $(G \times \Gm)$-stable closed subscheme of $\Spr \times \Spr$. (In fact, $Y_s$ is a smooth variety, and an irreducible component of $Z'_s$.) We have canonical short exact sequences of $(G \times \Gm)$-equivariant coherent sheaves
\begin{gather}
\label{eqn:ses3-braid-gp}
\sO_{\Delta \Spr} \langle -2 \rangle \hookrightarrow \sO_{Z'_s}(-\rho,\rho-\alpha_s) \twoheadrightarrow \sO_{Y_s}(-\rho,\rho-\alpha_s), \\
\label{eqn:ses4-braid-gp}
\sO_{Y_s}(-\rho,\rho-\alpha_s) \hookrightarrow \sO_{Z_s'} \twoheadrightarrow \sO_{\Delta \Spr}
\end{gather}
where we use the same notational conventions as above; see~\cite[\S 6.1]{Ric08} and~\cite[\S 5.3]{Ric10} for details and references.

Denoting by $q^1_s, q^2_s : (Y_s)^{(1)} \to \Spr^{(1)}$ the two natural projections, we consider the functor
\begin{multline*}
\Upsilon_s := R(q^1_s)_*(L(q^2_s)^* (-) \otimes_{\sO_{Y_s}} \sO_{Y_s}(-\rho,\rho-\alpha_s)) \langle 1 \rangle : \\
\Db\Coh^{H \times \Gm}(\Spr^{(1)}) \to \Db\Coh^{H \times \Gm}(\Spr^{(1)}).
\end{multline*}
In other words, $\Upsilon_s$ is convolution with the kernel $\sO_{(Y_s)^{(1)}}(-\rho, \rho-\alpha_s)\langle 1 \rangle$. We will also consider the similar functor in the non-$\Gm$-equivariant setting, for which we will use similar notation. Because of the exact sequences~\eqref{eqn:ses3-braid-gp}--\eqref{eqn:ses4-braid-gp}, we have for any $\sF$ in $\Db\Coh^{H \times \Gm}(\Spr^{(1)})$ functorial distinguished triangles
\begin{equation}
\label{eqn:triangles-Upsilon-J}
\sF \langle -1 \rangle \to \mathbb{J}_{(\rH_s)^{-1}}(\sF) \to \Upsilon_s(\sF) \xrightarrow{[1]}, \quad
\Upsilon_s(\sF) \to \mathbb{J}_{\rH_s}(\sF) \to \sF \langle 1 \rangle \xrightarrow{[1]}.
\end{equation}
We also have similar triangles in the non-$\Gm$-equivariant setting.

Below we will also use an alternative interpretation of the functors $\Upsilon_s$, as follows. Recall the closed immersion $\jmath_s$ considered above, and consider also the (flat and projective) projection $\varrho_s : \sB \times_{\sP_s} \Spr_s \to \Spr_s$. We have derived push/pull functors associated with these morphisms. We also set $(\jmath_s)_\dag(-) := (\jmath_s)_*(-) \otimes_{\sO_{\Spr^{(1)}}} \sO_{\Spr^{(1)}}(-\alpha_s)[-1] \langle 2 \rangle$. Using Grothendieck--Serre duality, one sees that $(\jmath_s)_\dag$ is left adjoint to $L(\jmath_s)^*$. Using definitions and the flat base change theorem, one sees that there exists a canonical isomorphism of functors
\begin{multline}
\label{eqn:isom-Upsilon}
(\jmath_s)_\dag \circ L(\varrho_s)^* \circ R(\varrho_s)_* \circ L(\jmath_s)^*(-) \cong \\
\sO_{\Spr^{(1)}}(-\rho) \otimes_{\sO_{\Spr^{(1)}}} \Upsilon_s \left( \sO_{\Spr^{(1)}}(\rho) \otimes_{\sO_{\Spr^{(1)}}} (-) \right) [-1]\langle 1 \rangle.
\end{multline}

\begin{lem}
\label{lem:morph-adjunction-kernels}
Under the identification~\eqref{eqn:isom-Upsilon}, the adjunction morphism $(\jmath_s)_\dag \circ L(\varrho_s)^* \circ R(\varrho_s)_* \circ L(\jmath_s)^* \to \id$ is induced by the morphism $\sO_{(Y_s)^{(1)}}(-\rho,\rho-\alpha_s) \to \sO_{\Delta \Spr^{(1)}} \langle -2 \rangle [1]$ appearing in~\eqref{eqn:ses3-braid-gp}.
\end{lem}

\begin{proof}
The adjunction morphism under consideration is a composition
\[
(\jmath_s)_\dag \circ L(\varrho_s)^* \circ R(\varrho_s)_* \circ L(\jmath_s)^* \to (\jmath_s)_\dag \circ L(\jmath_s)^* \to \id
\]
where both maps are induced by adjunction. As explained above the first of these functors is given by convolution with the kernel $\sO_{(Y_s)^{(1)}}(-\alpha_s,0) \langle 2 \rangle [-1]$. Similar considerations show that the functor $(\jmath_s)_\dag \circ L(\jmath_s)^*$ is given by convolution with the kernel $\sO_{\Delta (\sB \times_{\sP_s} \Spr_s)^{(1)}}(-\alpha_s) \langle 2 \rangle [-1]$, where $\Delta (\sB \times_{\sP_s} \Spr_s)^{(1)} \subset \Spr^{(1)} \times \Spr^{(1)}$ is the diagonal copy of $(\sB \times_{\sP_s} \Spr_s)^{(1)}$, and moreover that the first morphism above is induced by the morphism $\sO_{(Y_s)^{(1)}} \to \sO_{\Delta (\sB \times_{\sP_s} \Spr_s)^{(1)}}$ coming from the fact that $\Delta (\sB \times_{\sP_s} \Spr_s)^{(1)} \subset (Y_s)^{(1)}$ (see~\cite[Lemma~1.2.2]{Ric08}).

Now using a Koszul resolution one sees that $\sO_{\Delta (\sB \times_{\sP_s} \Spr_s)^{(1)}}$ is quasi-isomorphic to a complex $(\cdots \to \sO_{\Delta \Spr^{(1)}}(\alpha_s) \langle -2 \rangle \to \sO_{\Delta \Spr^{(1)}} \to \cdots)$ where $\sO_{\Delta \Spr^{(1)}}$ is in degree $0$, in such a way that the adjunction morphism above is induced by the canonical morphism of complexes
\[
(\cdots \to \sO_{\Delta \Spr^{(1)}} \to \sO_{\Delta \Spr^{(1)}} (-\alpha_s) \langle 2 \rangle \to \cdots)[-1] \to \sO_{\Delta \Spr^{(1)}}.
\]
One checks that the composition of the morphisms considered here is induced by the morphism appearing in~\eqref{eqn:ses3-braid-gp}, which concludes the proof.
\end{proof}

For $s \in S_\aff \smallsetminus S$, with our fixed $b \in \Br_\ex$ and $t \in S$ such that $b \rH_t b^{-1} = \rH_s$ (see~\S\ref{ss:intertwining-RT}), we set
\[
\Upsilon_s := \mathbb{J}_{\imath(b)} \circ \Upsilon_t \circ \mathbb{J}_{\imath(b)^{-1}}.
\]
With this definition, we again have triangles similar to those in~\eqref{eqn:triangles-Upsilon-J}, now for any $s \in S_\aff$.

\subsection{Geometric wall-crossing and intertwining functors -- fixed Frobenius character} 
\label{ss:wall-crossing-fixed-Fr}

For $s \in S$,
replacing in the construction of~\S\ref{ss:localization-fixed} the scheme $\Groth$ by $\Groth_s$, one makes sense of the category
\[
\DGCoh^H(\Groth_s^{(1)} \times^R_{\g^{*(1)}} \{\chi\}).
\]
The morphism $\tpi_s : \Groth \to \Groth_s$ induces a morphism between the dg-ringed spaces involved in the definitions of $\DGCoh^H(\Groth^{(1)} \times_{\g^{*(1)}} \{\chi\})$ and $\DGCoh^H(\Groth_s^{(1)} \times_{\g^{*(1)}} \{\chi\})$, and we have associated (adjoint) derived pushforward and pullback functors
\[
R(\hpi_s)_* : \DGCoh^H(\Groth^{(1)} \times^R_{\g^{*(1)}} \{\chi\}) \leftrightarrows
\DGCoh^H(\Groth_s^{(1)} \times^R_{\g^{*(1)}} \{\chi\}) : L(\hpi_s)^*.
\]
We set
\[
\Xi_s^\dg := L(\hpi_s)^* \circ R(\hpi_s)_* : \DGCoh^H(\Groth^{(1)} \times^R_{\g^{*(1)}} \{\chi\}) \to \DGCoh^H(\Groth^{(1)} \times^R_{\g^{*(1)}} \{\chi\}).
\]
Then the following diagram commutes:
\begin{equation}
\label{eqn:diagram-Xi-dg}
\begin{tikzcd}
\DGCoh^H(\Groth^{(1)} \times^R_{\g^{*(1)}} \{\chi\}) \arrow[r,"\Xi_s^\dg"] \arrow[d, "\eqref{eqn:forget-dg-completion}"'] & \DGCoh^H(\widetilde{\g}^{(1)}\times^R_{{\g}^{*(1)}} \{\chi\}) \arrow[d, "\eqref{eqn:forget-dg-completion}"] \\ 
\Db\Coh^H(\Groth^{(1)}_{\hchi}) \arrow[r,"\Xi_s"] & \Db\Coh^H(\Groth^{(1)}_{\hchi}).
\end{tikzcd}
\end{equation}
In case $\chi=0$ and $H=T$, we also have a version of this construction including $\Gm$-equivariance for the action induced by the action on $\Groth^{(1)}$; we will use similar notation in this setting.

Using the same considerations as in~\cite[\S 4--5]{BR12} one constructs also a right braid group action in this setting, i.e.~a group morphism from $(\Br_\ex)^\op$ to the group of isomorphism classes of autoequivalences of $\DGCoh^H(\Groth^{(1)} \times^R_{\g^{*(1)}} \{\chi\})$, denoted $b \mapsto \mathbb{I}^\dg_b$, which is compatible with the actions on $\Db\Coh^H(\Groth^{(1)})$ and $\Db\Coh^H(\Groth_{\hchi}^{(1)})$ in the obvious way. Arguing as in~\cite[\S 6.2]{Ric10} one sees that for any $s \in S$ and $\sF$ in $\DGCoh^H(\Groth^{(1)} \times^R_{\g^{*(1)}} \{\chi\})$ there are functorial distinguished triangles
\begin{equation}
\label{eqn:triangles-Xi-L}
\sF \to \Xi_s^{\dg}(\sF) \to \mathbb{I}^\dg_{\rH_s}(\sF) \xrightarrow{[1]}, \quad \mathbb{I}^\dg_{(\rH_s)^{-1}}(\sF) \to \Xi_s^{\dg}(\sF) \to \sF \xrightarrow{[1]}.
\end{equation}

For $s \in S_\aff \smallsetminus S$, with our fixed $b \in \Br_\ex$ and $t \in S$ such that $b \rH_t b^{-1} = \rH_s$ (see once again~\S\ref{ss:intertwining-RT}), we set
\[
\Xi^\dg_s := \mathbb{I}^\dg_{\imath(b)} \circ \Xi^\dg_t \circ \mathbb{I}^\dg_{\imath(b)^{-1}}.
\]
Then the diagram as in~\eqref{eqn:diagram-Xi-dg}, now for $s \in S_\aff$, commutes, and we also have triangles as in~\eqref{eqn:triangles-Xi-L} for any $s \in S_\aff$. All these constructions also have versions where we add $\Gm$-equivariance, and we will use the same notation for these variants.



\subsection{Geometric description of standard objects}

In this subsection we assume that $\chi=0$ and $H=T$.
We conclude this section by explaining a description of the objects $\hM_x$ and $\sZ_x$ which generalizes Proposition~\ref{prop:M-Z-w0}.

\begin{corollary}
\label{cor:M-Z}
Let $x \in W_\ex$, and write $x=t_\lambda w_\circ w$ with $\lambda\in \X$ and $w\in W$. 
Then we have isomorphisms 
\[
\hM_x \cong \mathbb{I}_{(\rH_{w^{-1}})^{-1}} (\sO_{(\g/\n)^{*(1)}_{\widehat{0}}}) \langle {p\lambda} \rangle, \qquad
\sZ_{x} \cong \mathbb{I}^\dg_{(\rH_{w^{-1}})^{-1}} (\sO_{\{(w_\circ^{(1)},0)\}}) \langle {p\lambda} \rangle.
\]
\end{corollary}

\begin{proof}
Let us first consider the case of the objects $\hM_x$.
The proof is easily reduced to the case $\lambda=0$, so we will assume that $x \in W$.
The case $x=w_\circ$ (i.e.~$w=e$) is the first case in Proposition~\ref{prop:M-Z-w0}. The general case is proved by induction on $w$. Namely, let $w \in W \smallsetminus \{e\}$, and write $w=w's$ with $s \in S$ and $\ell(w')=\ell(w)-1$.
Then $w_\circ w' > w_\circ w$ in the Bruhat order, so $w_\circ w' \bullet 0 \prec w_\circ w \bullet 0$. By Proposition \ref{prop:translation-hDelta}\eqref{it:translation-hDelta-2}, we therefore have
\[
\hDelta_{w_\circ w} \cong \mathbb{S}_s^-(\hDelta_{w_\circ w'}).
\]
Hence
$\hM_{w_\circ w} \cong \gamma_{\ho}^{\ho} \circ \mathbb{S}_s^-(\hDelta_{w_\circ w'})$, which by Corollary~\ref{cor:compatibility-bbI-bbS} implies that
\[
\hM_{w_\circ w} \cong \mathbb{I}^{-}_s(\hM_{w_\circ w'}).
\]
Since $\mathbb{I}_s^- = \mathbb{I}_{(\rH_s)^{-1}}$ and $(\rH_{w^{-1}})^{-1} = (\rH_{(w')^{-1}})^{-1} (\rH_s)^{-1}$,
assuming the claim for $w'$ we deduce it for $w$, which finishes the proof.

For the case of baby Verma modules one proceeds in the same way. Namely, one can assume that $x \in W$, and then proceed by induction on $w$. Using the same notation as above, by Remark~\ref{rmk:intertwining-bV} we have
\[
\bV_{w_\circ w} \cong \mathbb{S}_s^-(\bV_{w_\circ w'}),
\]
hence $\gamma_{\ho}^{\ho}(\bV_{w_\circ w}) \cong \mathbb{I}^-_s(\gamma_{\ho}^{\ho}( \bV_{w_\circ w'}))$. Assuming the claim for $w'$ and using the compati\-bility of the equivalences $\gamma_0^{\ho}$ and $\gamma_{\ho}^{\ho}$ (see Theorem~\ref{thm:localization-fixed-Fr}), the compatibility of the functors $\mathbb{I}_{(\rH_s)^{-1}}$ and $\mathbb{I}^\dg_{(\rH_s)^{-1}}$ (see~\S\ref{ss:wall-crossing-fixed-Fr}), and fully faithfulness of the functor $\mod^T(\sU_0^{\ho}) \to \mod^T(\sU_{\ho}^{\ho})$, this implies the claim for $w$ and finishes the proof.
%
\end{proof} 

For $x \in W_\ex$, we write $x=t_\lambda w_\circ w$ with $\lambda\in \X$ and $w\in W$, 
and set
\[
\tM_x := \mathbb{I}_{(\rH_{w^{-1}})^{-1}} (\sO_{(\g/\n)^{*(1)}}) \langle {p\lambda} \rangle \quad \in \Db\Coh^{T^{(1)}}(\Groth^{(1)})
\]
and 
\begin{equation*}
{\sM}_x := \mathbb{J}_{(\rH_{w^{-1}})^{-1}}(\sO_{(\g/\b)^{*(1)}}) \langle {p\lambda} \rangle \quad \in \Db\Coh^{T^{(1)}}(\Spr^{(1)}), 
\end{equation*}
where $(\g/\b)^{*(1)}$ is identified with the fiber of $\Spr^{(1)}$ over $w_\circ^{(1)}$.
Using the notation of~\S\ref{ss:localization-comp},
it follows from Corollary~\ref{cor:M-Z} that for any $x \in W_\ex$ we have
\begin{equation}
\label{eqn:hM-tM}
\hM_x \cong (\tM_x)_{\ho}.
\end{equation}
For any $x\in W_\ex$ we also have
\begin{equation}
\label{eqn:sM-tM}
{\sM}_x\cong Li^*\tM_x.
\end{equation}

\section{Koszul duality} 
\label{sec:Koszul-duality}

\subsection{Linear Koszul duality}

In this subsection we recall the formulation of the ``linear Koszul duality'' equivalence constructed in~\cite[\S 2]{Ric10}. More specifically we will use a slight improvement of this construction, which allows to get rid of some technical complications (more specifically, the boundedness condition in the definition given in~\cite[Eqn.~(2.3.6)]{Ric10}), based on the later work~\cite{mr-lkd2} of the first author with Mirkovi{\'c}. 

\subsubsection{The equivalence}
\label{sss:lkd-equivalence}

We consider an arbitrary scheme $X$ over a base commutative ring $k$, endowed with an action of a flat affine group scheme $K$ over $k$. We will consider some $(H \times \Gm)$-equivariant quasi-coherent sheaves of dg-algebras on $X$ where the multiplicative group $\Gm$ acts trivially, or in other words some $\Z^2$-graded $H$-equivariant quasi-coherent sheaves of algebras equipped with a differential of bidegree $(1,0)$, which satisfies the (graded) Leibniz rule. (So, the first component of the $\Z^2$-grading is the ``cohomological'' grading, and the second one is the ``internal'' grading determined by the $\Gm$-action.) We have an obvious notion of $(H \times \Gm)$-equivariant quasi-coherent sheaves of dg-modules over such dg-algebras.

Fix an $H$-equivariant locally free $\sO_X$-module $\sF$ of finite rank on $X$, and consider the $(H \times \Gm)$-equivariant quasi-coherent sheaves of dg-algebras $\sS$, $\sT$ and $\sR$ defined as follows:
\begin{itemize}
\item
$\sS$ is the symmetric algebra of $\sF^\vee$, with $\sF^\vee$ placed in bidegree $(2,-2)$, endowed with the trivial differential;
\item
$\sR$ is the symmetric algebra of $\sF^\vee$, with $\sF^\vee$ placed in bidegree $(0,-2)$, endowed with the trivial differential;
\item
$\sT$ is the exterior algebra of $\sF$, with $\sF$ placed in bidegree $(-1,2)$, endowed with the trivial differential.
\end{itemize}
Denote by $\mathsf{C}^{H \times \Gm}(\sS)$ the category of $(H \times \Gm)$-equivariant quasi-coherent sheaves of $\sS$-dg-modules, by $\mathsf{C}^{H \times \Gm}_-(\sS)$ the full subcategory of dg-modules which are concentrated (for the $\Z^2$-grading) in a region of the form $\Z \times (N + \Z_{\leq 0})$ for some $N \in \Z$ (depending on the dg-module), and by $D(\mathsf{C}^{H \times \Gm}(\sS))$, $D(\mathsf{C}_-^{H \times \Gm}(\sS))$ the corresponding localizations at quasi-isomorphisms of their homotopy categories. We define the corresponding categories 
for $\sR$ and $\sT$
in a similar way.

With these definitions, the constructions of~\cite[Theorem~2.6]{mr-lkd2} (with additional $H$-equivariance imposed) provide a canonical equivalence of triangulated categories
\begin{equation}
\label{eqn:lkd}
D(\mathsf{C}_-^{H \times \Gm}(\sS)) \simto D(\mathsf{C}_-^{H \times \Gm}(\sT)).
\end{equation}

Now, let us assume that $X$ is noetherian, and denote by $D_{\mathrm{fg}}(\mathsf{C}_-^{H \times \Gm}(\sS))$, resp.~by $D_{\mathrm{fg}}(\mathsf{C}_-^{H \times \Gm}(\sT))$, the full subcategory of $D(\mathsf{C}_-^{H \times \Gm}(\sS))$, resp.~$D(\mathsf{C}_-^{H \times \Gm}(\sT))$, consisting of modules whose cohomology is locally finitely generated (over the respective dg-algebra, which coincides with its cohomology). Then the same arguments as in~\cite[\S 2.8]{mr-lkd2} show that the equivalence~\eqref{eqn:lkd} restricts to an equivalence
\begin{equation}
\label{eqn:lkd-fg}
D_{\mathrm{fg}}(\mathsf{C}_-^{H \times \Gm}(\sS)) \simto D_{\mathrm{fg}}(\mathsf{C}_-^{H \times \Gm}(\sT)).
\end{equation}
(See~\cite[Proposition~1.1.5]{r-Frobenius} for more details.) Moreover, if we define the categories $D_{\mathrm{fg}}(\mathsf{C}^{H \times \Gm}(\sS))$, $D_{\mathrm{fg}}(\mathsf{C}^{H \times \Gm}(\sT))$ in a similar way, there exist canonical functors
\begin{align*}
D_{\mathrm{fg}}(\mathsf{C}_-^{H \times \Gm}(\sS)) &\to D_{\mathrm{fg}}(\mathsf{C}^{H \times \Gm}(\sS)), \\
D_{\mathrm{fg}}(\mathsf{C}_-^{H \times \Gm}(\sT)) &\to D_{\mathrm{fg}}(\mathsf{C}^{H \times \Gm}(\sT)),
\end{align*}
which are easily seen to be equivalences, as in~\cite[Eqn.~(2.3)]{mr-lkd2}. (For $\sS$-dg-modules, this follows from the fact that any object in $D_{\mathrm{fg}}(\mathsf{C}^{H \times \Gm}(\sS))$ has its cohomology concentrated in a region of the form $\Z \times (N + \Z_{\leq 0})$ for some $N \in \Z$, and then that inclusion of the submodule obtained by replacing all components not in this region by $0$ into this object is a quasi-isomorphism. A similar argument using quotients applies for $\sT$.) Hence, from~\eqref{eqn:lkd-fg} we obtain an equivalence of triangulated categories
\begin{equation}
\label{eqn:lkd-fg-2}
D_{\mathrm{fg}}(\mathsf{C}^{H \times \Gm}(\sS)) \simto D_{\mathrm{fg}}(\mathsf{C}^{H \times \Gm}(\sT)).
\end{equation}

Next, we consider the similar categories of $\sR$-dg-modules. We have a ``regrading'' equivalence
\[
\xi : \mathsf{C}^{H \times \Gm}(\sS) \simto \mathsf{C}^{H \times \Gm}(\sR),
\]
which simply re-arranges the $\Z^2$-grading of an $(H \times \Gm)$-equivariant $\sS$-dg-module to make it an $\sR$-dg-module, according to the rule $\xi(\sM)^{(i,j)} = \sM^{(i-j,j)}$. On the categories $\mathsf{C}^{H \times \Gm}(\sS)$ and $\mathsf{C}^{H \times \Gm}(\sR)$ (and also, for later use, on their analogue for $\sT$) we have the cohomological shift functor $[1]$, but also the ``internal shift'' $\langle 1 \rangle$ with respect to the $\Gm$-equivariant structure, see~\S\ref{sss:grading-shift}. The functor $\xi$ commutes with cohomological shifts, but for the internal shifts it satisfies $\xi \circ \langle 1 \rangle = \langle 1 \rangle [-1] \circ \xi$. This induces an equivalence between the corresponding derived categories, which preserves dg-modules with locally finitely generated cohomology. We therefore obtain an equivalence of triangulated categories
\[
D_{\mathrm{fg}}(\mathsf{C}^{H \times \Gm}(\sS)) \simto D_{\mathrm{fg}}(\mathsf{C}^{H \times \Gm}(\sR))
\]
which, combined with~\eqref{eqn:lkd-fg-2}, provides an equivalence of triangulated categories
\[
\varkappa : D_{\mathrm{fg}}(\mathsf{C}^{H \times \Gm}(\sR)) \simto D_{\mathrm{fg}}(\mathsf{C}^{H \times \Gm}(\sT))
\]
which satisfies $\varkappa \circ \langle 1 \rangle = \langle 1 \rangle [1] \circ \varkappa$. Finally, let us denote by $F$ the relative spectrum of $\sR$, seen as a sheaf of algebras on $X$. Then $F$ is a noetherian scheme endowed with an affine morphism $F \to X$ and a compatible action of $H \times \Gm$, and $\sR$ identifies with the pushforward of $\sO_F$ to $X$. By standard arguments (see e.g.~\cite[\href{https://stacks.math.columbia.edu/tag/01SB}{Tag 01SB}]{stacks-project}), pushforward induces an equivalence between the categories $D \QCoh^{H \times \Gm}(F)$ and $D(\mathsf{C}^{H \times \Gm}(\sR))$, which restricts to an equivalence between $\Db \Coh^{H \times \Gm}(F)$ and $D_{\mathrm{fg}}(\mathsf{C}^{H \times \Gm}(\sR))$. We can therefore interpret $\varkappa$ as an equivalence of triangulated categories
\[
\varkappa : \Db \Coh^{H \times \Gm}(F) \simto D_{\mathrm{fg}}(\mathsf{C}^{H \times \Gm}(\sT)).
\]

\subsubsection{Compatibilities -- 1}
\label{sss:lkd-compatibilities}

We now state two compatibility properties of the equivalence $\varkappa$ of~\S\ref{sss:lkd-equivalence}. See~\cite[\S\S2.4--2.5]{Ric10} for statements of this kind, whose proofs are similar, and~\cite[\S\S 3--4]{mr-lkd2} for generalizations.

First, let $X$, $Y$ be two noetherian $k$-schemes endowed with actions of $H$, and let $\pi : X \to Y$ be a proper and flat morphism. (Here properness is important to ensure that appropriate pushforward functors send coherent sheaves to coherent sheaves; flatness is imposed only for simplicity.) Consider an $H$-equivariant locally free $\sO_Y$-module of finite rank $\sF$, and its pullback $\pi^*(\sF)$. Then we have the associated scheme $F$ over $Y$, and the scheme associated with $\pi^*(\sF)$ is $X \times_Y F$. We also have the equivariant quasi-coherent sheaves of dg-algebras $\sT_X$ on $X$ and $\sT_Y$ on $Y$ such that $\sT_X = \pi^*(\sT_Y)$, and the equivalences
\begin{gather*}
\varkappa_X : \Db \Coh^{H \times \Gm}(X \times_Y F) \simto D_{\mathrm{fg}}(\mathsf{C}^{H \times \Gm}(\sT_X)), \\
\varkappa_Y : \Db \Coh^{H \times \Gm}(F) \simto D_{\mathrm{fg}}(\mathsf{C}^{H \times \Gm}(\sT_Y)).
\end{gather*}

We have a morphism $\pi_F : X \times_Y F \to F$ induced by $\pi$, and associated (adjoint) derived pushforward and pullback functors
\[
R(\pi_F)_* : \Db \Coh^{H \times \Gm}(X \times_Y F) \leftrightarrows \Db \Coh^{H \times \Gm}(F) : L(\pi_F)^*.
\]

Similarly, since $\pi$ is flat it induces a (derived) pullback functor
\[
L(\pi_{\sT})^* : D_{\mathrm{fg}}(\mathsf{C}^{H \times \Gm}(\sT_Y)) \to D_{\mathrm{fg}}(\mathsf{C}^{H \times \Gm}(\sT_X)).
\]
Using K-injective resolutions (see the comments made in the course of the proof of Theorem~\ref{thm:localization-fixed-Fr}) one can also construct a (right adjoint) derived pushforward functor
\[
R(\pi_{\sT})_* : D_{\mathrm{fg}}(\mathsf{C}^{H \times \Gm}(\sT_X)) \to D_{\mathrm{fg}}(\mathsf{C}^{H \times \Gm}(\sT_Y)).
\]

With these definitions, one can check that the diagrams
\[
\begin{tikzcd}
\Db \Coh^{H \times \Gm}(X \times_Y F) \arrow[r, "\varkappa_X", "\sim"'] \arrow[d, "R(\pi_F)_*"'] & D_{\mathrm{fg}}(\mathsf{C}^{H \times \Gm}(\sT_X)) \arrow[d, "R(\pi_{\sT})_*"] \\
\Db \Coh^{H \times \Gm}(F) \arrow[r, "\varkappa_Y", "\sim"'] & D_{\mathrm{fg}}(\mathsf{C}^{H \times \Gm}(\sT_Y))
\end{tikzcd}
\]
and
\[
\begin{tikzcd}
\Db \Coh^{H \times \Gm}(F) \arrow[r, "\varkappa_Y", "\sim"'] \arrow[d, "L(\pi_F)^*"'] & D_{\mathrm{fg}}(\mathsf{C}^{H \times \Gm}(\sT_Y)) \arrow[d, "L(\pi_{\sT})^*"] \\
\Db \Coh^{H \times \Gm}(X \times_Y F) \arrow[r, "\varkappa_X", "\sim"'] & D_{\mathrm{fg}}(\mathsf{C}^{H \times \Gm}(\sT_X))
\end{tikzcd}
\]
commute, and more specifically that the isomorphisms making these diagrams commutative can be chosen in such a way that the natural adjunction morphisms $L(\pi_{\sT})^* \circ R(\pi_{\sT})_* \to \id$ and $L(\pi_F)^* \circ R(\pi_F)_* \to \id$ are intertwined.

\subsubsection{Compatibilities -- 2}
\label{sss:lkd-compatibilities-2}

In this subsection, for simplicity we assume that $k$ is an algebraically closed field and $X$ is a smooth variety over $k$.
We consider two $H$-equivariant locally free $\sO_X$-modules $\sF_1$ and $\sF_2$ of constant finite ranks over $X$, and an $H$-equivariant embedding $\sF_1 \hookrightarrow \sF_2$ whose cokernel is locally free. Then we have the schemes $F_1$ and $F_2$ over $X$ associated with $\sF_1$ and $\sF_2$ respectively, the corresponding sheaves of dg-algebras $\sT_1$ and $\sT_2$, and the equivalences
\begin{gather*}
\varkappa_1 : \Db \Coh^{H \times \Gm}(F_1) \simto D_{\mathrm{fg}}(\mathsf{C}^{H \times \Gm}(\sT_1)), \\
\varkappa_2 : \Db \Coh^{H \times \Gm}(F_2) \simto D_{\mathrm{fg}}(\mathsf{C}^{H \times \Gm}(\sT_2)).
\end{gather*}

Our morphism $\sF_1 \hookrightarrow \sF_2$ induces a closed immersion $f : F_1 \hookrightarrow F_2$, and we have associated derived pushforward and pullback functors
\[
Rf_* : \Db \Coh^{H \times \Gm}(F_1) \leftrightarrows \Db \Coh^{H \times \Gm}(F_2) : Lf^*.
\]
We let $\sL_1$ and $\sL_2$ be the top exterior powers of $\sT_1$ and $\sT_2$ respectively, denote the ranks of $\sT_1$ and $\sT_2$ by $r_1$ and $r_2$ respectively, and set
\[
f_\dag(-) := Rf_*(-) \otimes_{\sO_X} \sL_2 \otimes_{\sO_X} \sL_1^\vee [r_1-r_2] \langle 2r_2-2r_1 \rangle
\]
(where we identify the categories of sheaves on $F_1$ and $F_2$ with some categories of sheaves on $X$ as in~\S\ref{sss:lkd-equivalence}).
Using Grothendieck--Serre duality one sees that $f_\dag$ is left adjoint to $Lf^*$.

Similarly we have an embedding of equivariant dg-algebras $\sT_1 \hookrightarrow \sT_2$, and associated (adjoint) extension and restriction of scalars functors
\[
Rg_* : D_{\mathrm{fg}}(\mathsf{C}^{H \times \Gm}(\sT_2)) \leftrightarrows D_{\mathrm{fg}}(\mathsf{C}^{H \times \Gm}(\sT_1)) : Lg^*.
\]
(Here the functors at the level of categories of dg-modules are exact, so the construction of derived functors is not an issue.) 


With these definitions, one can check that the diagrams
\[
\begin{tikzcd}
\Db \Coh^{H \times \Gm}(F_2) \arrow[r, "\varkappa_2", "\sim"'] \arrow[d, "Lf^*"'] & D_{\mathrm{fg}}(\mathsf{C}^{H \times \Gm}(\sT_2)) \arrow[d, "Rg_*"] \\
\Db \Coh^{H \times \Gm}(F_1) \arrow[r, "\varkappa_1", "\sim"'] & D_{\mathrm{fg}}(\mathsf{C}^{H \times \Gm}(\sT_1))
\end{tikzcd}
\]
and
\[
\begin{tikzcd}
\Db \Coh^{H \times \Gm}(F_1) \arrow[r, "\varkappa_1", "\sim"'] \arrow[d, "f_!"'] & D_{\mathrm{fg}}(\mathsf{C}^{H \times \Gm}(\sT_1)) \arrow[d, "Lg^*"] \\
\Db \Coh^{H \times \Gm}(F_2) \arrow[r, "\varkappa_2", "\sim"'] & D_{\mathrm{fg}}(\mathsf{C}^{H \times \Gm}(\sT_2))
\end{tikzcd}
\]
commute. Moreover, the corresponding isomorphisms can be chosen compatibly with adjunctions as above.

\subsection{Koszul duality for \texorpdfstring{$\g$}{g}-modules}

We now come back to the setting of Section~\ref{sec:localization}.
Recall that a weight $\lambda\in \X$ is called \emph{restricted} if $0\leq \langle \lambda, \alpha^\vee \rangle < p$ for any $\alpha\in \Phis$. 
We will denote by $W^\res_\ex$ the subset of elements $x\in W_\ex$ such that $x\bullet 0$ is restricted. It is easily seen that the projection morphism $W_\ex \to W$ restricts to a bijection $W_\ex^\res \simto W$, see e.g.~\cite[\S 4.1]{Ric10}. 
We will consider the permutation $\check{-} : W_\ex \simto W_\ex$ defined as follows.
For any $x\in W_\ex$, there is a unique expression $x=t_\lambda w$ with $\lambda\in \X$ and $w\in W^\res_\ex$; then we set $\check{x}:=t_\lambda w_\circ w$. It is clear that $\check{x\omega} = \check{x} \omega$ for any $\omega \in \Omega$ and $x \in W_\ex$, and that via the bijection between $W_\aff$ and alcoves as in~\cite[beginning of \S 5]{Soe97}, the restriction of this operation to $W_\aff$ corresponds to the operation $A \mapsto \check{A}$ of~\cite[End of~\S 4]{Soe97}.


We consider the construction of~\S\ref{sss:lkd-equivalence} in the following setting: the group scheme ``$H$'' is $T$, the scheme ``$X$'' is $\sB^{(1)}$, and $\sF$ is the sheaf of sections of the vector bundle $\Spr^{(1)} \to \sB^{(1)}$. Of course, we then have $F=\Spr^{(1)}$. Denoting by $\mathsf{C}^T(\sT)$, $D(\mathsf{C}^T(\sT))$ and $D_{\mathrm{fg}}(\mathsf{C}^T(\sT))$ the categories similar to $\mathsf{C}^{T \times \Gm}(\sT)$, $D(\mathsf{C}^{T \times \Gm}(\sT))$ and $D_{\mathrm{fg}}(\mathsf{C}^{T \times \Gm}(\sT))$ but with the $(T \times \Gm)$-equivariance condition replaced by mere $T$-equivariance, there exists an equivalence of triangulated categories
\begin{equation}
\label{eqn:equiv-dgGroth-sT}
D_{\mathrm{fg}}(\mathsf{C}^T(\sT)) \cong \DGCoh^{T}(\Groth^{(1)}\times^R_{{\g}^{*(1)}} \{0\})
\end{equation}
where the right-hand side is as in~\S\ref{ss:localization-fixed}. In fact, this claim follows the arguments in~\cite[Lemma~4.1.1]{mr-lkd} or in~\cite[\S 2.3]{Ric10}, based on the fact that $\sT$ is the complex obtained by tensoring a ``Koszul-type'' resolution of the pushforward of $\sO_{\Groth^{(1)}}$ to $\sB^{(1)}$, seen as a module over the pushforward of $\sO_{\g^{*(1)} \times \sB^{(1)}}$, with $\sO_{\sB^{(1)}}$, seen as the ``trivial'' module over the latter algebra, and invariance of derived categories of (equivariant) dg-modules under quasi-isomorphism. (This also uses the fact that $\Groth$ identifies with the orthogonal in $\g^* \times \sB$ of $\Spr$; see~\cite[\S 3.1]{Ric10}.)

In view of this, the category $D_{\mathrm{fg}}(\mathsf{C}^{T \times \Gm}(\sT))$ in this particular setting will be denoted $\DGCoh^{T \times \Gm}(\Groth^{(1)}\times^R_{{\g}^{*(1)}} \{0\})$. There is a natural triangulated ``forgetful'' functor
\[
\for_{\Gm} : \DGCoh^{T \times \Gm}(\Groth^{(1)}\times^R_{{\g}^{*(1)}} \{0\}) \to \DGCoh^{T}(\Groth^{(1)}\times^R_{{\g}^{*(1)}} \{0\}),
\]
which is a ``degrading functor'' in the sense that we have $\for_{\Gm} \circ \langle 1 \rangle = \for_{\Gm}$, and the natural morphism
\[
\bigoplus_{n \in \Z} \Hom(\sF, \sG \langle n \rangle) \to \Hom(\for_{\Gm}(\sF), \for_{\Gm}(\sG))
\]
is an isomorphism
for any $\sF, \sG$ in $\DGCoh^{T \times \Gm}(\Groth^{(1)}\times^R_{{\g}^{*(1)}} \{0\})$.

With this notation, the constructions of~\S\ref{sss:lkd-equivalence} provide an equivalence of triangulated categories
\[
\kappa: \Db\Coh^{T\times \Gm}(\Spr^{(1)}) \simto
\DGCoh^{T\times \Gm}(\Groth^{(1)}\times^R_{{\g}^{*(1)}} \{0\})
\] 
satisfying $\kappa \circ \langle 1\rangle \cong \langle 1\rangle [1] \circ \kappa$.
We also have a canonical triangulated forgetful functor
\[
\for_{\Gm} : \Db\Coh^{T\times \Gm}(\Spr^{(1)}) \to \Db\Coh^{T}(\Spr^{(1)})
\]
which is a degrading functor in the same sense as above. Recall that the category $\Db\Coh^{T^{(1)} \times \Gm}(\Spr^{(1)})$ identifies with a direct summand in $\Db\Coh^{T \times \Gm}(\Spr^{(1)})$. Replacing in the definitions above $T$ by $T^{(1)}$ one makes sense of the triangulated category $\DGCoh^{T^{(1)}\times \Gm}(\Groth^{(1)}\times^R_{{\g}^{*(1)}} \{0\})$, which identifies with a direct summand in $\DGCoh^{T \times \Gm}(\Groth^{(1)}\times^R_{{\g}^{*(1)}} \{0\})$, and $\kappa$ restricts to an equivalences of triangulated categories $\Db\Coh^{T^{(1)} \times \Gm}(\Spr^{(1)}) \simto \DGCoh^{T^{(1)}\times \Gm}(\Groth^{(1)}\times^R_{{\g}^{*(1)}} \{0\})$.

Recall that we say that
\emph{Lusztig's conjecture holds} if the equivalent conditions in~\cite[Proposition~II.C.17(a)]{Jan03} are true. See~\cite[Chap.~II.C]{Jan03} and~\cite{fiebig} for discussions of these conditions, and the relation with Lusztig's original conjecture on characters of simple $G$-modules. In particular, it is now known that these conditions are true provided $p$ is larger than a bound depending on $\Phi$, see~\cite{fiebig-upper}. Recall the objects $(\sP_x : x \in W_\ex)$ and $(\sL_x : x \in W_\ex)$ defined in~\S\ref{ss:standard-notation}.

The following statement is an equivariant version of~\cite[Theorem~4.4.3]{Ric10}.

\begin{thm}
\label{thm:KD} 
Assume that Lusztig's conjecture holds.
There exist collections
\[
(\sL_x^\gr : x \in W_\ex)
\]
of objects of the category $\Db\Coh^{T^{(1)} \times \Gm}(\Spr^{(1)})$ and
\[
(\sP_x^\gr : x \in W_\ex)
\]
of objects of the category $\DGCoh^{T^{(1)} \times \Gm}(\Groth^{(1)}\times^R_{{\g}^{*(1)}} \{0\})$ which satisfy the following properties:
\begin{enumerate}
\item
for all $x \in W_\ex$ we have
$\for_{\Gm}(\sL_x^\gr) \cong \sL_x$ and $\for_{\Gm}(\sP_x^\gr) \cong \sP_x$;
\item
for any $x \in W_\ex$ we have an isomorphism
\begin{equation*}
\kappa(\sL^\gr_x\otimes_{\sO_{\Spr^{(1)}}} \sO_{\Spr^{(1)}}(-\rho)) \cong \sP^\gr_{t_\rho\check{x}};
\end{equation*}
\item
for $x \in W_\ex$ and $\lambda \in \X$ we have $\sL^\gr_{t_\lambda x}\cong \sL^\gr_{x} \langle p\lambda \rangle$ and $\sP^\gr_{t_\lambda x}\cong \sP^\gr_{x} \langle p\lambda \rangle$.
\end{enumerate}
\end{thm}


%

Theorem~\ref{thm:KD} can be proved by an immediate generalization of the methods of~\cite{Ric10}. We will give an outline of this proof (which incorporates slight simplifications compared to the version in~\cite{Ric10}) in~\S\ref{ss:proof-KD} below. (No detail of this proof will be used in Section~\ref{sec:dim}.)

\subsection{Outline of the proof of Theorem~\ref{thm:KD}}
\label{ss:proof-KD}

In this subsection we assume that Lusztig's conjecture holds.

\subsubsection{Semisimple functors -- representation theoretic interpretation}

If $w \in W_\ex^\res$ and $s \in S_\aff$ are such that $ws>w$ and $ws \in W_\ex^\res$, we consider the morphisms
\[
\Simp_w \xrightarrow{\eta^s_w} \Theta_s(\Simp_w) \xrightarrow{\epsilon^s_w} \Simp_w
\]
induced by the morphisms $\id \to \Theta_s$ and $\Theta_s \to \id$ induced by adjunction. Since $\sfT_0^{\mu_s}(\Simp_w) \neq 0$ (see~\cite[Proposition~II.7.15]{Jan03}) these morphisms are nonzero, hence $\eta^s_w$ is injective and $\epsilon^s_w$ is surjective, and their composition is zero. Note also that since $\sfT_0^{\mu_s}(\Simp_w)$ is simple, $\eta^s_w$ is a generator of $\Hom_{\mod^T(\sU_{\ho}^{\ho})}(\Simp_w, \Theta_s(\Simp_w))$, and $\epsilon^s_w$ is a generator of $\Hom_{\mod^T(\sU_{\ho}^{\ho})}(\Theta_s(\Simp_w), \Simp_w)$. We set $\mathsf{Q}^s_w := \ker(\epsilon^s_w)/ \mathrm{Im}(\eta^s_w)$. 

\begin{lem}
\label{lem:Upsilon-simples}
Let $w \in W_\ex^\res$ and $s \in S_\aff$ be such that $ws>w$ and $ws \in W_\ex^\res$. Then we have
\[
\Upsilon_s(\sL_w) \cong \gamma_{\ho}^0(\mathsf{Q}^s_w).
\]
\end{lem}

\begin{proof}
To prove the lemma it suffices to construct an isomorphism $i_* \Upsilon_s(\sL_w) \cong \gamma_{\ho}^{\ho}(\mathsf{Q}^s_w)$.
As explained in~\S\ref{ss:ss-functors}, we have a distinguished triangle
\[
\Upsilon_s(\sL_w) \to \mathbb{J}_{\rH_s}(\sL_w) \to \sL_w \xrightarrow{[1]},
\]
hence a distinguished triangle
\begin{equation}
\label{eqn:triangle-Upsilon-Q}
i_* \Upsilon_s(\sL_w) \to \mathbb{I}_{\rH_s}(i_* \sL_w) \to i_* \sL_w \xrightarrow{[1]}.
\end{equation}

By the considerations in~\S\ref{ss:geometric-wall-crossing-aff}, we also have a distinguished triangle
\[
i_* \sL_w \to \Xi_s(i_* \sL_w) \to \mathbb{I}_{\rH_s}(i_* \sL_w) \xrightarrow{[1]}.
\]
Here, since $i_* \sL_w$ is indecomposable and $\mathbb{I}_{\rH_s}$ is an equivalence, the first morphism must be nonzero. Since the image under
$(\gamma_{\ho}^{\ho})^{-1}$ of the first, resp.~second, term in this triangle is $\Simp_w$, resp.~$\Theta_s(\Simp_w)$ (see Proposition~\ref{prop:wall-crossing-aff}), we deduce that $\mathbb{I}_{\rH_s}(i_* \sL_w) \cong \gamma_{\ho}^{\ho}(\mathrm{cok}(\eta^s_w))$. To conclude the proof, it therefore suffices to justify that the second morphism in~\eqref{eqn:triangle-Upsilon-Q} is nonzero. However, the composition of this morphism with the morphism $\Xi_s(i_* \sL_w) \to \mathbb{I}_{\rH_s}(i_* \sL_w)$ considered above fits in a distinguished triangle
\[
\mathbb{I}_{(\rH_s)^{-1}}(i_* \sL_w) \to
\Xi_s(i_* \sL_w) \to i_* \sL_w \xrightarrow{[1]},
\]
see again~\S\ref{ss:geometric-wall-crossing-aff}. As above, since $\mathbb{I}_{(\rH_s)^{-1}}$ is an equivalence this composition must be nonzero, hence so is our morphism.
\end{proof}

The following statement is analogous to~\cite[Theorem~5.5.3 and Proposition~5.5.4]{Ric10}. (Here, we use Lusztig's conjecture to prove semisimplicity of the module appearing in~\eqref{it:decomposition-Q}.)

\begin{prop}
\label{prop:Q-ThetaP}
Let $w \in W_\ex^\res$ and $s \in S_\aff$.
\begin{enumerate}
\item
\label{it:decomposition-Q}
If $ws>w$ and $ws \in W_\ex^\res$, then in $\mod^T(\sU_0^{\ho})$ we have
\[
\mathsf{Q}^s_w \cong \Simp_{ws} \oplus \bigoplus_{y} \Simp_y^{\oplus n_y}
\]
for some nonnegative integers $n_y$, where $y$ runs over the elements of $W_\ex$ of the form $t_\nu z$ with $\nu \in \X$ and $z \in W_\ex^\res$ which satisfies $\ell(z)\leq \ell(w)$.
\item
\label{it:Theta-Proj}
If $ws<w$ and $ws \in W_\ex^\res$, then in $\mod^T(\sU_0^{\ho})$ we have
\[
\Theta_s(\Pro_w) \cong \Pro_{ws} \oplus \bigoplus_y \Pro_y^{\oplus n_y}
\]
for some nonnegative integers $n_y$, where $y$ runs over the elements of $W_\ex$ of the form $t_\nu z$ with $\nu \in \X$ and $z \in W_\ex^\res$ which satisfies $\ell(z) \geq \ell(w)$.
\end{enumerate}
\end{prop}

\subsubsection{Geometric description of some simple and projective modules}

It is clear that $\Omega \subset W_\ex^\res$. For such elements we have an explicit description of the objects $\sL_\omega$, as follows.

\begin{lem}
\label{lem:simples-length-0}
Let $\omega \in \Omega$, and write $\omega = wt_\mu$ with $w \in W$ and $\mu \in \X$. Then we have
\[
\sL_\omega \cong \sO_{\sB^{(1)}}(-\rho+\mu)[\ell(w)] \langle p\rho \rangle
\]
(where we omit pushforward under the closed immersion $\sB^{(1)} \hookrightarrow \Spr^{(1)}$).
\end{lem}

\begin{proof}
The proof is similar to that of~\cite[Lemma~7.2.2]{Ric10}. Its main ingredients are the isomorphism~\eqref{eqn:splitting-bundle-0}, the fact that $(\mathrm{Fr}_{\sB})_* \sO_\sB(-\rho) \cong \Simp((p-1)\rho) \otimes_\k \sO_{\sB^{(1)}}(-\rho)$, and the Borel--Weil--Bott theorem.
%
\end{proof}

Similarly, we consider the elements of the form $t_\rho w_\circ \omega$ with $\omega \in \Omega$. By~\cite[Proposition~4.1.2]{Ric10}, these elements belong to $W_\ex^\res$, and are exactly the elements of maximal length in it.

\begin{lem}
\label{lem:proj-max-length}
Let $\omega \in \Omega$, and write $\omega = wt_\mu$ with $w \in W$ and $\mu \in \X$. Then we have
\[
\sP_{t_\rho w_\circ \omega} \cong \sO_{\Groth^{(1)} \times_{\g^{*(1)}} \{0\}}(\mu) \langle p\rho \rangle.
\]
\end{lem}

\begin{proof}
The proof is similar to that of~\cite[Eqn.~(6.4.5)]{Ric10}.
\end{proof}

\subsubsection{Graded wall-crossing functors and their duals}

Recall that in~\S\ref{ss:wall-crossing-fixed-Fr} we have constructed, for $s \in S$, an endofunctor $\Xi_s^{\dg}$
of $\DGCoh^{T \times \Gm}(\Groth^{(1)} \times_{\g^{*(1)}} \{0\})$. This functor is constructed using pushforward and pullback functors for a morphism induced by $\tpi_s$. Writing $\tpi_s$ as a composition
\[
\Groth \to \Groth_s \times_{\sP_s} \sB \to \Groth_s
\]
where the first morphism is the natural closed immersion of a sub-vector bundle and the second one is the natural projection map one sees that, under the identifications involved in the construction of the equivalence~\eqref{eqn:equiv-dgGroth-sT}, this morphism fits in the setting considered in~\S\ref{sss:lkd-compatibilities}--\ref{sss:lkd-compatibilities-2}. More specifically, applying these statements we obtain an isomorphism of functors
\begin{equation}
\label{eqn:dual-Xi}
\Xi_s^{\dg} \circ \kappa \cong \kappa \bigl( \sO_{\Spr^{(1)}}(-\rho) \otimes_{\sO_{\Spr^{(1)}}} \Upsilon_s((-) \otimes_{\sO_{\Spr^{(1)}}} \sO_{\Spr^{(1)}}(\rho)) \bigr) \langle 1 \rangle,
\end{equation}
such that the adjunction morphism $\Xi_s^{\dg} \to \id$ corresponds to the morphism considered in Lemma~\ref{lem:morph-adjunction-kernels}.

\begin{lem}
\label{lem:Koszul-duality-braid-action}
For any $b \in \Br_\ex$ and $\sF$ in $\Db\Coh^{T \times \Gm}(\Spr^{(1)})$, there exists an isomorphism
\[
\mathbb{I}^\dg_b (\kappa(\sF)) \cong \kappa \bigl( \sO_{\Spr^{(1)}}(-\rho) \otimes_{\sO_{\Spr^{(1)}}} \mathbb{J}_b(\sF \otimes_{\sO_{\Spr^{(1)}}} \sO_{\Spr^{(1)}}(\rho)) \bigr).
\]
\end{lem}

\begin{proof}
It suffices to prove the lemma for $b$ running over a subset of generators of $\Br_\ex$. When $b= \theta_x$ for some $x \in \X$ the statement is obvious from definitions (and we even have an isomorphism of functors). Next we consider the case $b=(\rH_s)^{-1}$ for some $s \in S$. In this case, the claim follows from the identification of the distinguished triangle
\[
\mathbb{I}^\dg_{(\rH_s)^{-1}}(\kappa(\sF)) \to \Xi_s^\dg(\kappa(\sF)) \langle -1 \rangle \to \kappa(\sF) \langle -1 \rangle \xrightarrow{[1]}
\]
(see~\eqref{eqn:triangles-Xi-L}) with the distinguished triangle induced by the first triangle in~\eqref{eqn:triangles-Upsilon-J}, based on~\eqref{eqn:dual-Xi} and Lemma~\ref{lem:morph-adjunction-kernels}.
\end{proof}

With Lemma~\ref{lem:Koszul-duality-braid-action} at hand, in view of definitions we obtain the following generalization of~\eqref{eqn:dual-Xi}: for any $s \in S_\aff$ and $\sF$ in $\Db\Coh^{T \times \Gm}(\Spr^{(1)})$ we have an isomorphism
\begin{equation}
\label{eqn:dual-Xi-aff}
\Xi_s^{\dg} \circ \kappa(\sF) \cong \kappa \bigl( \sO_{\Spr^{(1)}}(-\rho) \otimes_{\sO_{\Spr^{(1)}}} \Upsilon_s(\sF \otimes_{\sO_{\Spr^{(1)}}} \sO_{\Spr^{(1)}}(\rho)) \bigr) \langle 1 \rangle.
\end{equation}

\subsubsection{Completion of the proof}

We can now finish the description of the proof of Theorem~\ref{thm:KD}. Here we will use the following easy fact from the combinatorics of $W_\ex$, proved in~\cite[Proposition~4.1.2]{Ric10}: the map $x \mapsto t_\rho \check{x}$ restricts to an involution of $W_\ex^\res$, and moreover $\ell(t_\rho \check{x}) = \ell(t_\rho w_\circ) - \ell(x)$ for any $x \in W_\ex^\res$.

We divide the arguments into 3 steps.

\emph{Step 1}. 
We first construct, for any $x \in W_\ex$, a ``graded lift'' of $\sL_x$, i.e.~an object $\sL_x^\gr \in \Db\Coh^{T^{(1)}\times\Gm}(\widetilde{\sN}^{(1)})$ such that $\for_{\Gm}(\sL_x^\gr) \cong \sL_x$.
Since $\sL_{t_\lambda x}\cong \sL_x \langle p\lambda \rangle$ for any $x\in W_\ex$ and $\lambda\in \X$, it is enough to construct these objects when $x\in W_\ex^\res$. 
We will do this by induction on $\ell(x)$. 
If $\ell(x)=0$, then writing $x=w t_\mu$ with $w \in W$ and $\mu \in \X$, by Lemma~\ref{lem:simples-length-0} we can set 
\[
\sL^\gr_{x} := \sO_{\sB^{(1)}}(-\rho+\mu)[\ell(w)]\langle (p\rho, \ell(w_0)-\ell(w))\rangle.
\]

Suppose now that $x\in W^\res_\ex$ satisfies $\ell(x)>0$, and that we have constructed objects $\sL^\gr_y$ for all $y\in W^\res_\ex$ such that $\ell(y)<\ell(x)$. 
There exists $s \in S_\aff$ such that $xs \in W_\ex^\res$ and $xs<x$.
By Lemma~\ref{lem:Upsilon-simples} and Proposition~\ref{prop:Q-ThetaP}\eqref{it:decomposition-Q}
there is a decomposition 
\[
\Upsilon_s (\sL_{xs})=\sL_{x}\oplus \bigoplus_{\substack{y\in W^\res_\ex,\\ \ell(y)<\ell(x)}} \sL_y \otimes M_y 
\]
for some finite-dimensional $T^{(1)}$-modules $M_y$. Here, for any $y$ we have
\[
M_y = \Hom_{\Db\Coh(\Spr^{(1)})}(\sL_y, \Upsilon_s (\sL_{xs})),
\]
and the morphism
\[
\bigoplus_{\substack{y\in W^\res_\ex,\\ \ell(y)<\ell(x)}} \sL_y \otimes M_y \to \Upsilon_s (\sL_{xs})
\]
is the natural morphism obtained from these identifications.
Since we have fixed ``lifts'' $\sL_y^\gr$ of $\sL_y$ and $\sL_{xs}^\gr$ of $\sL_{xs}$, each vector space $M_y$ acquires a canonical action of $\Gm$ (i.e.~a grading), and we will denote by $M_y^\gr$ the corresponding $(T^{(1)} \times \Gm)$-module. Then we have a canonical morphism
\[
\bigoplus_{\substack{y\in W^\res_\ex,\\ \ell(y)<\ell(x)}} \sL_y^\gr \otimes M_y^\gr \to \Upsilon_s (\sL_{xs}^\gr).
\]
The cone $\sL^\gr_x$ of this morphism is then a graded lift of $\sL_x$.



\emph{Step 2}. We will now show that, in the setting of Step~1, we have
\begin{equation}
\label{eqn:Upsilon-L-proof-main}
\Upsilon_s (\sL_{xs}^\gr) \cong \sL_{x}^\gr \oplus \bigoplus_{\substack{y\in W^\res_\ex,\\ \ell(y)<\ell(x)}} \sL_y^\gr \otimes M_y^\gr.
\end{equation}
In fact, since $\End(\sL_x)=\k$, we have $\Hom(\sL_x,\Upsilon_s (\sL_{xs}))=\k$, hence there exists a unique $m \in \Z$ and a unique (up to scalar) nonzero morphism $\sL_x^\gr \langle m \rangle \to \Upsilon_s (\sL_{xs}^\gr)$. Similarly, there is a unique $m'$ and a unique (up to scalar) nonzero morphism $\Upsilon_s (\sL_{xs}^\gr) \to \sL_x^\gr \langle m' \rangle$, and moreover the composition $\sL_x^\gr \langle m \rangle \to \Upsilon_s (\sL_{xs}^\gr) \to \sL_x^\gr \langle m' \rangle$ is nonzero. Hence we must have $m=m'$, and by construction of $\sL_x^\gr$ we have $m'=0$. This morphism $\sL_x^\gr \langle m \rangle \to \Upsilon_s (\sL_{xs}^\gr)$ provides a splitting of our given morphism $\Upsilon_s (\sL_{xs}^\gr) \to \sL_x^\gr$, which shows the desired decomposition.

\emph{Step 3}.
We will finally show that, defining the objects $(\sP_y^\gr : y \in W_\ex)$ in such a way that $\kappa(\sL^\gr_x\otimes_{\sO_{\Spr^{(1)}}} \sO_{\Spr^{(1)}}(-\rho)) \cong \sP^\gr_{t_\rho\check{x}}$ for any $x \in W_\ex^\res$, we obtain graded lifts of the objects $\sP_y$ in the sense that $\for_{\Gm}(\sP_y^\gr) \cong \sP_y$ for any $y$.
 
As in Step 1, it is enough to prove this property when $y \in W_\ex^\res$, which will be done by downward induction on $\ell(y)$.
By the comments above,
the maximal possible length is obtained when
$y=t_\rho\check{\omega}$ for some $\omega\in \Omega$. In this case, writing $\omega=wt_\mu$ with $w\in W$ and $\mu\in \X$, one checks by explicit computation that
\[
\sP^\gr_{t_\rho\check{\omega}} = \sO_{\Groth^{(1)}\times^R_{\g^{*(1)}} \{0\}}(\mu)\langle (p\rho, -\ell(w_0)-\ell(w)) \rangle.
\]
Hence the desired claim follows from Lemma~\ref{lem:proj-max-length}.

Now, suppose that $y\in W^\res_\ex$ is not of maximal length, and that the claim is known for all strictly longer elements in $W^\res_\ex$.
Write $y = t_\rho \check{x}$ with $x \in W_\ex^\res \smallsetminus \Omega$,
and choose as in Step~1 an element $s \in S_\aff$ such that $xs \in W_\ex^\res$ and $xs<x$.
Applying $\kappa(-\otimes_{\sO_{\Spr^{(1)}}} \sO_{\Spr^{(1)}}(-\rho))$ in~\eqref{eqn:Upsilon-L-proof-main} and using~\eqref{eqn:dual-Xi-aff},
we obtain a decomposition 
\begin{equation}
\label{eqn:Xi-Pgr}
\Xi^\dg_s(\sP^\gr_{y s})\langle -1\rangle = \sP_y^\gr \oplus \bigoplus_{\substack{z \in W^\res_\ex,\\ \ell(z) > \ell(y)}} \sP^\gr_{z} \otimes M^{\gr,\prime}_z
\end{equation}
for some finite-dimensional $(T^{(1)}\times\Gm)$-modules $M^{\gr,\prime}_z$.
We have isomorphisms 
\[
\End(\for_{\Gm}(\sP^\gr_y))
\cong \bigoplus_{i \in \Z} \Hom(\sP_y^\gr, \sP_y^\gr\langle -i\rangle) \cong
\bigoplus_{i \in \Z} \Hom(\sL^\gr_x, \sL^\gr_x \langle -i\rangle[i]), 
\]
and the rightmost term embeds in $\bigoplus_{i \in \Z} \Ext^i(\Simp_x,\Simp_x)$.
Hence $\End(\for_{\Gm}(\sP_y^\gr))$ is a non-negatively graded finite-dimensional algebra with $\k$ as its degree-$0$ component;
in particular it 
local. 
Hence $\for_{\Gm}(\sP_y^\gr)$ is indecomposable. 
Comparing the image under $\for_{\Gm}$ of~\eqref{eqn:Xi-Pgr} with the decomposition in Proposition~\ref{prop:Q-ThetaP}\eqref{it:Theta-Proj} (using the induction hypothesis and Proposition~\ref{prop:wall-crossing-aff}) we deduce that $\for_{\Gm}(\sP_y^\gr) \cong \sP_y$, which finishes the proof.

\subsection{Duals of Verma modules}
\label{ss:duals-Verma}

%
%

To conclude this section we explain the effect of Koszul duality on standard objects. We equip $\sO_{(\g/\b)^{*(1)}}$ and $\sO_{(w_\circ^{(1)},0)}$ with their natural $\Gm$-equivariant structures. Then, for $x \in W_\ex$, we write
$x=t_\lambda w_\circ w$ with $\lambda\in \X$ and $w\in W$, and set 
\[
{\sM}^\gr_x:=\J_{(\rH_{w^{-1}})^{-1}} (\sO_{(\g/\b)^{*(1)}}) \langle p\lambda \rangle, \qquad
\sZ^\gr_x:=\mathbb{I}^\dg_{(\rH_{w^{-1}})^{-1}} (\sO_{(w_\circ^{(1)},0)}) \langle p\lambda \rangle.
\]
Then, by definition and Corollary~\ref{cor:M-Z} we have
\[
\for_{\Gm}(\sM^\gr_x) \cong \sM_x, \quad \for_{\Gm}(\sZ^\gr_x) \cong \sZ_x.
\]

\begin{lem}
\label{lem:Koszul-standards}
For any $x\in W_\ex$ we have
\[
\kappa(\sM^\gr_x\otimes_{\sO_{\Spr^{(1)}}} \sO_{\Spr^{(1)}}(-\rho))\ \cong \sZ^\gr_{t_\rho x}.
\]
\end{lem}

\begin{proof}
First we consider the case $x=w_\circ$ (i.e.~$w=e$ and $\lambda=0$).
We compute that 
\[
\kappa(\sM^\gr_{w_\circ} \otimes_{\sO_{\Spr^{(1)}}} \sO_{\Spr^{(1)}}(-\rho))
= \kappa\big(\sO_{(\g/\b)^{*(1)}} \langle {p\rho} \rangle \big)
\cong \sO_{(w_\circ^{(1)},0)} \langle {p\rho} \rangle=\sZ^\gr_{t_\rho w_\circ}.
\]
Using the compatibility of $\kappa$ with the $\Br_\ex$-actions (see Lemma~\ref{lem:Koszul-duality-braid-action}) and the twist functors we deduce the desired isomorphism for a general $x \in W_\ex$.
\end{proof}


\section{Dimensions of extension spaces}
\label{sec:dim}


In this section, we finally prove Theorem~\ref{thm:main}. 
From now on we assume that Lusztig's conjecture holds, so that in particular Theorem~\ref{thm:KD} can be applied.
By Proposition~\ref{prop:Ext-gB-gT},
for any $x,y\in W_\aff$ and $m \in \Z$ we have 
\begin{equation}
\label{eqn:Ext-gps-Section-proof} 
\Ext^m_{\Mod(\g,B)}(\Simp_x, \nabla_y) \cong
\Ext^m_{\Mod^T(\sU(\g))}(\tDelta_y, \Simp_x), 
\end{equation}
and we now want to compute the dimension of these spaces.

\subsection{The Koszul grading} 

Let $\E$ be the $(p\X)$-graded algebra
whose opposite algebra is
\begin{align*}
\E^\op &=
\bigoplus_{\lambda\in p\X} 
\Hom_{G_1 T} \left(\bigoplus_{w\in W} \Pro_{w} \langle \lambda \rangle,\bigoplus_{w\in W} \Pro_{w} \right)\\ 
&\cong \bigoplus_{\lambda\in p\X} 
\Hom_{\DGCoh^{T^{(1)}}(\Groth^{(1)}\times^R_{\g^{*(1)}} \{0\})} \left(\bigoplus_{w\in W} \sP_{w} \langle \lambda \rangle,\bigoplus_{w\in W} \sP_{w} \right),
\end{align*} 
where the isomorphism is induced by the equivalence $\gamma_0^{\ho}$, see~\S\ref{ss:localization-fixed}. It is easily seen that, neglecting the grading, we have
\begin{equation}
\label{eqn:E-no-grading}
\E = \End_{\sU_0^{\ho}} \left( \bigoplus_{w\in W} \Pro_{w} \right)^\op.
\end{equation}
In particular, this algebra is finite-dimensional.

We will denote by $\mod^{p\X}(\E)$ the category of $(p\X)$-graded finitely generated (equivalently, finite-dimensional) $\E$-modules. For any $\lambda \in p\X$ we have an autoequivalence
\[
\langle \lambda \rangle : \mod^{p\X}(\E) \simto \mod^{p\X}(\E),
\]
see~\S\ref{sss:grading-shift}.

\begin{rmk}
The center $\mathrm{Z}(G) \subset T$ acts on $\Pro_{w} \langle \lambda \rangle$ via the character corresponding to the image of $p\lambda$ in $\X/\Z\Phi$. Hence the $\lambda$-graded component of $\E$ vanishes unless $p\lambda \in \Z\Phi$, which under our running assumption that $p>h$ is equivalent to $\lambda \in \Z\Phi$. Hence the grading on $\E$ is concentrated in $\Z\Phi$.
\end{rmk}

Recall the category
$\mod^{T,0}(\sU_0^{\ho})=\mod^{T}(\sU_0^{\ho}) \cap \mod^{T,0}(\sU_0)$, see~\S\ref{ss:U0-modules}.
By Remark~\ref{rmk:G1T}, $\mod^{T,0}(\sU_0^{\ho})$ identifies with the sum of the ``blocks'' of the category of finite-dimensional $G_1T$-modules corresponding to weights in $\Omega \bullet 0$ in the sense of~\cite[\S II.9.22]{Jan03}. The simple objects in $\mod^{T,0}(\sU_0^{\ho})$ are (up to isomorphism) the objects $(\Simp_w : w \in W_\ex)$.

We have a natural functor
\begin{equation}
\label{eqn:Morita-equiv}
\bigoplus_{\lambda \in p\X} \Hom_{\mod^T(\sU_0^{\ho})} \left( \bigoplus_{w \in W} \Pro_{w} \langle \lambda \rangle,- \right) : \mod^{T,0}(\sU_0^{\ho}) \to \mod^{p\X}(\E),
\end{equation}
where to define the action of $\E$ we use the canonical isomorphisms
\[
\Hom_{\mod^T(\sU_0^{\ho})} ( \Pro_{w} \langle \lambda \rangle, \Pro_{y}) \cong \Hom_{\mod^T(\sU_0^{\ho})} ( \Pro_{ w} \langle \lambda+\mu \rangle, \Pro_{y} \langle \mu \rangle)
\]
for $y,w \in W$ and $\lambda,\mu \in p\X$ induced by the equivalence $\langle \mu \rangle$. 
By standard arguments (see e.g.~\cite[Proposition~E.4]{AJS94}) this functor is an equivalence of categories.
For any $x\in W_\ex$
we will denote by 
\[
\bfP_x, \quad \bfZ_x, \quad \bfL_x
\]
the images under this equivalence of $\Pro_x$, $\bV_x$, $\Simp_x$ respectively. 

Consider, for $x \in W_\ex$, the ``lift'' $\sP^\gr_x$ of the object $\sP_x$ to $\DGCoh^{T^{(1)} \times \Gm}(\widetilde{\g}^{(1)}\times^R_{{\g}^{*(1)}} \{0\})$ from Theorem~\ref{thm:KD}. For any $w,y \in W$ and $\lambda \in p\X$, the functor $\for_{\Gm}$ induces an isomorphism
\begin{multline*}
 \bigoplus_{n \in \Z} \Hom_{\DGCoh^{T^{(1)} \times \Gm}(\Groth^{(1)}\times^R_{\g^{*(1)}} \{0\})} ( \sP^\gr_{w} \langle (\lambda,n) \rangle, \sP^\gr_{y} ) \\
 \simto \Hom_{\DGCoh^{T^{(1)}}(\Groth^{(1)}\times^R_{\g^{*(1)}} \{0\})} ( \sP_{w} \langle \lambda \rangle, \sP_{y} ).
\end{multline*}
These decompositions define an additional $\Z$-grading on $\E$, which makes it $(p\X \times \Z)$-graded; this $(p\X \times \Z)$-graded ring will be denoted $\underline{\E}$. The results of~\cite[\S 9]{Ric10} show that, as a $\Z$-graded ring, $\underline{\E}$ is a Koszul ring. In particular, $\underline{\E}$ is concentrated in nonnegative degrees, and its degree-$0$ component is a semisimple ring.

We will denote by $\mod^{p\X \times \Z}(\underline{\E})$ the category of $(p\X \times \Z)$-graded finite-dimensional $\underline{\E}$-modules.
For $x \in W_\ex$ we set
\[
\bfP^\gr_{x} = \bigoplus_{(\lambda,n) \in p\X \times \Z} \Hom_{\DGCoh^{T^{(1)} \times \Gm}(\Groth^{(1)}\times^R_{\g^{*(1)}} \{0\})} \left( \bigoplus_{w \in W} \sP^\gr_{w} \langle (\lambda,n) \rangle, \sP^\gr_{x} \right).
\]
Then $\bfP^\gr_{x}$ is a $(p\X \times \Z)$-graded $\underline{\E}$-module, isomorphic to $\bfP_x$ as a $p\X$-graded $\E$-module. Moreover, for any $x \in W_\ex$ and $\lambda \in \X$ we have
$\bfP^\gr_{t_\lambda x}\cong \bfP^\gr_x \langle p\lambda \rangle$.

We have
\[
\underline{\E} = \bigoplus_{w \in W} \bfP^\gr_w
\]
as $(p\X \times \Z)$-graded $\underline{\E}$-modules; in particular, for the $\Z$-grading, each $\bfP^\gr_x$ ($x \in W_\ex$) is concentrated in nonnegative degrees, and its head, denoted $\bfL^\gr_{x}$, is concentrated in degree $0$. Note that $\bfL^\gr_{x}$ is isomorphic to $\bfL_x$ as a $(p\X)$-graded $\E$-module, and that the assignment $(x,n) \mapsto \bfL^\gr_x \langle n \rangle$ induces a bijection between $W_\ex \times \Z$ and the set of isomorphism classes of simple objects in $\mod^{p\X \times \Z}(\underline{\E})$. For any $x \in W_\ex$ and $\lambda \in \X$ we also have
$\bfL^\gr_{t_\lambda x}\cong \bfL^\gr_x \langle p\lambda\rangle$.


\begin{lem}
\label{lem:socle-Px}
For any $x \in W_\ex$, $\bfP^\gr_x$ is the injective hull of $\bfL_x^\gr \langle 2\ell(w_\circ) \rangle$ in the category $\mod^{p\X \times \Z}(\underline{\E})$.
\end{lem}

\begin{proof}
We already know that, as a $(p\X)$-graded $\E$-module, $\bfP_x$ is injective, with socle $\bfL_x$; see the equivalence~\eqref{eqn:Morita-equiv} and the discussion in~\S\ref{ss:standard-notation}. What remains to be understood is the $\Z$-degree in which this socle occurs in $\bfP^\gr_x$. Now by~\cite[Proposition~2.4.1]{BGS96}, this degree is $\ell\ell(\bfP_x^\gr)-1$ where $\ell\ell$ means the Loewy length. We conclude by~\cite[Proposition~II.D.8]{Jan03}.
\end{proof}



Below we will need the following statement.

\begin{lem}
\label{lem:Morita-equiv-graded}
Let $\sF,\sG$ in $\DGCoh^{T^{(1)} \times \Gm}(\Groth^{(1)}\times^R_{\g^{*(1)}} \{0\})$, and assume that the images $\sF'$ and $\sG'$ of $\sF$ and $\sG$ in $\DGCoh^{T^{(1)}}(\Groth^{(1)}\times^R_{\g^{*(1)}} \{0\})$ correspond, under the equivalence $\gamma_0^{\ho}$, to objects in $\mod^{T,0}(\sU_0^{\ho}) \subset \Db \mod^T(\sU_0^{\ho})$. Set
\begin{align*}
M &= \bigoplus_{(\lambda,n) \in p\X \times \Z} \Hom_{\DGCoh^{T^{(1)} \times \Gm}(\Groth^{(1)}\times^R_{\g^{*(1)}} \{0\})} \left( \bigoplus_{w \in W} \sP^\gr_{w} \langle (\lambda,n) \rangle, \sF \right), \\
N &= \bigoplus_{(\lambda,n) \in p\X \times \Z} \Hom_{\DGCoh^{T^{(1)} \times \Gm}(\Groth^{(1)}\times^R_{\g^{*(1)}} \{0\})} \left( \bigoplus_{w \in W} \sP^\gr_{w} \langle (\lambda,n) \rangle, \sG \right),
\end{align*}
seen as $(p\X \times \Z)$-graded $\underline{\E}$-modules.
Then the functor
\[
\bigoplus_{(\lambda,n) \in p\X \times \Z} \Hom_{\DGCoh^{T^{(1)} \times \Gm}(\Groth^{(1)}\times^R_{\g^{*(1)}} \{0\})} \left( \bigoplus_{w \in W} \sP^\gr_{w} \langle (\lambda,n) \rangle, - \right)
\]
induces an isomorphism
\[
\Hom_{\DGCoh^{T^{(1)} \times \Gm}(\Groth^{(1)}\times^R_{\g^{*(1)}} \{0\})} (\sF,\sG) \simto \Hom_{\mod^{p\X \times \Z}(\underline{\E})}(M,N).
\]
\end{lem}

\begin{proof}
Let $M'$ and $N'$ be the images of $M$ and $N$ in $\mod^{p\X}(\E)$. Then the functor
\[
\bigoplus_{\lambda \in \X} \Hom_{\DGCoh^{T^{(1)} \times \Gm}(\Groth^{(1)}\times^R_{\g^{*(1)}} \{0\})} \left( \bigoplus_{w \in W} \sP^\gr_{w} \langle \lambda \rangle, - \right)
\]
induces an isomorphism
\[
\Hom_{\DGCoh^{T^{(1)}}(\Groth^{(1)}\times^R_{\g^{*(1)}} \{0\})} (\sF',\sG') \simto \Hom_{\mod^{p\X}(\E)}(M',N').
\]
In fact, this morphism identifies with the composition of the morphism induced by $\gamma_0^{\ho}$ and the one induced by~\eqref{eqn:Morita-equiv}. Since both functors are equivalences, this justifies the claim.

Now we have isomorphisms
\[
\Hom_{\DGCoh^{T^{(1)}}(\Groth^{(1)}\times^R_{\g^{*(1)}} \{0\})} (\sF',\sG') \cong
\bigoplus_{n \in \Z} \Hom_{\DGCoh^{T^{(1)} \times \Gm}(\Groth^{(1)}\times^R_{\g^{*(1)}} \{0\})} (\sF,\sG \langle n \rangle)
\]
and
\[
\Hom_{\mod^{p\X}(\E)}(M',N') \cong
\bigoplus_{n \in \Z} \Hom_{\mod^{p\X \times \Z}(\underline{\E})}(M,N \langle n \rangle),
\]
so that the morphism of the lemma is a direct summand of the isomorphism considered above. The claim follows.
\end{proof}



\subsection{Graded baby Verma modules} 

For $x \in W_\ex$ we set
\[
\bfZ^\gr_{x} = \bigoplus_{(\lambda,n) \in p\X \times \Z} \Hom_{\DGCoh^{T^{(1)} \times \Gm}(\Groth^{(1)}\times^R_{\g^{*(1)}} \{0\})} \left( \bigoplus_{w \in W} \sP^\gr_{w} \langle (\lambda,n + \ell(w_\circ)) \rangle, \sZ^\gr_{x} \right),
\]
where $\sZ^\gr_x$ is as in~\S\ref{ss:duals-Verma}.
Then $\bfZ^\gr_{x}$ is a $(p\X \times \Z)$-graded $\underline{\E}$-module isomorphic to $\bfZ_x$ as a $(p\X)$-graded $\E$-module. 

In any finite length category, given an object $M$ and a simple object $L$, we denote by $[M:L]$ the multiplicity of $L$ as a composition factor of $M$, see e.g.~\cite[\href{https://stacks.math.columbia.edu/tag/0FCD}{Tag 0FCD}]{stacks-project}.

\begin{prop}
\label{prop:dim-Ext-multiplicities}
For any $x,y \in W_\aff$ and $m \in \Z$ we have 
\begin{equation*}
\dim_\k \Ext^m_{\mod^T(\sU(\g))}(\tDelta_y, \Simp_x)
=[\bfZ^\gr_{y}\langle m-\ell(w_\circ) \rangle:\bfL^\gr_{\check{x}}]. 
\end{equation*}
\end{prop}

\begin{proof}
Fix $x,y \in W_\aff$ and $m \in \Z$.
By Lemma~\ref{lem:completion-Ext-inf-neighborhoods} and~\eqref{eqn:completion-tDelta-1} we have an isomorphism
\[
\Ext^m_{\mod^T(\sU(\g))}(\tDelta_y, \Simp_x) \cong \bigoplus_{w\in W}
\Ext^m_{\mod^T(\sU^{\widehat{0}}_{\widehat{0}})} \left( \hDelta_{wy} \langle y\bullet 0-wy\bullet 0 \rangle, \Simp_x \right).
\]
We claim that the summand in the right-hand side parametrized by $w$ can be nonzero only if $w=e$. 
In fact, by \eqref{eqn:decomp-Uhat0mod} it is nonzero only if $y\bullet 0-wy\bullet 0\in p\X$. 
Since $y\bullet 0-wy\bullet 0\in \Z\Phi$, if this condition is satisfied then we have $y\bullet 0-wy\bullet 0\in p\X \cap \Z\Phi = p\Z\Phi$, namely there exists $\nu \in \Z\Phi$ such that $wy\bullet 0 = y \bullet 0 + p\nu$, i.e.~$wy\bullet 0=t_\nu y\bullet 0$. Since $0$ is regular this indeed implies that $w=e$ (and $\nu=0$).

We have isomorphisms 
\begin{align*}
\Ext^m_{\mod^T(\sU^{\widehat{0}}_{\widehat{0}})}(\hDelta_y, \Simp_x) 
&\cong \Hom_{\Db\Coh^{T^{(1)}}(\widetilde{\g}^{(1)}_{\widehat{0}})}(\hM_y , (i_*\sL_x)_{\ho}[m]) \\ 
&\cong \Hom_{\Db\Coh^{T^{(1)}}(\widetilde{\g}^{(1)})}(\tM_y, i_* \sL_x[m]) \\ 
&\cong \Hom_{\Db\Coh^{T^{(1)}}(\widetilde{\sN}^{(1)})}({\sM}_y ,\sL_x[m]), 
\end{align*}
where the first isomorphism is induced by the equivalence $\gamma_{\ho}^{\ho}$ (and we use~\eqref{eqn:L-Groth}),
the second one follows from Lemma~\ref{lem:completion-Ext-inf-neighborhoods} (using~\eqref{eqn:hM-tM}), and the last one follows from adjunction and~\eqref{eqn:sM-tM}.

Now we have
\[
\Hom_{\Db\Coh^{T^{(1)}}(\widetilde{\sN}^{(1)})}({\sM}_y,\sL_x[m]) \cong
\bigoplus_{n \in \Z} \Hom_{\Db\Coh^{T^{(1)} \times \Gm}(\Spr^{(1)})}({\sM}_y^\gr \langle n \rangle,\sL^\gr_x [m]),
\]
and using Theorem~\ref{thm:KD} and Lemma~\ref{lem:Koszul-standards}, for any $n \in \Z$ we have isomorphisms  
\begin{multline*}
 \Hom_{\DGCoh^{T^{(1)} \times \Gm}(\Spr^{(1)})}({\sM}^\gr_y \langle n \rangle,\sL_x^\gr [m]) \\
\cong \Hom_{\DGCoh^{T^{(1)} \times \Gm}(\Groth^{(1)}\times^R_{{\g}^{*(1)}} \{0\})}(\sZ^\gr_{t_\rho y} \langle n \rangle,\sP_{t_\rho \check{x}}^\gr [m-n]) \\ 
\cong \Hom_{\DGCoh^{T^{(1)} \times \Gm}(\Groth^{(1)}\times^R_{{\g}^{*(1)}} \{0\})}(\sZ^\gr_{y} \langle n \rangle,\sP_{\check{x}}^\gr [m-n]).
\end{multline*}
We claim that the space
\[
\Hom_{\DGCoh^{T^{(1)} \times \Gm}(\Groth^{(1)}\times^R_{{\g}^{*(1)}} \{0\})}(\sZ^\gr_{y} \langle n \rangle,\sP_{\check{x}}^\gr [m-n])
\]
vanishes unless $m=n$. 
In fact, as e.g.~in the proof of Lemma~\ref{lem:Morita-equiv-graded}, this space is a direct summand in
\[
\Hom_{\DGCoh^{T^{(1)}}(\Groth^{(1)}\times^R_{{\g}^{*(1)}} \{0\})}(\sZ_{y},\sP_{\check{x}}[m-n]),
\]
which via the equivalence $\gamma_0^{\ho}$ identifies with
\[
\Hom_{\Db \Mod^T(\sU_0^{\ho})}(\bV_y, \Pro_{\check{x}}[m-n]).
\]
Since $\bV_y$ is a module and $\Pro_{\check{x}}$ is an injective module, this space is zero unless $m=n$. 

By the discussions above, we obtain an isomorphism
\[
\Ext^m_{\mod^T(\sU(\g))}(\tDelta_y, \Simp_x) \cong
\Hom_{\DGCoh^{T^{(1)} \times \Gm}(\Groth^{(1)}\times^R_{{\g}^{*(1)}} \{0\})}(\sZ^\gr_{y}\langle m \rangle,\sP_{\check{x}}^\gr).
\]
By Lemma~\ref{lem:Morita-equiv-graded}, the right-hand side identifies with the space $\Hom_{\mod^{p\X \times \Z}(\underline{\E})}(\bfZ_y^\gr \langle m+\ell(w_\circ) \rangle, \bfP_{\check{x}}^\gr)$. Here $\bfP_{\check{x}}^\gr$ is the injective hull of $\bfL_{\check{x}}^\gr \langle 2\ell(w_\circ) \rangle$ by Lemma~\ref{lem:socle-Px}, so this space has dimension $[\bfZ^\gr_y \langle m - \ell(w_\circ) \rangle: \bfL^\gr_{\check{x}}]$, which finishes the proof.
\end{proof}

\subsection{Conclusion of the proof} 

We are now almost ready to conclude the proof of Theorem~\ref{thm:main}. The last ingredient we need to introduce is the description of composition factors in the Loewy series of baby Verma modules, due to Andersen--Kaneda~\cite{ak}. (This result can also be deduced from~\cite[Proposition~18.19]{AJS94}.) We consider, for $x \in W_\aff$, the radical filtration $(\mathrm{rad}^m(\bfZ_x) : m \geq 0)$ of $\bfZ_x$.

\begin{prop}
\label{prop:loewy-filtration-bV}
For any $y,w \in W_\aff$ we have
\[
\sum_{m \geq 0} [\mathrm{rad}^m(\bfZ_w)/\mathrm{rad}^{m+1}(\bfZ_w) : \bfL_y] \cdot v^m = p_{w_\circ w, w_\circ y}.
\]
\end{prop}

\begin{proof}
In view of the equivalence~\eqref{eqn:Morita-equiv}, this statement is~\cite[Eqn.~(1) in~\S D.13]{Jan03}.
\end{proof}

\begin{corollary}
\label{cor:filtrations-rad-deg}
For any $x \in W_\aff$ and $m \in \Z$ we have
\[
\mathrm{rad}^m(\bfZ_x) = \bigoplus_{j \geq m} (\bfZ_x^\gr)^j,
\]
where in the right-hand side we consider the $\Z$-grading components.
\end{corollary}

\begin{proof}
In view of~\cite[Proposition~2.4.1]{BGS96}, we know that the radical filtration of $\bfZ_x$ coincides with the grading filtration $\left( \bigoplus_{j \geq m} (\bfZ_x^\gr)^j : m \geq 0 \right)$ up to a shift; the corollary amounts to saying that this shift is $0$. 

To prove this fact we will consider the multiplicity of the module $\bfL_{\check{x}}$. In fact, by Proposition~\ref{prop:dim-Ext-multiplicities} we have
\[
[\bfZ_x^\gr : \bfL_{\check{x}}^\gr \langle \ell(w_\circ) \rangle]=1;
\]
in particular, $\bfL_{\check{x}}$ is a composition factor of $(\bfZ_x^\gr)^{\ell(w_\circ)}$ (identified as a subquotient of $\bfZ_x$ via the grading filtration).
On the other hand, by~\cite[Lemma~4.21]{Soe97} we have
\[
p_{w_\circ x, w_\circ \check{x}} = v^{\ell(w_\circ)}
\]
since $\check{w_\circ \check{x}} = w_\circ x$. By Proposition~\ref{prop:loewy-filtration-bV}, this shows that $\bfL_{\check{x}}$ has multiplicity one in $\bfZ_x$, and that it occurs in $\mathrm{rad}^{\ell(w_\circ)}(\bfZ_w)/\mathrm{rad}^{\ell(w_\circ)+1}(\bfZ_w)$, which finishes the proof.
\end{proof}

\begin{proof}[Proof of Theorem~\ref{thm:main}]
By~\eqref{eqn:Ext-gps-Section-proof}, Proposition~\ref{prop:dim-Ext-multiplicities} and Corollary~\ref{cor:filtrations-rad-deg} we have
\begin{multline*}
\sum_{m \in \Z} \dim_\k \Ext^m_{\Mod(\g,B)}(\Simp_w, \nabla_y) \cdot v^m = \sum_{m \in \Z} [\bfZ^\gr_{y}\langle m-\ell(w_\circ) \rangle:\bfL^\gr_{\check{w}}] \cdot v^m \\
= \sum_{m \in \Z} [\mathrm{rad}^{\ell(w_\circ)-m}(\bfZ_y)/\mathrm{rad}^{\ell(w_\circ)-m+1}(\bfZ_y):\bfL_{\check{w}}] \cdot v^m
\end{multline*}
(where, by convention, $\mathrm{rad}^n(\bfZ_y)=\bfZ_y$ if $n \leq 0$).
Using Proposition~\ref{prop:loewy-filtration-bV} we deduce that
\[
\sum_{m \geq 0} \dim_\k \Ext^m_{\Mod(\g,B)}(\Simp_w, \nabla_y) \cdot v^m = v^{\ell(w_\circ)} \cdot p_{w_\circ y,w_\circ \check{w}}(v^{-1}).
\]
Finally, by an equality in~\cite[Proof of Theorem~6.1(2) on p.~105]{Soe97}, the right-hand side identifies with $p_{y,w}$, which finishes the proof.
\end{proof}

\begin{rmk}
One can phrase the arguments above in a slightly different way as follows. Using~\cite[\S 2.5]{BGS96} one can show that $\underline{\E}$ is isomorphic to the graded ring $A$ of~\cite[\S 18.17]{AJS94} (up to the identification of $\X$ with $p\X$), and then replace the reference to Proposition~\ref{prop:loewy-filtration-bV} by a reference to~\cite[Proposition~18.19(c)]{AJS94}.
\end{rmk}

\appendix

\section{Equivariant sheaves on completions}
\label{app:equiv-sheaves} 

Given a scheme over some base equipped with an action of a flat affine group scheme, there is a convenient notion of equivariant quasi-coherent sheaves, reviewed in detail in~\cite[Appendix]{MR16}.
In this section, we explain how to adapt this definition for some ``formal completions'' of schemes.
(We insist that formal completion is used with quotation marks; all the geometric objects considered below are \emph{schemes} and not formal schemes; see Remark~\ref{rmk:IndMod}\eqref{it:comments-formal-schemes} for comments on these aspects.)

\subsection{Setting} 
\label{ss:app-setting}

We fix a commutative ring $k$. (For the applications in the body of the paper, the case when $k$ is an algebraically closed field would suffice; we treat the case of general base rings because it does not really require more work.) Below, undecorated tensor products (resp.~fiber products) are taken over $k$ (resp.~over $\Spec(k)$).

Let $R$ be a commutative $k$-algebra and $I\subset R$ be an ideal.
For a scheme $X$ over $R$, we abbreviate 
\[
X_{\widehat{I}}:=X\times_{\Spec (R)} \Spec (R_{\widehat{I}}).
\]
(See~\S\ref{sss:rings-modules} for the notation $R_{\widehat{I}}$.)
If $R$ is noetherian, then $R_{\widehat{I}}$ is also noetherian, see~\cite[\href{https://stacks.math.columbia.edu/tag/05GH}{Tag 05GH}]{stacks-project}. In this case, the completion $\widehat{I}$ of $I$ as an $R$-module identifies with the ideal in $R_{\widehat{I}}$ generated by $I$, see~\cite[\href{https://stacks.math.columbia.edu/tag/05GG}{Tag 05GG}]{stacks-project}. Moreover, if $X$ is of finite type over $R$, then $X_{\widehat{I}}$ is of finite type over $R_{\widehat{I}}$, hence is noetherian too, see~\cite[\href{https://stacks.math.columbia.edu/tag/01T6}{Tag 01T6}]{stacks-project}.

Let now $H$ be a flat affine group scheme over $k$. 
Assume that $R$ is endowed with an action of $H$ compatible with the ring structure, and that $I\subset R$ is $H$-stable. 
Assume also that $X$ is equipped with an action of $H$ (as a scheme over $k$), and that the structure morphism 
$X\rightarrow \Spec (R)$
is $H$-equivariant.
We have action and projection maps $H\times X \rightrightarrows X$ which are compatible in the obvious way with the action and projection maps $H\times \Spec(R) \rightrightarrows \Spec(R)$. Each of the latter maps restricts to a morphism from $H\times \Spec(R/I)$ to $\Spec(R/I)$, hence they induce natural maps
\[
\Spec \left( (\sO(H) \otimes R)_{\widehat{\sO(H)\otimes I}} \right) \rightrightarrows \Spec \left( R_{\widehat{I}} \right).
\]
Note that $\sO(H)\otimes I$ identifies with an ideal in $\sO(H)\otimes R$ because of our flatness assumption on $H$. Combining these constructions, we deduce canonical morphisms
\[
\act, \pr: (H\times X)_{\widehat{\sO(H)\otimes I}} \rightrightarrows X_{\widehat{I}},
\]
where $H\times X$ is considered as a scheme over $\sO(H)\otimes R$ in the obvious way.

\begin{rmk}
Let us insist that the action of $H$ on $X$ does \emph{not} induce an action on $X_{\widehat{I}}$, since completion does not commute with tensor products; the maps $\act$ and $\pr$ will be used as a substitute for such an action.
\end{rmk}

\begin{lem}
\label{lem:flatness}
Let $R$, $I$, $X$, $H$ be as above.
Assume that $R$ is noetherian and that $H$ is of finite type over $k$. Then
the morphisms $\act$ and $\pr$ are flat.
\end{lem}

\begin{proof}
The isomorphism $H \times X \simto H \times X$ defined by $(h,x) \mapsto (h,h \cdot x)$ intertwines the projection and action maps to $X$. We have a similar isomorphism for $H \times \Spec(R)$, which stabilizes the closed subscheme $H \times \Spec(R/I)$, hence induces an automorphism of $(\sO(H) \otimes R)_{\widehat{\sO(H) \otimes I}}$. Combining these maps we deduce an automorphism of $(H\times X)_{\widehat{\sO(H)\otimes I}}$ intertwining the maps $\act$ and $\pr$, which reduces the proof to the case of $\pr$.

In this case we have
\begin{multline*}
(H\times X)_{\widehat{\sO(H)\otimes I}} = (H\times X) \times_{\Spec(\sO(H) \otimes R)} \Spec \left( (\sO(H) \otimes R)_{\widehat{\sO(H) \otimes I}} \right) \\
= (H\times X) \times_{\Spec(\sO(H) \otimes R)} \Spec \left( \sO(H) \otimes R_{\widehat{I}} \right) \times_{\Spec(\sO(H) \otimes R_{\widehat{I}})} \Spec \left( (\sO(H) \otimes R)_{\widehat{\sO(H) \otimes I}} \right) \\
= \left( H\times X_{\widehat{I}} \right) \times_{\Spec(\sO(H) \otimes R_{\widehat{I}})} \Spec \left( (\sO(H) \otimes R)_{\widehat{\sO(H) \otimes I}} \right).
\end{multline*}
Here the ring $\sO(H) \otimes R_{\widehat{I}}$ is noetherian because it is finitely generated over the noetherian ring $R_{\widehat{I}}$, and $(\sO(H) \otimes R)_{\widehat{\sO(H) \otimes I}}$ identifies with the completion of that ring with respect to the ideal $\sO(H) \otimes \widehat{I}$. Hence the map
\[
\Spec \left( (\sO(H) \otimes R)_{\widehat{\sO(H) \otimes I}} \right) \to \Spec \left( \sO(H) \otimes R_{\widehat{I}} \right)
\]
is flat, see~\cite[\href{https://stacks.math.columbia.edu/tag/00MB}{Tag 00MB}]{stacks-project}. It follows that the natural morphism
\[
(H\times X)_{\widehat{\sO(H)\otimes I}} \to H\times X_{\widehat{I}}
\]
is flat, and then that so is its composition with projection to $X_{\widehat{I}}$, which coincides with $\pr$.
%
\end{proof}

\subsection{Equivariant quasi-coherent sheaves} 
\label{ss:def-equiv-sheaves}

We continue with the data above, assuming from now on that $R$ is noetherian and that $H$ is of finite type over $k$. 
The following definition is used implictly in~\cite{BM13}, and is made more explicit in~\cite{Tan21}. Here we denote by
\[
p_{2,3}, \, \id \times \act, \, m \times \id : (H\times H\times X)_{(\widehat{\sO(H)\otimes\sO(H)\otimes I})} \to (H\times X)_{\widehat{\sO(H)\otimes I}}
\]
the maps induced by projection on the second and third factors, by the action of $H$ on $X$ (on the second and third factors), and by multiplication in $H$ respectively.

\begin{definition}
\label{def:equiv-coh-sheaves} 
An $H$-equivariant quasi-coherent 
sheaf on $X_{\widehat{I}}$ is a pair $(\sF,\theta)$ where $\sF$ is a quasi-coherent 
sheaf on $X_{\widehat{I}}$ and $\theta$ is an isomorphism $\theta: \act^*\sF \simto \pr^*\sF$ satisfying the usual cocycle condition 
\[
\bigl( (p_{2,3})^* \theta \bigr) \circ \bigl( (\id \times \act)^* \theta \bigr) = (m \times \id)^* \theta
\]
on $(H\times H\times X)_{(\widehat{\sO(H)\otimes\sO(H)\otimes I})}$.

If $(\sF,\theta)$ and $(\sF',\theta')$ are $H$-equivariant quasi-coherent sheaves on $X_{\widehat{I}}$, then a morphism from $(\sF,\theta)$ to $(\sF',\theta')$ is a morphism of quasi-coherent sheaves $\varphi : \sF \to \sF'$ such that $(\pr^* \varphi) \circ \theta = \theta' \circ (\act^* \varphi)$.
\end{definition}

It is clear that the data in Definition~\ref{def:equiv-coh-sheaves} define a category of $H$-equivariant quasi-coherent sheaves on $X_{\widehat{I}}$, which will be denoted $\QCoh^H(X_{\widehat{I}})$. Lemma~\ref{lem:flatness} guarantees that if $\varphi : (\sF,\theta) \to (\sF',\theta')$ is a morphism of $H$-equivariant quasi-coherent sheaves on $X_{\widehat{I}}$, then the kernel and cokernel of $\varphi$, considered as a morphism of quasi-coherent sheaves, admit canonical structures of $H$-equivariant quasi-coherent sheaves. As a consequence, the category $\QCoh^H(X_{\widehat{I}})$ is abelian. It is also equipped with a natural symmetric monoidal structure, with monoidal product given by tensor product over $\sO_{X_{\widehat{I}}}$. By definition we have a canonical exact, faithful and monoidal ``forgetful'' functor
\[
\for : \QCoh^H(X_{\widehat{I}}) \to \QCoh(X_{\widehat{I}}).
\]
Recall that a colimit of quasi-coherent sheaves is always quasi-coherent, see~\cite[\href{https://stacks.math.columbia.edu/tag/01LA}{Tag 01LA}]{stacks-project}, and that pullback commutes with colimits (because it has a right adjoint). Hence the category $\QCoh^H(X_{\widehat{I}})$ admits colimits, and the functor $\for$ commutes with them.

By compatibility of pullback functors with composition, pullback under the (flat) projection map $X_{\widehat{I}} \to X$ induces a monoidal exact functor
\begin{equation}
\label{eqn:QCoh-pullback-completion}
\QCoh^H(X) \to \QCoh^H(X_{\widehat{I}}),
\end{equation}
where in the left-hand side we consider the usual category of $H$-equivariant quasi-coherent sheaves on $X$. For any 
$V \in \Rep(H)$ one can consider the equivariant quasi-coherent sheaf $V \otimes \sO_X$ on $X$. By pullback we deduce that $V \otimes \sO_{X_{\widehat{I}}}$ defines an object in $\QCoh^H(X_{\widehat{I}})$. For $\sF \in \QCoh^H(X_{\widehat{I}})$, we will usually write $V \otimes \sF$ for $(V \otimes \sO_{X_{\widehat{I}}}) \otimes_{\sO_{X_{\widehat{I}}}} \sF$.

\begin{rmk}
\label{rmk:QCohX-over-quotient}
Assume that the morphism $X \to \Spec(R)$ factors through $\Spec(R/I^n)$, for some $n \geq 0$. Then we have
\begin{multline*}
X_{\widehat{I}} = X \times_{\Spec(R)} \Spec(R_{\widehat{I}}) = X \times_{\Spec(R/I^n)} \Spec(R/I^n) \times_{\Spec(R)} \Spec(R_{\widehat{I}}) \\
= X \times_{\Spec(R/I^n)} \Spec(R_{\widehat{I}} / (I^n \cdot R_{\widehat{I}})) = X,
\end{multline*}
where the last identification uses~\cite[\href{https://stacks.math.columbia.edu/tag/05GG}{Tag 05GG}]{stacks-project}. We have similar identifications for $H \times X$ and $H\times H\times X$, so that in this case~\eqref{eqn:QCoh-pullback-completion} is an equivalence of categories.
\end{rmk}


More generally, if $\sC$ is a quasi-coherent sheaf of $\sO_{X_{\widehat{I}}}$-algebras equipped with a compatible structure of $H$-equivariant quasi-coherent sheaf, we define an $H$-equivariant $\sC$-module as an $H$-equivariant quasi-coherent sheaf $\sM$ equipped with a $\sC$-action such that the structure map $\sC \otimes_{\sO_{X_{\widehat{I}}}} \sM\rightarrow \sM$ is a morphism of $H$-equivariant quasi-coherent sheaves. These objects naturally form an abelian category, which will be denoted $\Mod^H(\sC)$. There exists a canonical exact faithful functor 
\begin{equation}
\label{eqn:forget-action-C}
\Mod^H(\sC) \to \QCoh^H(X_{\widehat{I}})
\end{equation}
and a tensor product bifunctor
\[
(-) \otimes_\sC (-) : \Mod^H(\sC^{\op}) \times \Mod^H(\sC) \to \QCoh^H(X_{\widehat{I}})
\] 
where $\sC^{\op}$ is the opposite sheaf of algebras.

\begin{exm}
\label{ex:equiv-qcoh-alg}
If $\sA$ is an $H$-equivariant quasi-coherent sheaf of $\sO_X$-algebras, then its pullback ${\sA}_{\widehat{I}}$ to $X_{\widehat{I}}$ satisfies the conditions above. Moreover, as for~\eqref{eqn:QCoh-pullback-completion}, pullback under the canonical morphism $X_{\widehat{I}} \to X$ induces an exact functor
\begin{equation}
\label{eqn:pullback-completion}
\Mod^H(\sA) \to \Mod^H(\sA_{\widehat{I}}).
\end{equation}
As in Remark~\ref{rmk:QCohX-over-quotient}, in case the morphism $X \to \Spec(R)$ factors through $\Spec(R/I^n)$ for some $n$, this functor is an equivalence.
\end{exm}

\subsection{Infinitesimal action} 
\label{ss:infinitesimal}

We continue with the assumptions of~\S\ref{ss:def-equiv-sheaves}.
In the usual theory of equivariant quasi-coherent sheaves on schemes, any sheaf equivariant for the action of a flat affine group scheme acquires a canonical action of its Lie algebra, see e.g.~\cite[\S 3.1]{kashiwara}. Here we explain the counterpart of this construction for our notion of equivariance.


Let $I_e\subset \sO(H)$ be the ideal of functions vanishing on the unit $e\in H(k)$, and denote by $\h$ the Lie algebra of $H$.
Recall that by definition we have $\mathfrak{h}=\Hom_k(I_e/I^2_e,k)$.
We consider the closed subscheme $H_{(1)}:=\Spec(\sO(H)/I_e^2)$ of $H$. (This subscheme should not be confused with the Frobenius kernel $H_1$ in case $k$ is a field of characteristic $p$.) 
The action and projection morphisms $(H\times X)_{\widehat{\sO(H)\otimes I}} \rightrightarrows X_{\widehat{I}}$ induce morphisms 
\[
\act_{(1)}, \pr_{(1)}: H_{(1)}\times X_{\widehat{I}} \cong (H_{(1)}\times X)_{\widehat{\sO(H_{(1)})\otimes I}} \rightrightarrows X_{\widehat{I}}. 
\]
Note that the unit morphism $\Spec(k)\rightarrow H$ factors through $H_{(1)}$ and induces a homeomorphism of topological spaces $X_{\widehat{I}}\rightarrow H_{(1)}\times X_{\widehat{I}}$. 
Hence the morphism $\act_{(1)}:H_{(1)}\times X_{\widehat{I}}\rightarrow X_{\widehat{I}}$ is a homeomorphism, whose restriction to any open subscheme $U\subset X_{\widehat{I}}$ gives a morphism $H_{(1)}\times U\rightarrow U$. 

Let $(\sF,\theta)\in \QCoh^H(X_{\widehat{I}})$. 
The isomorphism $\theta: \act^*\sF \xs \pr^*\sF$ restricts to an isomorphism $\theta_{(1)}: \act^*_{(1)}\sF \xs \pr^*_{(1)}\sF$. 
For any open subscheme $U\subset X_{\widehat{I}}$, we consider the composition 
\begin{align*}
\sigma^\#: \Gamma(U,\sF)\xrightarrow{\act^*_{(1)}}  \Gamma(H_{(1)}\times U, \act^*_{(1)} &\sF) \xrightarrow[\sim]{\theta_{(1)}} \Gamma(H_{(1)}\times U, \pr^*_{(1)}\sF)\\ 
&\cong \sO(H_{(1)})\otimes \Gamma(U,\sF) \twoheadrightarrow I_e/I^2_e \otimes \Gamma(U,\sF), 
\end{align*}
where the last map is induced by the projection morphism associated with the canonical decomposition $\sO(H_{(1)})\cong k \oplus I_e/I^2_e$.
We obtain the following map induced by $-\sigma^\#$:
\[
\sigma: \mathfrak{h} \rightarrow \End(\Gamma(U,\sF)).
\]
One can show that $\sigma$ is a Lie algebra homomorphism.
It thus defines an action of $\h$ on $\sF$.

\subsection{Equivariant coherent sheaves and colimits of such} 
\label{ss:equiv-coh-sheaves}

In this subsection we assume (in addition to the conditions above) that $X$ is of finite type over $R$; then (as explained in~\S\ref{ss:app-setting}) $X_{\widehat{I}}$ is a noetherian scheme. We define an $H$-equivariant coherent sheaf on $X_{\widehat{I}}$ to be an $H$-equivariant quasi-coherent sheaf $(\sF,\theta)$ such that $\sF$ is coherent. The abelian full subcategory of $\QCoh^H(X_{\widehat{I}})$ whose objects are the $H$-equivariant coherent sheaves will be denoted $\Coh^H(X_{\widehat{I}})$.

A difficulty with this setting is that it is not clear (and probably not true in general) that any object of $\QCoh^H(X_{\widehat{I}})$ is the union of its subobjects which are coherent. To bypass this difficulty, we define
\[
\IndCoh^H(X_{\widehat{I}})
\]
as the full subcategory of $\QCoh^H(X_{\widehat{I}})$ whose objects are the $H$-equivariant quasicoherent sheaves that are the colimit of their subobjects belonging to $\Coh^H(X_{\widehat{I}})$. Equivalently, an object in $\QCoh^H(X_{\widehat{I}})$ belongs to $\IndCoh^H(X_{\widehat{I}})$ if and only if it is a filtered colimit of objects in $\Coh^H(X_{\widehat{I}})$. The embedding $\IndCoh^H(X_{\widehat{I}}) \to \QCoh^H(X_{\widehat{I}})$ admits a right adjoint, which sends an object to the colimit of its coherent subobjects. It is clear that $\Coh^H(X_{\widehat{I}})$ and $\IndCoh^H(X_{\widehat{I}})$ are monoidal subcategories of $\QCoh^H(X_{\widehat{I}})$ (with respect to tensor product).

Similarly, given an $H$-equivariant coherent sheaf of $\sO_{X_{\widehat{I}}}$-algebras $\sC$, the full subcategory of $\Mod^H(\sC)$ whose objects are coherent over $\sO_{X_{\widehat{I}}}$ will be denoted $\mod^H(\sC)$.
We also define the full subcategory
\[
\IndMod^H(\sC)
\]
of $\Mod^H(\sC)$ as consisting of objects which are the colimit of their coherent subobjects.
(In fact, it is easily seen that an object of $\Mod^H(\sC)$ belongs to $\IndMod^H(\sC)$ if and only if its image in $\QCoh^H(X_{\widehat{I}})$ belongs to $\IndCoh^H(X_{\widehat{I}})$.)
This subcategory is closed under taking subobjects and quotients, hence is an abelian subcategory. 
Of course it contains $\mod^H(\sC)$; moreover since the latter is a Serre subcategory in $\mod^H(\sC)$, it is also a Serre subcategory in $\IndMod^H(\sC)$. 
By \cite[\href{https://stacks.math.columbia.edu/tag/0FCL}{Tag 0FCL}]{stacks-project}, the natural functor 
\[
D^- \mod^H(\sC) \rightarrow D^- \IndMod^H(\sC) 
\]
induces an equivalence of categories between $D^-\mod^H(\sC)$ and the full subcategory in $D^- \IndMod^H(\sC)$ consisting of complexes all of whose cohomology objects are in $\mod^H(\sC)$. 

\begin{rmk}
\phantomsection
\label{rmk:IndMod}
\begin{enumerate}
\item
It is not clear to us if the subcategory $\IndMod^H(\sC) \subset \Mod^H(\sC)$ is stable under extensions.
\item
\label{it:IndMod-completion}
Consider the setting of Example~\ref{ex:equiv-qcoh-alg}, with our current additional assumption that $X$ is of finite type over $R$. Then the functor~\eqref{eqn:pullback-completion} restricts to a functor from $\mod^H(\sA)$ to $\mod^H(\sA_{\widehat{I}})$. Assuming in addition that $\sA$ is coherent, using the fact that any $H$-equivariant quasi-coherent sheaf on $X$ is the colimit of its coherent $H$-equivariant submodules (see~\cite[Lemma~1.4]{thomason}) one sees that a similar property holds for $\sA$-modules, which implies that the functor~\eqref{eqn:pullback-completion} takes values in $\IndMod^H(\sA_{\widehat{I}})$.
\item
\label{it:comments-formal-schemes}
Assume that the structure morphism $X \to \Spec(R)$ is projective (hence, in particular, of finite type). In this case, by~\cite[Corollaire~5.1.6]{ega3}, the datum of a coherent sheaf on $X_{\widehat{I}}$ is equivalent to the datum of a projective system $(\sF_n : n \geq 1)$ where each $\sF_n$ is a coherent sheaf on the scheme $X \times_{\Spec(R)} \Spec(R/I^n)$ and for $m \leq n$ the restriction of $\sF_n$ to $X \times_{\Spec(R)} \Spec(R/I^m)$ coincides with $\sF_m$. A similar comment applies to coherent sheaves on $(H\times X)_{\widehat{\sO(H)\otimes I}}$. Moreover, in these terms the pullback functors
\[
\pr^*, \act^* : \Coh(X_{\widehat{I}}) \to \Coh((H\times X)_{\widehat{\sO(H)\otimes I}})
\]
are induced by the pullback functors associated with the induced morphisms
\[
H \times \left( X \times_{\Spec(R)} \Spec(R/I^n) \right) \to X \times_{\Spec(R)} \Spec(R/I^n).
\]
As a consequence, the datum of an $H$-equivariant coherent sheaf on $X_{\widehat{I}}$ is equivalent to the datum of a projective system $(\sF_n : n \geq 1)$ where each $\sF_n$ is an $H$-equivariant coherent sheaf on the scheme $X \times_{\Spec(R)} \Spec(R/I^n)$ (in the usual sense) and for $m \leq n$ the restriction of $\sF_n$ to $X \times_{\Spec(R)} \Spec(R/I^m)$ coincides with $\sF_m$.
\end{enumerate}
\end{rmk}



\subsection{Functorialities}
\label{ss:functorialities}

We continue with the setting of~\S\ref{ss:def-equiv-sheaves}.
%
Let $Y$ be another scheme over $R$ equipped with an action of $H$ such that the structure morphism $Y \to \Spec(R)$ is $H$-equivariant, and let $f: Y\rightarrow X$ be an $H$-equivariant morphism of $R$-schemes. Then we have a natural morphism of schemes
\[
f_{\widehat{I}} : Y_{\widehat{I}} \to X_{\widehat{I}},
\]
and commutative diagrams
\begin{equation}
\label{eqn:diagrams-morphisms-completions}
\begin{tikzcd}
(H\times Y)_{\widehat{\sO(H)\otimes I}} \arrow[r] \arrow[d] & (H\times X)_{\widehat{\sO(H)\otimes I}} \arrow[d] \\ 
Y_{\widehat{I}} \arrow[r] & X_{\widehat{I}}
\end{tikzcd}
\end{equation}
where the horizontal arrows are induced by $f$ and the vertical ones are either the maps $\pr$ or the maps $\act$. Using compatibility of pullback functors with composition, we obtain a natural pullback functor
\[
(f_{\widehat{I}})^* : \QCoh^H(X_{\widehat{I}}) \to \QCoh^H(Y_{\widehat{I}})
\]
which is compatible in the natural way with its nonequivariant variant and with the pullback functor $f^* : \QCoh^H(X) \to \QCoh^H(Y)$. In case $X$ and $Y$ are of finite type over $R$, the functor $(f_{\widehat{I}})^*$ restricts to functors
\[
\Coh^H(X_{\widehat{I}}) \to \Coh^H(Y_{\widehat{I}}), \quad
\IndCoh^H(X_{\widehat{I}}) \to \IndCoh^H(Y_{\widehat{I}}).
\]

Now, assume that $f$ is quasi-compact and quasi-separated. Then so is $f_{\widehat{I}}$, hence we have a pushforward functor
\[
(f_{\widehat{I}})_* : \QCoh(Y_{\widehat{I}}) \to \QCoh(X_{\widehat{I}}),
\]
see~\cite[\href{https://stacks.math.columbia.edu/tag/01LC}{Tag 01LC}]{stacks-project}. Using once again the commutative diagrams~\eqref{eqn:diagrams-morphisms-completions} and the flat base change theorem (see~\cite[\href{https://stacks.math.columbia.edu/tag/02KH}{Tag 02KH}]{stacks-project}) we obtain that this functor induces a functor
\[
 \QCoh^H(Y_{\widehat{I}}) \to \QCoh^H(X_{\widehat{I}}),
\]
which will also be denoted $(f_{\widehat{I}})_*$. It is clear that $(f_{\widehat{I}})_*$ is right adjoint to $(f_{\widehat{I}})^*$. 

Assume now that $X$ and $Y$ are of finite type over $R$. In general the functor $(f_{\widehat{I}})_*$ might not send $\IndCoh^H(Y_{\widehat{I}})$ into $\IndCoh^H(X_{\widehat{I}})$. However the functor $(f_{\widehat{I}})^* : \IndCoh^H(X_{\widehat{I}}) \to \IndCoh^H(Y_{\widehat{I}})$ still has a right adjoint, given by the composition of the restriction of $(f_{\widehat{I}})_*$ to $\IndCoh^H(Y_{\widehat{I}})$ with the right adjoint of the embedding $\IndCoh^H(X_{\widehat{I}}) \to \QCoh^H(X_{\widehat{I}})$, see~\S\ref{ss:equiv-coh-sheaves}.
In case $f$ is proper the functor $(f_{\widehat{I}})_*$ sends coherent sheaves to coherent sheaves, see~\cite[\href{https://stacks.math.columbia.edu/tag/02O5}{Tag 02O5}]{stacks-project}. Since this functor commutes with colimits (see e.g.~the discussion in~\cite[\S 3.9.3]{lipman}), in this case it therefore also sends $\IndCoh^H(Y_{\widehat{I}})$ into $\IndCoh^H(X_{\widehat{I}})$.


More generally, let $\sC_X$, resp. $\sC_Y$, be an $H$-equivariant quasi-coherent sheaf of $\sO_{X_{\widehat{I}}}$-algebras, resp. $\sO_{Y_{\widehat{I}}}$-algebras. 
Suppose we are also given an $H$-equivariant morphism of quasi-coherent sheaves of algebras 
$(f_{\widehat{I}})^{*}\sC_X \rightarrow \sC_Y$. 
Then we have a natural pullback functor
$(f_{\widehat{I}})^{*} : \Mod^H(\sC_X) \to \Mod^H(\sC_Y)$. Note that, contrary to what the (slightly abusive) notation might suggest, this functor does not only depend on the morphism $f_{\widehat{I}}$, but also on the given morphism $(f_{\widehat{I}})^{*}\sC_X \rightarrow \sC_Y$. In case this morphism is an isomorphism, the following diagram commutes, where the lower line is as above and the vertical arrows are the forgetful functors (see~\eqref{eqn:forget-action-C}):
\[
\begin{tikzcd}[column sep=large]
\Mod^H(\sC_X) \ar[r, "(f_{\widehat{I}})^{*}"] \ar[d] & \Mod^H(\sC_Y) \ar[d] \\
\QCoh^H(X_{\widehat{I}}) \ar[r, "(f_{\widehat{I}})^{*}"] & \QCoh^H(Y_{\widehat{I}}).
\end{tikzcd}
\]

If $f$ is quasi-compact and quasi-separated, we have a pushforward functor $(f_{\widehat{I}})_{*} : \Mod^H(\sC_Y) \to \Mod^H(\sC_X)$ which is right adjoint to $(f_{\widehat{I}})^{*}$, and such that the following diagram commutes, where the lower line is as above and the vertical arrows are the forgetful functors (see~\eqref{eqn:forget-action-C}):
\[
\begin{tikzcd}[column sep=large]
\Mod^H(\sC_Y) \ar[r, "(f_{\widehat{I}})_{*}"] \ar[d] & \Mod^H(\sC_X) \ar[d] \\
\QCoh^H(Y_{\widehat{I}}) \ar[r, "(f_{\widehat{I}})_{*}"] & \QCoh^H(X_{\widehat{I}}).
\end{tikzcd}
\]

If $X$, $Y$ are of finite type over $R$ and $\sC_X$, $\sC_Y$ are coherent, then $(f_{\widehat{I}})^{*}$ restricts to functors $\mod^H(\sC_X) \to \mod^H(\sC_Y)$ and $\IndMod^H(\sC_X) \to \IndMod^H(\sC_Y)$. If $f$ is moreover proper, the functor $(f_{\widehat{I}})_{*}$ restricts to functors $\mod^H(\sC_Y) \to \mod^H(\sC_X)$ and $\IndMod^H(\sC_Y) \to \IndMod^H(\sC_X)$.


\begin{rmk}
\label{rmk:noetherian-morphisms}
In practice all the schemes we will encounter will be noetherian. Recall that any noetherian scheme is quasi-separated, see~\cite[\href{https://stacks.math.columbia.edu/tag/01OY}{Tag 01OY}]{stacks-project}, hence any morphism between such schemes is quasi-separated, see~\cite[\href{https://stacks.math.columbia.edu/tag/01KV}{Tag 01KV}]{stacks-project}. Similarly, morphisms between noetherian schemes are automatically quasi-compact, see~\cite[\href{https://stacks.math.columbia.edu/tag/01P0}{Tag 01P0}]{stacks-project}.
\end{rmk}


\begin{exm}
\label{ex:infinitesimal-neighborhoods}
Consider a scheme $X$ as above, choose some $n \geq 1$, and set $Y = X \times_{\Spec(R)} \Spec(R/I^n)$. Then by Remark~\ref{rmk:QCohX-over-quotient} we have $Y_{\widehat{I}}=Y$, and $\QCoh^H(Y_{\widehat{I}})$ identifies with the familiar category $\QCoh^H(Y)$. The morphism $Y_{\widehat{I}} \to X_{\widehat{I}}$ is a closed immersion, and the associated pushforward functor $\QCoh^H(X \times_{\Spec(R)} \Spec(R/I^n)) \to \QCoh^H(X_{\widehat{I}})$ is fully faithful and exact, with essential image consisting of objects annihilated by $I^n$. It admits as left adjoint the functor $\QCoh^H(X_{\widehat{I}}) \to \QCoh^H(X \times_{\Spec(R)} \Spec(R/I^n))$ given by restriction.

More generally, if $\sC$ is an $H$-equivariant quasi-coherent sheaf of algebras on $X_{\widehat{I}}$, its pullback $\sC_n$ to $X \times_{\Spec(R)} \Spec(R/I^n)$ is an $H$-equivariant quasi-coherent sheaf of algebras in the usual sense; we have a fully faithful and exact pushforward functor $\Mod^H(\sC_n) \to \Mod^H(\sC)$, which admits as left adjoint the pullback functor $\Mod^H(\sC) \to \Mod^H(\sC_n)$.
\end{exm}

\subsection{Invariants} 
\label{ss:app-invariants}



In this subsection we consider the case $X=\Spec(R)$. In this case, given an $H$-equivariant coherent sheaf of algebras $\sC$, with global sections $C$, we will write
$\Mod^H(C)$, resp.~$\IndMof^H(C)$, resp.~$\mod^H(C)$,
for $\Mod^H(\sC)$, resp. $\IndMod^H(\sC)$, resp.~$\mod^H(\sC)$. 

We first consider the case $\sC=\sO_{X_{\widehat{I}}}$, so that $C=R_{\widehat{I}}$.
An object of $\Mod^H(R_{\widehat{I}})$ is an $R_{\widehat{I}}$-module $M$ endowed with an isomorphism of $(\sO(H)\otimes R)_{\widehat{\sO(H)\otimes I}}$-modules
\begin{equation}
\label{eqn:equiv-structure-affine}
(\sO(H)\otimes R)_{\widehat{\sO(H)\otimes I}} \otimes_{R_{\widehat{I}}} M \simto (\sO(H)\otimes R)_{\widehat{\sO(H)\otimes I}} \otimes_{R_{\widehat{I}}} M
\end{equation}
satisfying an appropriate cocycle condition; here, in the left-hand side the morphism $R_{\widehat{I}} \to (\sO(H)\otimes R)_{\widehat{\sO(H)\otimes I}}$ is induced by the coaction morphism $R \to \sO(H) \otimes R$, and in the right-hand side it is induced by the obvious morphism (sending $r$ to $1 \otimes r$). In particular, in this setting we have a ``coaction morphism''
\[
\Delta_M: M\rightarrow (\sO(H)\otimes R)_{\widehat{\sO(H)\otimes I}} \otimes_{R_{\widehat{I}}} M
\]
(where in the right-hand side the tensor product is taken with respect to the ``obvious'' morphism $R_{\widehat{I}} \to (\sO(H)\otimes R)_{\widehat{\sO(H)\otimes I}}$)
defined as the composition of the obvious morphism $M \to (\sO(H)\otimes R)_{\widehat{\sO(H)\otimes I}} \otimes_{R_{\widehat{I}}} M$ (sending $m$ to $1 \otimes m$) with~\eqref{eqn:equiv-structure-affine}.

Given $M\in \Mod^H(R_{\widehat{I}})$, we set
\[
M^H = \{m \in M \mid \Delta_M(m)=1 \otimes m\}.
\]
It is clear that the assignment $M \mapsto M^H$ defines a left exact functor
\[
(-)^H : \Mod^H(R_{\widehat{I}})\rightarrow \Mod(k).
\]
In case $M$ is finitely generated over $R_{\widehat{I}}$, we have canonical identifications
\[
M=\varprojlim_n M/I^n M, \quad (\sO(H)\otimes R)_{\widehat{\sO(H)\otimes I}} \otimes_{R_{\widehat{I}}} M = \varprojlim_n \, (\sO(H) \otimes (M/I^n M)),
\]
see~\cite[\href{https://stacks.math.columbia.edu/tag/00MA}{Tag 00MA}]{stacks-project} (applied to the noetherian rings $R_{\widehat{I}}$ and $(\sO(H)\otimes R)_{\widehat{\sO(H)\otimes I}}$). Since taking projective limits is left exact, we deduce that we have
\begin{equation}
\label{eqn:invariants-fg-module}
M^H= \varprojlim_n \, (M/I^n M)^H
\end{equation}
where in the right-hand side we use the fixed points functor for representations of an affine group scheme, as in~\cite[\S I.2.10]{Jan03}.

\begin{rmk}
\label{rmk:inv-colimits}
It is easily checked that the functor $(-)^H$ commutes with filtered colimits. (See~\cite[Lemma~I.4.17]{Jan03} for the similar property in the usual setting of $H$-modules.)
\end{rmk}




Let now $\Lambda$ be a finitely generated abelian group, and assume that $H=\mathrm{Diag}_k(\Lambda)$ is the associated diagonalizable group scheme, see~\S\ref{sss:grading-shift}.

\begin{lem}
\label{lem:exactness-fixedpts-diag}
If $H=\mathrm{Diag}_k(\Lambda)$ as above, then the functor 
\[
(-)^H : \IndMof^H(R_{\widehat{I}}) \rightarrow \Mod(k)
\]
is exact. 
\end{lem}

\begin{proof}
We have to show that the functor preserves surjectivity. 
We therefore consider a surjection $M \twoheadrightarrow N$ of modules in $\IndMof^H(R_{\widehat{I}})$; then we need to show that any $x \in N^H$ belongs to the image of the induced morphism $M^H \to N^H$. 
Replacing $M$ by a finitely generated $H$-equivariant submodule whose image contains $x$, and $N$ by its image, we can assume that $M$ and $N$ are finitely generated.
Then $M^H$ and $N^H$ can be described as in~\eqref{eqn:invariants-fg-module}. In particular $x$ corresponds to a compatible sequence $(x_n : n \geq 1)$ of elements $x_n \in (N/I^n N)^H$. 
Each $N/I^n N$ is an $H$-module, i.e.~admits a canonical $\Lambda$-grading, and $x_n$ belongs to the component of $N/I^n N$ corresponding to the neutral element $0 \in \Lambda$. Choose a preimage $y$ of $x$ in $M$; then $y$ corresponds to a family $(y_n : n \geq 1)$ of elements $y_n \in M/I^n M$. As for $N$, each $M/I^n M$ admits a canonical $\Lambda$-grading; then denoting by $y'_n$ the component of $y_n$ in degree $0$, we obtain an element $(y'_n : n \geq 1)$ in $\varprojlim_n (M/I^n M)^H = M^H$ which is a preimage of $x$.
%
\end{proof}



We record the following consequence for later use. Here we consider an $H$-equiva\-riant coherent sheaf of algebras $\sC$ on $\Spec(R_{\widehat{I}})$, and denote its global sections by $C$.

\begin{lem}
\label{lem:proj-resolution-affine}
Assume that $H=\mathrm{Diag}_k(\Lambda)$ as above. 
For any object $M \in \mod^H(C)$, there exists a projective object $P \in \mod^H(C)$ which is free of finite rank as a $C$-module, and a surjection $P \twoheadrightarrow M$ in $\mod^H(C)$.
\end{lem}

\begin{proof}
Let $M\in \mod^H(C)$. 
Since $M$ is finitely generated over $R_{\widehat{I}}$, as above we have $M=\varprojlim_{n}M/I^nM$. 
We can choose a finite family $(m_a)_{a \in A}$ of generators of the $R/I$-module $M/IM$, such that $m_a$ is a weight vector of the $H$-module $M/I$, say of weight $\lambda_a$. As in the proof of Lemma~\ref{lem:exactness-fixedpts-diag}
we can choose for any $a \in A$ a lift $\widetilde{m}_a \in M$ of $m_a$ whose image in $M/I^nM$ belongs to the $\lambda_a$-weight space of $M/I^nM$, for any $n \geq 1$. Denoting, for any $\lambda \in \Lambda$, by $k_H(\lambda)$ the free rank-$1$ $k$-module with the action of $H$ determined by $\lambda$,
these elements determine a morphism of $H$-equivariant $R_{\widehat{I}}$-modules $\bigoplus_a k_H(\lambda_a) \otimes R_{\widehat{I}} \rightarrow M$.
Using Item~(1) in~\cite[\href{https://stacks.math.columbia.edu/tag/0315}{Tag 0315}]{stacks-project} (for the ring $R_{\widehat{I}}$) we see that this morphism is surjective.
We deduce a surjection of $H$-equivariant $C$-modules $\bigoplus_a k_H(\lambda_a) \otimes C \twoheadrightarrow M$. Clearly the left-hand side is free over $C$. To conclude the proof, it suffices to show that $k_H(\lambda) \otimes C$ is projective in $\mod^H(C)$ for any $\lambda \in \Lambda$.
This follows from the obvious isomorphism
\[
\Hom_{\mod^H(C)}(k_H(\lambda) \otimes C,-) \cong ( k_H(-\lambda) \otimes - )^H
\]
and Lemma \ref{lem:exactness-fixedpts-diag}. 
\end{proof}

\subsection{Morphism spaces} 

We come back to the setting of~\S\ref{ss:equiv-coh-sheaves}.
In particular we assume that $X$ is of finite type over $R$.

\begin{lem}
\label{lem:Hom-equiv}
For $\sF\in \Coh^H(X_{\widehat{I}})$ and $\sG\in \QCoh^H(X_{\widehat{I}})$, the sheaf 
\[
\sHom_{\sO_{X_{\widehat{I}}}}(\sF,\sG)
\]
admits a canonical structure of object in $\QCoh^H(X_{\widehat{I}})$, 
and the $R_{\widehat{I}}$-module 
\[
\Hom_{\QCoh(X_{\widehat{I}})}(\sF,\sG)
\]
admits a canonical structure of object in $\Mod^H(R_{\widehat{I}})$.
Moreover we have a canonical isomorphism 
\[
\Hom_{\QCoh(X_{\widehat{I}})}(\sF,\sG)^H=\Hom_{\QCoh^H(X_{\widehat{I}})}(\sF,\sG).
\]
\end{lem}

\begin{proof}
Since $\sF$ is coherent, the sheaf $\sHom_{\sO_{X_{\widehat{I}}}}(\sF,\sG)$ is quasi-coherent. Now by~\cite[\href{https://stacks.math.columbia.edu/tag/01XZ}{Tag 01XZ} \& \href{https://stacks.math.columbia.edu/tag/0C6I}{Tag 0C6I}]{stacks-project}, since the morphisms $\pr$ and $\act$ are flat (see Lemma~\ref{lem:flatness}) we have canonical isomorphisms
\begin{align*}
\act^*\sHom_{\sO_{X_{\widehat{I}}}}(\sF,\sG) &\cong \sHom_{\sO_{(H\times X)_{\widehat{\sO(H)\otimes I}}}}(\act^* \sF, \act^* \sG), \\
\pr^*\sHom_{\sO_{X_{\widehat{I}}}}(\sF,\sG) &\cong \sHom_{\sO_{(H\times X)_{\widehat{\sO(H)\otimes I}}}}(\pr^* \sF, \pr^* \sG).
\end{align*}
Hence from the given isomorphisms $\act^* \sF \simto \pr^* \sF$ and $\act^* \sG \simto \pr^* \sG$ we deduce an isomorphism
\[
\act^*\sHom_{\sO_{X_{\widehat{I}}}}(\sF,\sG) \simto \pr^*\sHom_{\sO_{X_{\widehat{I}}}}(\sF,\sG),
\]
which is easily seen to endow $\sHom_{\sO_{X_{\widehat{I}}}}(\sF,\sG)$ with the structure of an $H$-equivariant quasi-coherent sheaf. Since $\Hom_{\QCoh(X_{\widehat{I}})}(\sF,\sG)$ is the pushforward of $\sHom_{\sO_{X_{\widehat{I}}}}(\sF,\sG)$ under the canonical morphism $X_{\widehat{I}} \to \Spec(R_{\widehat{I}})$, we deduce the second claim using the considerations of~\S\ref{ss:functorialities}.
The final claim is clear from the definitions.
%
\end{proof}

Using this lemma, we therefore have a bifunctor
\[
\sHom_{\sO_{X_{\widehat{I}}}}(-,-) :  \Coh^H(X_{\widehat{I}}) \times \QCoh^H(X_{\widehat{I}}) \to \QCoh^H(X_{\widehat{I}}).
\]

\begin{rmk}
Since $X_{\widehat{I}}$ is a noetherian scheme, coherent sheaves are compact objects in $\QCoh(X_{\widehat{I}})$. Combining this observation with Lemma~\ref{lem:Hom-equiv} and Remark~\ref{rmk:inv-colimits}, we see that objects of $\Coh^H(X_{\widehat{I}})$ are compact in $\QCoh^H(X_{\widehat{I}})$. Since any object of $\IndCoh^H(X_{\widehat{I}})$ is by definition a filtered colimit of objects in $\Coh^H(X_{\widehat{I}})$, it follows as in~\cite[Corollary~6.3.5]{ks} that $\IndCoh^H(X_{\widehat{I}})$ identifies with the category of ind-objects in $\Coh^H(X_{\widehat{I}})$.
\end{rmk}


\subsection{Flat resolutions and derived pullback} 
\label{ss:derived-pullback}


We continue with the setting introduced in~\S\ref{ss:equiv-coh-sheaves}, and
let $\sC$ be an $H$-equivariant coherent sheaf of $\sO_{X_{\widehat{I}}}$-algebras. 

\begin{lem}
\label{lem:flat-res}
Assume that $H=\mathrm{Diag}_k(\Lambda)$ for some finitely generated abelian group $\Lambda$, and that $X$ is projective over $R$ and admits an $H$-equivariant ample line bundle.

\begin{enumerate}
\item
\label{it:flat-res-1}
For any $\sF\in \mod^H(\sC)$ there exists a surjection $\sG\twoheadrightarrow \sF$ in $\Mod^H(\sC)$ where $\sG$ is a locally free $\sC$-module of finite rank. 
\item
\label{it:flat-res-2}
For any $\sF\in \IndMod^H(\sC)$ there exists a surjection $\sG\twoheadrightarrow \sF$ in $\Mod^H(\sC)$ where $\sG$ is a locally free $\sC$-module (possibly not of finite rank)
\end{enumerate}
\end{lem}

\begin{proof}
If $\sL$ is an $H$-equivariant ample line bundle on $X$, since the projection morphism $X_{\widehat{I}} \to X$ is affine, the pullback $\widehat{\sL}$ of $\sL$ to $X_{\widehat{I}}$ is also ample by~\cite[\href{https://stacks.math.columbia.edu/tag/0892}{Tag 0892}]{stacks-project}, and admits a canonical $H$-equivariant structure (see Example~\ref{ex:equiv-qcoh-alg}). Fix $\sF\in \mod^H(\sC)$.
It follows from our assumption that the canonical morphism $p : X_{\widehat{I}} \to \Spec(R_{\widehat{I}})$ is projective, so that for any $m \in \Z$ the pushforward $p_*(\sF \otimes_{\sO_{X_{\widehat{I}}}} \widehat{\sL}^{\otimes m})$ is a coherent sheaf,
see e.g.~\cite[\href{https://stacks.math.columbia.edu/tag/02O5}{Tag 02O5}]{stacks-project}.
For $m>0$ large enough the canonical map
\[
\widehat{\sL}^{\otimes (-m)} \otimes_{\sO_{X_{\widehat{I}}}} p^* p_*( \sF \otimes_{\sO_{X_{\widehat{I}}}} \widehat{\sL}^{\otimes m}) \to \sF
\]
is surjective,
see~\cite[\href{https://stacks.math.columbia.edu/tag/01Q3}{Tag 01Q3}]{stacks-project}.
Here the left-hand side has a canonical $H$-equivariant structure by the constructions of~\S\ref{ss:functorialities}, 
and our morphism is a morphism of $H$-equivariant quasi-coherent sheaves.
Fix such an $m$.
By Lemma~\ref{lem:proj-resolution-affine} there exists an object $\sM$ in $\QCoh^H(\Spec(R_{\widehat{I}}))$ which is free of finite rank as $\sO_{\Spec(R_{\widehat{I}})}$-module and a surjective morphism $\sM \twoheadrightarrow p_*( \sF \otimes_{\sO_{X_{\widehat{I}}}} \widehat{\sL}^{\otimes m})$. We deduce a surjection of $H$-equivariant $\sC$-modules from $\sG := \sC\otimes_{\sO_{X_{\widehat{I}}}} \widehat{\sL}^{\otimes (-m)} \otimes_{\sO_{X_{\widehat{I}}}} p^*(\sM)$ to $\sF$. Here $\sG$ is locally free of finite rank over $\sC$, hence we have proved~\eqref{it:flat-res-1}.

If $\sF$ is now a general object of $\IndMod^H(\sC)$, one can write $\sF = \varinjlim_{a} \sF_a$ where each $\sF_a$ is coherent. By~\eqref{it:flat-res-1} and its proof, there exist objects $(\sG_a)_a$ and surjections $(\sG_a \twoheadrightarrow \sF_a)_a$ where each $\sG_a$ is locally free of finite rank over $\sC$, and moreover there is an open cover of $X_{\widehat{I}}$ which trivializes all of these modules. Then we obtain a surjection $\bigoplus_a \sG_a \twoheadrightarrow \sF$ where the left-hand side is locally free.
\end{proof}

Using~\cite[\href{https://stacks.math.columbia.edu/tag/05T7}{Tag 05T7}]{stacks-project}, one deduces from Lemma~\ref{lem:flat-res} that (under the assumptions of the lemma) any bounded above complex of objects of $\mod^H(\sC)$, resp.~$\IndMod^H(\sC)$, is quasi-isomorphic to a bounded above complex of objects which are locally free of finite rank, resp.~locally free.
As a corollary (for $\sC=\sO_{X_{\widehat{I}}}$), we deduce the following claim.

\begin{corollary}
Assume that $H=\mathrm{Diag}_k(\Lambda)$ for some finitely generated abelian group $\Lambda$, and that $X$ is projective over $R$ and admits an $H$-equivariant ample line bundle.
The bifunctors $-\otimes_{\sO_{X_{\widehat{I}}}}-$ and $\sHom_{\sO_{X_{\widehat{I}}}}(-,-)$ admit derived bifunctors 
\[
-\otimes^L_{\sO_{X_{\widehat{I}}}}-: D^-\IndCoh^H(X_{\widehat{I}})\times D^-\QCoh^H(X_{\widehat{I}})\rightarrow D^-\QCoh^H(X_{\widehat{I}}),
\]
\[
R\sHom_{\sO_{X_{\widehat{I}}}}(-,-): D^-\Coh^H(X_{\widehat{I}})\times D^+\QCoh^H(X_{\widehat{I}})\rightarrow D^+\QCoh^H(X_{\widehat{I}}).
\]
Moreover these bifunctors are compatible in the natural way with their non-equivariant counterparts.
\end{corollary}

Consider now the setting of~\S\ref{ss:functorialities}; in particular we have two $R$-schemes $X$ and $Y$ endowed with actions of $H$, and an $H$-equivariant morphism of $R$-schemes $f : Y \to X$. We also let $\sC_X$, resp. $\sC_Y$, be an $H$-equivariant coherent sheaf of $\sO_{X_{\widehat{I}}}$-algebras, resp. $\sO_{Y_{\widehat{I}}}$-algebras, and
suppose we are given an $H$-equivariant morphism of quasicoherent sheaves of algebras 
$(f_{\widehat{I}})^{*}\sC_X \rightarrow \sC_Y$. 

\begin{corollary}
\label{cor:derived-pullback}
Assume that $H=\mathrm{Diag}_k(\Lambda)$ for some finitely generated abelian group $\Lambda$, and that $X$ is projective over $R$ and admits an $H$-equivariant ample line bundle.
Then the functor
\[
(f_{\widehat{I}})^* : \mod^H(\sC_X) \to \mod^H(\sC_Y), \ \text{resp.} \ (f_{\widehat{I}})^* : \IndMod^H(\sC_X) \to \IndMod^H(\sC_Y),
\]
admits a derived functor 
\begin{multline*}
L(f_{\widehat{I}})^*: D^-\mod^H(\sC_X) \rightarrow D^-\mod^H(\sC_Y), \ \text{resp.} \\
L(f_{\widehat{I}})^* : D^-\IndMod^H(\sC_X) \to D^-\IndMod^H(\sC_Y).
\end{multline*}
Moreover, these functors are compatible with each other in the obvious sense and the following diagram commutes, where the lower arrow is the usual (derived) pullback functor for quasi-coherent modules:
\[
\begin{tikzcd}[column sep=large]
D^-\IndMod^H(\sC_X) \ar[r, "L(f_{\widehat{I}})^*"] \ar[d, "\for"'] & D^-\IndMod^H(\sC_Y) \ar[d, "\for"] \\
D^-\Mod(\sC_X) \ar[r, "L(f_{\widehat{I}})^*"] & D^-\Mod(\sC_Y).
\end{tikzcd}
\]
\end{corollary}

In case $\sC_X$ is flat over $\sO_{X_{\widehat{I}}}$ and the morphism $(f_{\widehat{I}})^{*}\sC_X \rightarrow \sC_Y$ is an isomorphism, we have a commutative diagram
\[
\begin{tikzcd}[column sep=large]
D^-\IndMod^H(\sC_X) \ar[r, "L(f_{\widehat{I}})^*"] \ar[d, "\for"'] & D^-\IndMod^H(\sC_Y) \ar[d, "\for"] \\
D^-\IndCoh^H(X_{\widehat{I}}) \ar[r, "L(f_{\widehat{I}})^*"] & D^-\IndCoh^H(Y_{\widehat{I}}).
\end{tikzcd}
\]

\subsection{Derived pushforward} 
\label{ss:derived-pushforward}

We now consider two separated $R$-schemes of finite type $X$, $Y$ endowed with actions of $H$ compatible with the projection to $\Spec(R)$, and an $H$-equivariant morphism of $R$-schemes $f : Y \to X$. 
Let $\sC_X$, resp. $\sC_Y$, be an $H$-equivariant coherent sheaf of $\sO_{X_{\widehat{I}}}$-algebras, resp. $\sO_{Y_{\widehat{I}}}$-algebras, and
suppose we are given an $H$-equivariant morphism of quasicoherent sheaves of algebras 
$(f_{\widehat{I}})^{*}\sC_X \rightarrow \sC_Y$. 

In case $f$ is affine, the functor $(f_{\widehat{I}})_* : \Mod^H(\sC_Y) \to \Mod^H(\sC_X)$ is exact, hence induces a t-exact functor between the associated derived categories. If $f$ is a closed immersion, then $(f_{\widehat{I}})_*$ sends coherent sheaves to coherent sheaves, so that we have similar functors for derived categories of coherent equivariant sheaves, and colimits of such, which are all compatible in the obvious way. This covers in particular the situation considered in Example~\ref{ex:infinitesimal-neighborhoods}.

For a general morphism $f$ it is more complicated to construct a derived functor $R(f_{\widehat{I}})_*$, mainly because we do not know if the category $\QCoh^H(Y_{\widehat{I}})$ (or, more generally, $\Mod^H(\sC_Y)$) as enough injective objects. Below we will consider this question, but will obtain an answer only under rather strong assumptions on the actions and morphism.



\begin{lem}
\label{lem:acyclic-resolution}
Assume that the following condition is satisfied. There is an affine open cover $\{V_a\}_{a \in A}$ of $X$ by $H$-stable subschemes, such that for each $a \in A$ there is an affine open cover $\{U_b\}_{b\in B_a}$ of $f^{-1}(V_a)$ by $H$-stable subschemes. 

For any $\sF\in \Mod^H(\sC_Y)$, there exists an injective morphism $\sF\hookrightarrow \sG$ in $\Mod^H(\sC_Y)$ where the image of $\sG$ in $\QCoh(Y_{\widehat{I}})$ is $(f_{\widehat{I}})_*$-acyclic. Moreover, in case $\sF$ belongs to $\IndMod^H(\sC_Y)$ one can take $\sG$ in $\IndMod^H(\sC_Y)$.
\end{lem}

\begin{proof}
We fix affine open covers as in the statement, which we assume are finite (which is possible since our schemes are noetherian) and set $B = \sqcup_a B_a$. 
For any $b \in B$ we denote by $\imath_b :U_b \hookrightarrow Y$ the open embedding. 
Then for any $\sF\in \Mod^H(\sC_Y)$ we have an injective morphism $\sF\hookrightarrow \bigoplus_{b} (\imath_{b,\widehat{I}})_*(\imath_{b,\widehat{I}})^*\sF$, each $(\imath_{b,\widehat{I}})_*(\imath_{b,\widehat{I}})^*\sF$ has a natural structure of object in $\Mod^H(\sC_Y)$ by the constructions of~\S\ref{ss:functorialities}, and their images in $\QCoh(Y_{\widehat{I}})$ are $(f_{\widehat{I}})_*$-acyclic, since $\imath_{b}$ and $f \circ \imath_b$ are affine morphisms. 

To conclude the proof, it now suffices to prove that if $\sF\in \IndMod^H(\sC_Y)$, then each $(\imath_{b,\widehat{I}})_*(\imath_{b,\widehat{I}})^*\sF$ belongs to $\IndMod^H(\sC_Y)$. 
By the projection formula (see~\cite[\href{https://stacks.math.columbia.edu/tag/08EU}{Tag 08EU}]{stacks-project}) we have
\[
(\imath_{b,\widehat{I}})_*(\imath_{b,\widehat{I}})^*\sF \cong (\imath_{b,\widehat{I}})_* \sO_{(U_b)_{\widehat{I}}} \otimes_{\sO_{Y_{\widehat{I}}}} \sF.
\]
Now by the flat base change theorem, $(\imath_{b,\widehat{I}})_* \sO_{(U_b)_{\widehat{I}}}$ is the pullback to $Y_{\widehat{I}}$ of $(\imath_{b})_* \sO_{U_b}$, hence by Remark~\ref{rmk:IndMod}\eqref{it:IndMod-completion} this object belongs to $\IndCoh^H(X_{\widehat{I}})$.
It follows that $(\imath_{b,\widehat{I}})_*(\imath_{b,\widehat{I}})^*\sF$ belongs to $\IndMod^H(\sC)$, as desired.
\end{proof}

Using~\cite[\href{https://stacks.math.columbia.edu/tag/05T6}{Tag 05T6}]{stacks-project} one deduces from Lemma~\ref{lem:acyclic-resolution} that (under the assumptions of the lemma) any bounded below complex of objects in $\Mod^H(\sC_Y)$, resp.~in $\IndMod^H(\sC_Y)$, is quasi-isomorphic to a bounded below complex of objects in $\Mod^H(\sC_Y)$, resp.~in $\IndMod^H(\sC_Y)$, whose images in $\QCoh(Y_{\widehat{I}})$ are $(f_{\widehat{I}})_*$-acyclic. 


\begin{corollary}
\label{cor:derived-pushforward}
Assume that the condition in Lemma~\ref{lem:acyclic-resolution} is satisfied.

\begin{enumerate}
\item
\label{it:derived-pushforward-1}
The functor $(f_{\widehat{I}})_* : \Mod^H(\sC_Y)\rightarrow \Mod^H(\sC_X)$ admits a derived functor 
\[
R(f_{\widehat{I}})_*: D^+\Mod^H(\sC_Y)\rightarrow D^+\Mod^H(\sC_X),
\]
and the diagram
\[
\begin{tikzcd}[column sep=large] 
 D^+\Mod^H(\sC_Y) \ar[r, "R(f_{\widehat{I}})_*"] \ar[d] & D^+\Mod^H(\sC_X) \ar[d] \\
  D^+\QCoh(Y_{\widehat{I}}) \ar[r, "R(f_{\widehat{I}})_*"] & D^+\QCoh(Y_{\widehat{I}})
\end{tikzcd}
\]
commutes, where the lower line is the usual (derived) pushforward for quasi-coherent sheaves, and the vertical arrows are the obvious forgetful functors.
\item
\label{it:derived-pushforward-2}
Assume furthermore that $f$ is proper. The functor $(f_{\widehat{I}})_* : \IndMod^H(\sC_Y)\rightarrow \IndMod^H(\sC_X)$ admits a derived functor 
\[
R(f_{\widehat{I}})_*: D^+\IndMod^H(\sC_Y)\rightarrow D^+\IndMod^H(\sC_X),
\]
which restricts to a functor
\[
R(f_{\widehat{I}})_*: \Db\mod^H(\sC_Y) \rightarrow \Db\mod^H(\sC_X).
\]
These functors are compatible with the functor in~\eqref{it:derived-pushforward-1} in the obvious sense.
\end{enumerate}
\end{corollary}

\begin{proof}
Item~\eqref{it:derived-pushforward-1} is a direct consequence of Lemma~\ref{lem:acyclic-resolution} and the comments above. In~\eqref{it:derived-pushforward-2}, the existence of the first functor follows similarly from the comments above, and its compatibility with the functor in~\eqref{it:derived-pushforward-1} is clear. For the categories of coherent modules, we note first that, as explained in~\S\ref{ss:equiv-coh-sheaves}, the category $\Db\mod^H(\sC_Y)$ identifies with a full subcategory of $\Db\IndMod^H(\sC_Y)$, hence of $D^-\IndMod^H(\sC_Y)$, so that the claim makes sense. The the stated property follows from the commutativity of the diagram in~\eqref{it:derived-pushforward-1} and the similar property of usual pushforward of coherent sheaves, see~\cite[\href{https://stacks.math.columbia.edu/tag/08E2}{Tag 08E2}]{stacks-project}.
\end{proof}

Using~\cite[\href{https://stacks.math.columbia.edu/tag/08D5}{Tag 08D5}]{stacks-project}, we deduce from the commutativity of the diagram in Corollary~\ref{cor:derived-pushforward}\eqref{it:derived-pushforward-1} that the functor $R(f_{\widehat{I}})_*$ in this statement restricts to a functor from $\Db\Mod^H(\sC_Y)$ to $\Db\Mod^H(\sC_X)$. Similarly, the functor in Corollary~\ref{cor:derived-pushforward}\eqref{it:derived-pushforward-1} restricts to a functor from $\Db\IndMod^H(\sC_Y)$ to $\Db\IndMod^H(\sC_X)$.

\subsection{Adjointness}
\label{ss:adjointness}

In this subsection we assume that both the constructions of~\S\ref{ss:derived-pullback} and those of~\S\ref{ss:derived-pushforward} are available. So, we consider $R$-schemes of finite type $X$, $Y$ endowed with compatible actions of $H$, an $H$-equivariant morphism of $R$-schemes $f : Y \to X$, and we assume that:
\begin{itemize}
\item
$X$, $Y$ are separated;
\item
$f$ is proper;
\item
$H=\mathrm{Diag}_k(\Lambda)$ for some finitely generated abelian group $\Lambda$;
\item
$X$ is projective over $R$ and admits an $H$-equivariant ample line bundle;
\item
there is an affine open cover $\{V_a\}_{a \in A}$ of $X$ by $H$-stable subschemes, such that for each $a \in A$ there is an affine open cover $\{U_b\}_{b\in B_a}$ of $f^{-1}(V_a)$ by $H$-stable subschemes.
\end{itemize}
We also let $\sC_X$, resp. $\sC_Y$, be an $H$-equivariant coherent sheaf of $\sO_{X_{\widehat{I}}}$-algebras, resp. $\sO_{Y_{\widehat{I}}}$-algebras, and
suppose we are given an $H$-equivariant morphism of quasicoherent sheaves of algebras 
$(f_{\widehat{I}})^{*}\sC_X \rightarrow \sC_Y$. 

Under these assumptions, we have functors
\begin{gather*}
L(f_{\widehat{I}})^* : D^-\IndMod^H(\sC_X) \to D^-\IndMod^H(\sC_Y), \\
R(f_{\widehat{I}})_*: D^+\IndMod^H(\sC_Y) \rightarrow D^+\IndMod^H(\sC_X),
\end{gather*}
which restrict to functors between derived categories of coherent modules as specified above. By general results on derived functors (see in particular~\cite[\href{https://stacks.math.columbia.edu/tag/05T5}{Tag 05T5}]{stacks-project} and~\cite[\href{https://stacks.math.columbia.edu/tag/0DVC}{Tag 0DVC}]{stacks-project}), for $\sF$ in $D^-\IndMod^H(\sC_X)$ and $\sG$ in $D^+\IndMod^H(\sC_Y)$ we have a canonical bifunctorial isomorphism
\[
\Hom_{D\IndMod^H(\sC_Y)}(L(f_{\widehat{I}})^* \sF, \sG) \cong \Hom_{D\IndMod^H(\sC_X)}(\sF, R(f_{\widehat{I}})_* \sG).
\]
In particular, if we assume that $L(f_{\widehat{I}})^*$ sends the subcategory $\Db\IndMod^H(\sC_X)$ into $\Db\IndMod^H(\sC_Y)$, then the restricted functors
\begin{gather*}
L(f_{\widehat{I}})^* : \Db\IndMod^H(\sC_X) \to \Db\IndMod^H(\sC_Y), \\
R(f_{\widehat{I}})_*: \Db\IndMod^H(\sC_Y) \rightarrow \Db\IndMod^H(\sC_X)
\end{gather*}
form an adjoint pair.

\begin{rmk}
\label{rmk:adjunction-inf-neighborhoods}
Assume that $H=\mathrm{Diag}_k(\Lambda)$ for some finitely generated abelian group $\Lambda$, and that $X$ is projective over $R$ and admits an $H$-equivariant ample line bundle. Consider also an $H$-equivariant coherent sheaf of algebras $\sC$ on $X$. Fix some $n \geq 1$, and
consider the setting of Example~\ref{ex:infinitesimal-neighborhoods}.
Then, as explained in~\S\ref{ss:derived-pushforward}, we have a derived pushforward functor $D \Mod^H(\sC_n) \to D\Mod^H(\sC)$, without further assumptions, and compatible functors $D \IndMod^H(\sC_n) \to D\IndMod^H(\sC)$ and $D \mod^H(\sC_n) \to D\mod^H(\sC)$ which restrict to functors
\begin{equation}
\label{eqn:derived-push-inf-neighborhood}
D^- \IndMod^H(\sC_n) \to D^- \IndMod^H(\sC), \quad D^- \mod^H(\sC_n) \to D^- \mod^H(\sC).
\end{equation}

On the other hand, the constructions of~\S\ref{ss:derived-pullback} provide derived pullback functors
\begin{equation}
\label{eqn:derived-pull-inf-neighborhood}
D^- \IndMod^H(\sC) \to D^-\IndMod^H(\sC_n), \quad D^- \mod^H(\sC) \to D^-\mod^H(\sC_n).
\end{equation}

The same considerations as above show that each functor in~\eqref{eqn:derived-push-inf-neighborhood} is right adjoint to the corresponding functor in~\eqref{eqn:derived-pull-inf-neighborhood}.
\end{rmk}

\subsection{Completion and complexes supported on an infinitesimal neighborhood}

In this subsection we consider the setting of Remark~\ref{rmk:adjunction-inf-neighborhoods}, in the special case when $\sC=\sA_{\widehat{I}}$ for some $H$-equivariant coherent sheaf of algebras $\sA$ on $X$. (See Example~\ref{ex:equiv-qcoh-alg} for the notation.) In particular we have the exact functor~\eqref{eqn:pullback-completion}, which induces a functor
$D^- \Mod^H(\sA) \to \Mod^H(\sA_{\widehat{I}})$, which will be denoted $\sF \mapsto \sF_{\widehat{I}}$. We fix some $n \geq 1$, and denote by $\sA_n$ the restriction of $\sA$ to $X \times_{\Spec(R)} \Spec(R/I^n)$.

Let $\sF, \sG$ be objects of $D^- \mod^H(\sA)$, and assume that $\sG$ belongs to the essential image of the natural pushforward functor $D^- \mod^H(\sA_n) \to D^- \mod^H(\sA)$.

\begin{lem}
\label{lem:completion-Ext-inf-neighborhoods}
The morphism
\[
\Hom_{D^- \mod^H(\sA)}(\sF,\sG) \to \Hom_{D^- \mod^H(\sA_{\widehat{I}})}(\sF_{\widehat{I}},\sG_{\widehat{I}})
\]
induced by the functor $(-)_{\widehat{I}}$ is an isomorphism.
\end{lem}

\begin{proof}
Let $i : X \times_{\Spec(R)} \Spec(R/I^n) \to X$ be the natural closed immersion. Then by assumption there exists a complex $\sG'$ in $D^- \mod^H(\sA_n)$ such that $\sG=i_* \sG'$. Recall from Example~\ref{ex:equiv-qcoh-alg} that the analogue of the functor $(-)_{\widehat{I}}$ for the scheme $X \times_{\Spec(R)} \Spec(R/I^n)$ is an equivalence. Identifying the corresponding categories, by the flat base change theorem we have $\sG_{\widehat{I}} = (i_{\widehat{I}})_* \sG'$. Using adjunction (see Remark~\ref{rmk:adjunction-inf-neighborhoods} for the ``completed'' case) we have isomorphisms
\begin{align*}
\Hom_{D^- \mod^H(\sA)}(\sF,\sG) &\cong \Hom_{D^- \mod^H(\sA_n)}(Li^* \sF,\sG'), \\
\Hom_{D^- \mod^H(\sA_{\widehat{I}})}(\sF_{\widehat{I}},\sG_{\widehat{I}}) &\cong \Hom_{D^- \mod^H(\sA_n)}(L(i_{\widehat{I}})^*(\sF_{\widehat{I}}),\sG').
\end{align*}
By compatibility of pullback functors with composition, we have $L(i_{\widehat{I}})^*(\sF_{\widehat{I}}) \cong Li^* \sF$, and under these identifications our morphism is the identity; it is therefore an isomorphism.
\end{proof}

\bibliographystyle{plain} 
\bibliography{MyBibtex}

\end{document}